\documentclass[11 pt]{amsart}
\usepackage{fullpage, amsthm, amsmath, amsfonts, amssymb, url, mathrsfs, stmaryrd}
\usepackage[colorlinks]{hyperref}

\author{Casey Rodriguez}
\address{Department of Mathematics\\
  University of Chicago\\
  5734 S. University Avenue \\
  Chicago,IL, 60637}
\email{c-rod216@math.uchicago.edu}

\numberwithin{equation}{section}

\newcommand{\R}{\mathbb R}

\newcommand{\N}{\mathbb N}
\newcommand{\Z}{\mathbb Z}

\newcommand{\ra}{\rangle}
\newcommand{\la}{\langle}

\newcommand{\cl}{\mathcal}
\newcommand{\energysp}{\dot{H}^1 \times L^2}

\newcommand{\supp}{\mbox{supp }}

\newcommand{\grad}{\nabla}

\newtheorem{lem}{Lemma}[section]
\newtheorem{thm}[lem]{Theorem}
\newtheorem{ppn}[lem]{Proposition}
\newtheorem{defn}[lem]{Definition}
\newtheorem{clm}[lem]{Claim}
\newtheorem{cor}[lem]{Corollary}

\title{Profiles for the Radial Focusing Energy--Critical \\ Wave Equation in Odd Dimensions }
\date{\today}

\begin{document}

\begin{abstract}
In this paper we consider global and non-global radial solutions of the focusing energy--critical wave equation on $\R \times \R^N$ where $N \geq 5$ is odd.  We prove that
if the solution remains bounded in the energy space as you approach the maximal forward time of existence, then along a sequence of times converging to the maximal forward time of existence, the solution decouples into a sum of dynamically rescaled
solitons, a free radiation term, and an error tending to zero in the energy space.  If, in addition, we assume a 
bound on the evolution that rules out formation of multiple solitons, then this decoupling holds for all times approaching
the maximal forward time of existence. 
\end{abstract}

\maketitle

\section{Introduction}
Consider the focusing energy--critical wave equation
\begin{align}
       \partial_t^2 u - \Delta u &= |u|^{\frac{4}{N-2}}u, \quad (t,x) \in \R \times \R^N \label{nlw},\\
       \vec u (0) &= (u_0,u_1) \in \energysp \nonumber,
\end{align}
where $\vec u(t)$ is the vector in $\energysp$
$$
\vec u(t) := (u(t,\cdot), \partial_t u(t,\cdot)).
$$
Equation \eqref{nlw} is locally well--posed in the energy space $\dot H^1 \times L^2$ for all $N \geq 3$: for any initial data $(u_0,u_1) \in \energysp$, there exists
a unique solution $u$ defined on a maximal interval of existence $I_{\max}(u) = (T_-(u),T_+(u))$ such that $\vec u(t) \in C(I_{\max}; \energysp )$
and $u \in S(J)$ for any compact $J \subset I_{\max}(u)$ where
$$
S(J) = L^{\frac{2(N+1)}{N-2}}(J \times \R^N).
$$
Moreover, if $\| (u_0, u_1) \|_{\energysp} \ll 1$, then $I_{\max}(u) = \R$ and $u$ scatters forward and backwards in time:
there exists $v_L^{\pm}(t) = S(t)(v_0^{\pm},v_1^{\pm})$ so that
$$
\lim_{t \rightarrow \pm\infty} \| \vec u(t) - \vec v_L^{\pm}(t) \|_{\energysp} = 0.
$$

The energy of the solution is constant on $I_{\max}(u)$:
\begin{align*}
E(\vec u(t)) := \frac{1}{2} \int_{\R^N} \left ( |\nabla u(t) |^2 + |\partial u(t)|^2 \right ) dx - \frac{N-2}{2N} \int_{\R^N}
|u(t)|^{\frac{2N}{N-2}} dx  = E(\vec u(0)).
\end{align*}
The Cauchy problem \eqref{nlw} is invariant under scaling
$$
u(t,x) \mapsto u_{\lambda}(t,x) := \frac{1}{\lambda^{\frac{N-2}{2}}} u \left ( \frac{t}{\lambda}, \frac{x}{\lambda} \right ).
$$
This scaling leaves unchanged the energy $E(\vec u(t))$, as well as the $\energysp$ of the initial data.  For this reason, \eqref{nlw} is called \emph{energy--critical}.

In this work, we will study the dynamics of radial solutions to \eqref{nlw} that are bounded in the energy space forward in time, i.e.,
\begin{align}\label{bdd}
\sup_{t \in [0,T_+(u))} \| \nabla u(t) \|_{L^2}^2 + \|\partial_t u(t) \|_{L^2}^2 < \infty.
\end{align}
We call these solutions type II solutions.  In the case $T_+ = T_+(u) < \infty$, we call these solutions type II blow--up solutions.

The Cauchy problem \eqref{nlw} has an explicit stationary solution:
\begin{align*}
W(x) = \frac{1}{\left (1 + \frac{|x|^2}{N(N-2)} \right )^{\frac{N-2}{2}}}.
\end{align*}
$W$ is the the unique (up to sign, dilation, and translation) nonnegative nontrivial $C^2$ solution to the elliptic equation
\begin{align}\label{stateq}
-\Delta W = |W|^{\frac{4}{N-2}}W, \quad x \in \R^N.
\end{align}
It is the unique (up to sign and dilation) radial nontrivial $\dot H^1$ solution to \eqref{stateq}.  $W$ is also the
unique (up to sign, dilation, and translation) extremizer for the Sobolev inequality
$$
\| f \|_{L^{\frac{2N}{N-2}}} \leq C_N \| \nabla f \|_{L^2},
$$
where $C_N$ is the best constant (see \cite{tal}).  Because of these characterizations, $W$ is referred to as the ground state.

Kenig and Merle, \cite{km08} classified the possible dynamics for for solutions below the energy of the ground state.
Indeed, if $3 \leq N \leq 5$ and
$$
E(\vec u) < E(W,0),
$$
the following dichotomy holds:
\begin{itemize}
 \item If $\| \nabla u_0 \|_{L^2} < \| \nabla W \|_{L^2}$, then $I_{\max}(u) = \R$ and $u$ scatters both forward and backwards in time.
\item If $\| \nabla u_0 \|_{L^2} > \| \nabla W \|_{L^2}$, then $I_{\max}(u)$ is finite.
\end{itemize}
The case $\| \nabla u_0 \|_{L^2} = \| \nabla W \|_{L^2}$ is impossible if $E(\vec u) < E(W,0)$.  Solutions with threshold
energy, $E(\vec u) = E(W,0)$ were classified by Duyckaerts and Merle \cite{dm08} for $3 \leq N \leq 5$ and by
Li and Zhang \cite{lz} for $N \geq 6$.

For type II solutions, is is known that $\| \nabla W\|^2_{L^2}$ is a sharp threshold for finite time blow--up and scattering.
Indeed, we have the following generalization (see \cite{dkm2}) of the scattering part from Kenig--Merle result: If $3 \leq N \leq 5$, $u(t)$ is a type II solution, and
\begin{align*}
&\sup_{t \in [0, T_+(u))} \| \nabla u(t) \|_{L^2}^2 + \frac{N-2}{2} \| \partial_t u(t) \|_{L^2}^2 < \| \nabla W \|_{L^2}^2
\mbox{ ($u$ non--radial),} \\
&\sup_{t \in [0, T_+(u))} \| \nabla u(t) \|_{L^2}^2 < \| \nabla W \|_{L^2}^2
\mbox{ ($u$ radial),}
\end{align*}
then $T_+(u) = +\infty$ and $u$ scatters forward in time.

In dimension $N = 3$, Krieger, Schlag, and Tataru \cite{krst} constructed, for every $\delta > 0$, a radial type II blow--up solution
$u$ so that $T_+(u) = 1$ and
$$
\sup_{0 \leq t < 1} \| \nabla u(t) \|_{L^2}^2 \leq \| \nabla W \|_{L^2}^2 + \delta.
$$
The blow--up occurs via a bubbling off of the ground state $W$:
$$
\vec u(t) = \left ( \frac{1}{\lambda(t)^{\frac{1}{2}}} W \left ( \frac{x}{\lambda(t)} \right ), 0  \right ) + \vec v(t),
$$
with $\lambda(t) = (1-t)^{1 + \nu}$ for $\nu > 1/2$.  The error $v(t)$ is regular up to $t = 1$.  This result was later extended to cover the case $0 < \nu \leq 1/2$ by
Krieger and Schlag \cite{krs}.  

In dimension $N = 4$, Hillairet and Raphael \cite{hillr} constructed, for every $\delta > 0$, a radial $C^\infty$ type II blow--up solution $u$ so that $T_+(u) = 1$,
$$
\sup_{0 \leq t < 1} \| \nabla u(t) \|_{L^2}^2 \leq \| \nabla W \|_{L^2}^2 + \delta,
$$
and again, the blow--up occurs via a bubbling off of the ground state $W$:
$$
\vec u(t) = \left ( \frac{1}{\lambda(t)} W \left ( \frac{x}{\lambda(t)} \right ), 0  \right ) + \vec v(t),
$$
with $\lambda(t) = (1 - t) \exp ( - \sqrt{\log|1 - t|}(1 + o(1))$ as $t \rightarrow 1^-$.  The construction of
such type II blow--up solutions in higher dimensions is still open.

This type of bubbling behavior is believed to be characteristic of all radial type II solutions that either break down in
finite time or do not scatter forward in time.  More precisely, it is believed that if $u$ is a radial type II solution of
\eqref{nlw}, and either $T_+(u) < +\infty$
or $T_+(u) = +\infty$ and $u$ does not scatter forward in time, then $u(t)$ asymptotically decouples into a sum of
dynamically rescalings of $W$, a radiation term, and a term which goes to zero in $\energysp$ as $t \rightarrow T_+(u)$.  This
soliton resolution type conjecture was verified in the seminal work by Duyckaerts, Kenig, and Merle \cite{dkm4} for $N = 3$, and
along a sequence of times approaching $T_+(u)$ by C\^ote, Kenig, Lawrie, and Schlag \cite{ckls} for $N = 4$.  For results of this
type for other equations, see for example \cite{ckls1} \cite{ckls2} \cite{kls1} \cite{klls2}.

In this paper, we treat the case $N \geq 5$ and odd.  The following results are the higher dimensional analogs of the main
results for $N = 3$ \cite{dkm3} and $N = 4$ \cite{ckls}.  We distinguish two cases: blow--up and global defined solutions.

\begin{thm}[Type II blow--up solutions]\label{thm1} Let $N \geq 5$ be odd.  Let $u$ be a radial solution to \eqref{nlw} with $T_+(u) < +\infty$ that satisfies \eqref{bdd}.
Then there exists $(v_0,v_1) \in \energysp$, a sequence $t_n \rightarrow T_+(u)$, an integer $J_0 \geq 1$, $J_0$ sequences of positive numbers 
$\{\lambda_{j,n}\}_n$,
 $j = 1, \ldots, J_0$, $J_0$ signs $\iota_j \in \{ \pm 1 \}$ such that
 \begin{align*}
\lambda_{1,n} \ll \lambda_{2,n} \ll \cdots \ll \lambda_{J_0,n} \ll T_+(u) - t_n,
\end{align*}
and
\begin{align*}
\vec u(t_n) = (v_0,v_1) + \sum_{j = 1}^{J_0} \left (\frac{\iota_j}{\lambda_{j,n}^{\frac{N-2}{2}}} W \left (\frac{x}{\lambda_{j,n}} \right ),0 \right ) + o(1)
\quad \mbox{in } \energysp \mbox{ as } n \rightarrow \infty.
\end{align*}

Furthermore
\begin{align*}
\lim_{t \rightarrow T_+(u)} \int_{ |x| \leq T_+(u) - t }  \frac{1}{2} |\nabla u(t) |^2 + \frac{1}{2} |\partial u(t)|^2  - \frac{N-2}{2N}
|u(t)|^{\frac{2N}{N-2}} \hspace{2pt} dx = J_0 E(W,0).
\end{align*}
\end{thm}

If we make an additional assumption on the size of $\nabla u(t)$ inside the backwards light cone, then we obtain the decomposition
above for any sequence of times and can prove the following classification of type II blow--up solutions.

\begin{thm}[Type-II blow-up solutions below $2 \|\nabla W\|^2_{L^2}$]\label{below1}
Let $u$ be a radial solution to \eqref{nlw} with $T_+(u) < +\infty$ that satisfies \eqref{bdd}.  Assume in addition, that
\begin{align*}
\sup_{t \in [0,T_+(u))} \int_{ |x| \leq T_+(u) - t } |\nabla u(t) |^2 dx < 2 \|\nabla W\|^2_{L^2}.
\end{align*}
Then there exists $(v_0,v_1) \in \energysp$, a continuous positive function $\lambda(t)$ defined for $t < T_+(u)$ close to $T_+(u)$, and a sign $\iota \in \{\pm 1\}$ such that
\begin{align*}
\lim_{t \rightarrow T_+(u)} \frac{\lambda(t)}{T_+(u)-t} = 0,
\end{align*}
and
\begin{align*}
\vec u(t) = (v_0,v_1) +  \left (\frac{\iota}{\lambda(t)^{(N-2)/2}} W \left (\frac{x}{\lambda(t)} \right ),0 \right ) + o(1)
\quad \mbox{in } \energysp \mbox{ as } t \rightarrow T_+(u).
\end{align*}
\end{thm}

In the globally defined case, the solution is, at least for a sequence of times,
a sum of rescaled $W$ and a finite energy
solution to the linear wave equation:
\begin{align}
       \partial_t^2 v - \Delta v  &= 0, \quad (t,x) \in \R \times \R^N \label{lw},\\
       \vec v (0) &= (v_0,v_1) \in \energysp \nonumber.
\end{align}

\begin{thm}[Type II global solutions]\label{thm2} Let $N \geq 5$ be odd.  Let $u$ be a radial solution to \eqref{nlw} with $T_+(u) = +\infty$ that satisfies \eqref{bdd}.
Then there exists a solution $v_L$ to the linear wave equation \eqref{lw}, a sequence $t_n \rightarrow +\infty$, an
 integer $J_0 \geq 0$, $J_0$ sequences of positive numbers $\{\lambda_{j,n}\}_n$,
 $j = 1, \ldots, J_0$, $J_0$ signs $\iota_j \in \{ \pm 1 \}$ such that
\begin{align*}
\lambda_{1,n} \ll \lambda_{2,n} \ll \cdots \ll \lambda_{J_0,n} \ll t_n,
\end{align*}
and
\begin{align*}
\vec u(t_n) = \vec v_L(t_n) + \sum_{j = 1}^{J_0} \left (\frac{\iota_j}{\lambda_{j,n}^{\frac{N-2}{2}}} W \left (\frac{x}{\lambda_{j,n}} \right ),0 \right ) + o(1)
\quad \mbox{in } \energysp \mbox{ as } n \rightarrow \infty.
\end{align*}

Furthermore for all $A > 0$, the limit
\begin{align*}
E_A := \lim_{t \rightarrow +\infty} \int_{ |x| \leq t - A }  \frac{1}{2} |\nabla u(t) |^2 + \frac{1}{2} |\partial u(t)|^2  - \frac{N-2}{2N}
|u(t)|^{\frac{2N}{N-2}} \hspace{2pt} dx
\end{align*}
exists and
\begin{align*}
\lim_{A \rightarrow +\infty} E_A = J_0 E(W,0).
\end{align*}
\end{thm}

As in the blow--up case, if we assume a bound on $\nabla u(t)$ which prevents there being more than one soliton, then
we can prove a global--in--time decomposition along all $t \rightarrow +\infty$.

\begin{thm}[Type-II global solutions below $2 \|\nabla W\|^2_{L^2}$]\label{below2}
Let $u$ be a radial solution to \eqref{nlw} with $T_+(u) = +\infty$ that satisfies \eqref{bdd}.  Assume in addition, that there exists
an $A > 0$ so that
\begin{align*}
\limsup_{t \rightarrow +\infty} \int_{ |x| \leq t - A } |\nabla u(t) |^2 dx < 2 \|\nabla W\|^2_{L^2}.
\end{align*}
Then, there exists a solution $v_L$ to the linear wave equation \eqref{lw}, so that one of the following holds:
\begin{enumerate}
\item $u(t)$ scatters to $v_L(t)$, i.e.
\begin{align*}
\vec u(t) = \vec v_L(t) + o(1) \quad \mbox{in } \energysp \mbox{ as } t \rightarrow +\infty.
\end{align*}
\item There exists a continuous positive function $\lambda(t)$ such that
\begin{align*}
\lim_{t \rightarrow +\infty} \frac{\lambda(t)}{t} = 0,
\end{align*}
and
\begin{align*}
\vec u(t) = \vec v_L(t) +  \left (\frac{\iota}{\lambda(t)^{(N-2)/2}} W \left (\frac{x}{\lambda(t)} \right ),0 \right ) + o(1)
\quad \mbox{in } \energysp \mbox{ as } t \rightarrow +\infty.
\end{align*}
\end{enumerate}
\end{thm}

Note that, at least for one sequence of times, Theorem \ref{thm1} and Theorem \ref{thm2} give a complete description of type II radial blow--up solutions as
a sum of rescaled, decoupled solitons $W$.  We expect that such a decomposition still holds for all sequences of times, but at
the moment this remains open.

The outline of this article is as follows.  In Section 2 we gather preliminary results that will be essential to the proof.  These include the
local well--posedness theory for \eqref{nlw} developed in \cite{stab} for higher dimensions, the linear and nonlinear
profile decompositions from \cite{bahger} \cite{bulut}, the exterior energy estimates for the linear wave equation from \cite{klls1}, and
the rigidity of radial solutions to \eqref{nlw} with compact trajectories proved for $3 \leq N \leq 5$ in \cite{dkm1} which we
extend to $N \geq 6$.  In Section 3 we show that no energy can concentrate in the self--similar region of the backwards light cone for type II
blow--up solutions.  In Section 4 we show that no energy can concentrate in the self--similar region of the forward light cone for
global type II  solutions.  The results from these two sections follow from channels of energy arguments that use the exterior energy estimates for the linear wave equation from \cite{klls1}.
Such channels of energy arguments have become very useful in the study of nonlinear wave equations.  See, for example,
\cite{dkm1} \cite{dkm2} \cite{dkm3} \cite{dkm4} \cite{ckls} for the energy-critical equation, \cite{shen} for the energy subcritical
equation, \cite{dkm5} \cite{dodl} for the energy supercritical equation, and \cite{ckls1} \cite{ckls2} \cite{kls1} 
\cite{klls2} for wave maps.
In Section 5 we prove Theorem \ref{thm1} and Theorem \ref{thm2}.  The vanishing of the energy in self--similar regions
established in Section 3 and Section 4 allows us to establish that the errors in the expansions in Theorem \ref{thm1} and Theorem \ref{thm2}
vanish in the Strichartz sense.  By using another channels of energy argument, we are able to show these errors vanish
in the energy space.  Finally, in Section 6 we prove Theorem \ref{below1} and Theorem \ref{below2} using methods recently developed
in \cite{dkm6}.

\textbf{Acknowledgments:}  This work was completed during the author's doctoral studies.  The author would like
to thank his adviser, Carlos Kenig, for introducing him to nonlinear dispersive equations and for his invaluable patience and guidance.

\section{Preliminaries}

In this section, we establish preliminary results that will be essential in the proofs of Theorem \ref{thm1} and \ref{thm2}.
  Unless otherwise indicated, we assume $N \geq 3$ is arbitrary.

\subsection{Cauchy Problem}

In this subsection, we review the the local well--posedness theory for \eqref{nlw} developed in \cite{stab} with $N \geq 6$.  Analogous statements hold in $3 \leq N \leq 5$, and
we refer the reader to Section 2  of \cite{dkm1} for a review.  If $I$ is an interval, we denote
\begin{align*}
S(I) &= L^{\frac{2(N+1)}{N-2}}(I \times \R^N), \\
W(I) &= L^{\frac{2(N+1)}{N-1}}
\left ( I ; \dot B^{\frac{1}{2}, 2}_{\frac{2(N+1)}{N-1}} \right ).
\end{align*}
For $0 < s < 1$ and $1 < p < \infty$, $\dot B^{s, 2}_{p}$ is the standard Besov space on $\R^N$ with norm
\begin{align*}
\| f \|_{\dot B^{s, 2}_{p}} = \left ( \int_{\R^N} |y|^{-N-2s} \| f(x + y) - f(x) \|
_{L^{p}_x}^2 dy \right )^{\frac{1}{2}}.
\end{align*}

Let $S(t)$ be the propagator for the linear wave equation \eqref{lw}. We have that
$$
S(t)(v_0, v_1) = \cos (t |\nabla|)v_0 + \frac{\sin(t |\nabla|)}{|\nabla|}v_1.
$$
By Strichartz estimates, if $v(t) = S(t)(v_0,v_1)$, then
$$
\| v \|_{S(\R)} + \| v \|_{W(\R)} \leq C(N) ( \| (v_0, v_1) \|_{\energysp} ).
$$
A solution to \eqref{nlw} on an interval $I$, where $0 \in I$, is a function $u$ such that $\vec u(t) \in C(I;\energysp)$,
$$
J \Subset I \implies \| u \|_{S(J)} + \| u \|_{W(J)} < \infty,
$$
and $u(t)$ satisfies the Duhamel formulation
$$
u(t) = S(t)(u_0,u_1) + \int_{0}^t \frac{\sin((t-s)|\nabla|)}{|\nabla|}\left ( |u(s)|^{\frac{4}{N-2}}u(s) \right ) ds.
$$

By \cite{stab}, there exists a small $\delta_0 > 0$ such that for any interval
$I$ containing 0 and any $(u_0,u_1) \in \energysp$ such that
$$
\| S(t)(u_0,u_1) \|_{S(I)} < \delta_0,
$$
there exists a unique solution $u$ to \eqref{nlw} on $I$.  Sticking together these
local solutions, we get that for any initial condition $(u_0,u_1) \in
\energysp$, there exists a unique solution $u$ of $\eqref{nlw}$, which
is defined on a maximal interval of definition $I_{\max}(u) =
(T_-(u),T_+(u))$.

If $\| S(t)(u_0,u_1) \|_{S(I)} = \delta < \delta_0$, then $u$ is close
to $S(t)(u_0,u_1)$ in the following sense: there exists a constant $C = C(N) > 0$
and an exponent $\alpha = \alpha(N)$ with $0 < \alpha < \frac{N+2}{N-2}$ such that
\begin{align*}
\| u(\cdot) - S(\cdot)(u_0,u_1) \|_{S(I)} +
\| u(\cdot) - S(\cdot)(u_0,u_1) \|_{W(I)} &+
\sup_{t \in I} \| \vec u(t) - \vec S(t)(u_0,u_1) \|_{\energysp} \\ &\leq
C \delta^\alpha \| (u_0,u_1)\|_{\energysp}^{\frac{N+2}{N-2} - \alpha}.
\end{align*}

Any solution $u$ of \eqref{nlw} satisfies the blow--up criterion
$T_+(u) < \infty \implies \| u \|_{S(0,T_+(u))} = +\infty.$
An analogous statement holds for negative time.  As a consequence,
if $\| u \|_{S(0,T_+(u))} < \infty$, then $T_+(u) = +\infty$.  Furthermore,
in this case, the solution scatters forward in time in $\energysp$:
there exists a solution $v_L$ of the linear equation \eqref{lw} such that
$$
\lim_{t \rightarrow +\infty} \| \vec u(t) - \vec v_L(t) \|_{\energysp} = 0.
$$
An analogous statement holds in negative time.

These facts along with Strichartz estimates yield
that if $\| (u_0, u_1) \|_{\energysp} \ll \delta_0$, then the solution
$u$ of \eqref{nlw} with initial data $(u_0,u_1)$ is global, scatters forward
and backwards in time, and satisfies
$$
\| u \|_{S(\R)} + \| u \|_{W(\R)} + \sup_{t \in \R} \| \vec u(t) \|_{\energysp}
\leq C \| (u_0, u_1) \|_{\energysp}.
$$

\subsection{Profile Decomposition}
An essential tool in our analysis will be the linear and nonlinear profile decompositions of Bahouri and Gerard (see \cite{bahger}).

\subsubsection{Linear profile decomposition} Let $\{(u_{0,n},u_{1,n})\}_n$ be a bounded sequence of radial functions in $\energysp$.  By Theorem 1.1 of \cite{bulut}, there exists a
subsequence of $\{(u_{0,n},u_{1,n})\}_n$ (that will still be denoted by $\{(u_{0,n},u_{1,n})\}_n$) with the following properties.  There
exists a sequence $\{U^j_L\}_{j \geq 1}$ of radial solutions to the linear wave equation \eqref{lw} with initial data $\{(U^j_{0},U^j_{1})\}_j$ in
$\energysp$, and, for $j \geq 1$, sequences $\{\lambda_{j,n}\}_n$, $\{t_{j,n}\}_n$ with $\lambda_{j,n} > 0$, $t_{j,n} \in \R$ that satisfying
the pseudo-orthogonality property
\begin{align}\label{paramorth}
j \neq k \implies \lim_{n \rightarrow \infty} \frac{\lambda_{j,n}}{\lambda_{k,n}} + \frac{\lambda_{k,n}}{\lambda_{j,n}} +
\frac{|t_{j,n}-t_{k,n}|}{\lambda_{j,n}} = + \infty.
\end{align}
such that, if
\begin{align*}
w^J_{0,n} &:= u_{0,n} - \sum_{j = 1}^J \frac{1}{\lambda_{j,n}^{\frac{N-2}{2}}} U^j_L \left (\frac{-t_{j,n}}{\lambda_{j,n}}, \frac{x}{\lambda_{j,n}} \right ) \\
w^J_{0,n} &:= u_{1,n} - \sum_{j = 1}^J \frac{1}{\lambda_{j,n}^{\frac{N}{2}}} \partial_t U^j_L \left (\frac{-t_{j,n}}{\lambda_{j,n}}, \frac{x}{\lambda_{j,n}} \right )
\end{align*}
then
\begin{align*}
\lim_{J \rightarrow \infty} \limsup_{n \rightarrow \infty} \| w^J_n \|_{S(\R)} = 0
\end{align*}
where
\begin{align*}
w^J_n(t) = S(t)(w^J_{0,n},w^J_{1,n}).
\end{align*}
We say that $\{(u_{0,n},u_{1,n})\}_n$ admits a \emph{profile decomposition} with profiles $\{U^j_L\}_j$ and parameters $\{\lambda_{j,n},t_{j,n}\}_{j,n}$.
The following expansions hold for all $J \geq 1$:
\begin{align*}
\|u_{0,n}\|_{\dot H^1}^2 &= \sum_{j = 1}^J \left \| U^j_L
\left (\frac{-t_{j,n}}{\lambda_{j,n}} \right ) \right \|_{\dot H^1}^2 + \| w^J_{0,n} \|^2_{\dot H^1} + o_n(1) \\
\|u_{1,n}\|_{L^2}^2 &= \sum_{j = 1}^J \left \| \partial_t U^j_L
\left (\frac{-t_{j,n}}{\lambda_{j,n}} \right ) \right \|_{L^2}^2 + \| w^J_{1,n} \|^2_{L^2} + o_n(1) \\
E(u_{0,n},u_{1,n}) &= \sum_{j = 1}^J E \left( U^j_L
\left (\frac{-t_{j,n}}{\lambda_{j,n}} \right ), \partial_t U^j_L
\left (\frac{-t_{j,n}}{\lambda_{j,n}} \right ) \right ) + E(w_{0,n}^J, w_{1,n}^J ) + o_n(1)
\end{align*}
We will denote
$$
U^j_{L,n}(t,x) =  \frac{1}{\lambda_{j,n}^{\frac{N-2}{2}}} U^j_L \left (\frac{t-t_{j,n}}{\lambda_{j,n}}, \frac{x}{\lambda_{j,n}} \right ).
$$
By rescaling and time-translating each profile $U^j_L$ and be extracting subsequences, we can, without loss of generality, assume for each
fixed $j$ that either we have
\begin{align}\label{timeassump}
t_{j,n} = 0 \mbox{ for all } n \geq 1,\mbox{ or } \lim_{n \rightarrow \infty} \frac{-t_{j,n}}{\lambda_{j,n}} = \pm \infty.
\end{align}

The following lemma from \cite{dkm1} (see Lemma 2.5)  will allow us to obtain control of the parameters appearing in a profile decomposition.

\begin{lem}\label{bddlem}
Let $\{(u_{0,n},u_{1,n})\}_n$ be a bounded sequence of functions in $\energysp$ admitting a profile decomposition with
profiles $\{U^j_L\}_j$ and parameters $\{\lambda_{j,n},t_{j,n}\}_{j,n}$. Let $\{ \mu_n \}_n$ be a sequence of positive numbers such that
$$
\lim_{R \rightarrow \infty} \limsup_{n \rightarrow \infty}
\int_{|x| \geq R \mu_n} |\nabla u_{0,n}|^2 + |u_{1,n}|^2 dx = 0.
$$
Then, for all $j$ the sequences $\left \{ \frac{\lambda_{j,n}}{\mu_{n}}\right \}_n$ and $\left \{ \frac{t_{j,n}}{\mu_{n}}\right \}_n$ are bounded.  Moreover, there exists at most one $j$ such that $\left \{ \frac{\lambda_{j,n}}{\mu_{n}}\right \}_n$ does not converge to 0.
\end{lem}

A technical tool used in obtaining new profile decompositions from old ones is the following lemma proved in \cite{dkm1} (see Lemma 4.1) in odd dimensions and later in \cite{cks} (see Lemma 9) in even
dimensions.

\begin{lem}\label{displem}
Let $\{(w_{0,n},w_{1,n})\}_n$ be a bounded sequence of radial functions in $\energysp$ and let $w_n(t) = S(t)(w_{0,n},w_{1,n})$.  Suppose
that
$$
\lim_{n \rightarrow \infty} \| w_n \|_{S(\R)} = 0.
$$
Let $\chi \in C^\infty_0$ be radial so that $\chi \equiv 1$ on $|x| \leq 1$ and supp $\psi \subset \{ |x| \leq 2 \}$.  Let
$\{\lambda_n\}_n$ be a sequence of positive numbers and consider the truncated data
$$
(\tilde w_{0,n} \tilde w_{1,n}) = \varphi( x / \lambda_n) (w_{0,n}, w_{1,n}),
$$
where wither $\varphi = \chi$ or $\varphi = 1- \chi$. Let $\tilde w_n(t) = S(t)(\tilde w_{0,n},\tilde w_{1,n})$.  Then
$$
\lim_{n \rightarrow \infty} \| \tilde w_n \|_{S(\R)} = 0.
$$
\end{lem}

\subsubsection{Nonlinear profile decomposition}
Using the local well-posedness of \eqref{nlw} for the first case appearing in \eqref{timeassump} and the existence of wave operators
for \eqref{nlw} for the second case appearing in \eqref{timeassump}, one can construct a radial solution $U^j$ of \eqref{nlw} such
that $-t_{j,n} / \lambda_{j,n} \in I_{\max}(U^j)$ for large $n$ and
\begin{align*}
\lim_{n \rightarrow \infty} \left \| \vec U^j \left ( \frac{-t_{j,n}}{\lambda_{j,n}} \right )
- \vec U^j_L \left ( \frac{-t_{j,n}}{\lambda_{j,n}} \right )\right \|_{\energysp} = 0.
\end{align*}
We call $U^j$ the \emph{nonlinear profile} associated to $\{U^j_L\}$, $\{\lambda_{j,n},t_{j,n}\}$.  Note that in the case that
$$
\lim_{n \rightarrow \infty} \frac{-t_{j,n}}{\lambda_{j,n}} = +\infty,
$$
the nonlinear profile $U^j$ scatters at $+ \infty$, i.e. if $T_-(U^j) < T$, then
$$
\| U^j \|_{S(T,+\infty)} < \infty.
$$
A similar statement holds if $\lim_{n} -t_{j,n} / \lambda_{j,n} = - \infty$.  We will denote
$$
U^j_n(t,x) = \frac{1}{\lambda_{j,n}^{\frac{N-2}{2}}} U^j \left (\frac{t-t_{j,n}}{\lambda_{j,n}}, \frac{x}{\lambda_{j,n}} \right ).
$$

Because the profiles interact so weakly due to the orthogonality of the parameters, we have the following approximate superposition
principle for the nonlinear evolution.

\begin{ppn}\label{approxthm}
Let $\{(u_{0,n},u_{1,n})\}_n$ be a bounded sequence of radial functions in $\energysp$ admitting a profile decomposition with profiles $\{U^j_L\}$ and parameters $\{\lambda_{j,n},t_{j,n}\}$.
Let $\theta_n \in [0,+\infty)$.  Assume
\begin{align*}
\forall j \geq 1, \mbox{ } \forall n \geq 1, \mbox{ } \frac{\theta_n - t_{j,n}}{\lambda_{j,n}} < T_+(U^j) \mbox{ and }
\limsup_{n \rightarrow \infty} \| U^j \|_{S \left ( \frac{-t_{j,n}}{\lambda_{j,n}}, \frac{\theta_n-t_{j,n}}{\lambda_{j,n}} \right )}
< \infty.
\end{align*}
Let $u_n$ be the solution of \eqref{nlw} with initial data $(u_{0,n},u_{1,n})$.  Then for large $n$, $u_n$ is defined on
$[0,\theta_n]$,
\begin{align*}
\limsup_{n \rightarrow \infty} \|u_n\|_{S   (0,\theta_n) } < \infty,
\end{align*}
and
\begin{align*}
\forall t \in [0,\theta_n],\mbox{ }  \vec{u}_n(t,x) = \sum_{j = 1}^J  \vec{U}^j_n(t,x) +  \vec{w}^J_n(t,x) +  \vec{r}^J_n(t,x),
\end{align*}
with
\begin{align*}
\lim_{J \rightarrow \infty} \limsup_{n \rightarrow \infty} \left [ \| r^J_n \|_{S \left ( 0,\theta_n \right )}
+ \sup_{t \in [0,\theta_n]} \left \| \vec r^J_n(t) \right \|_{\energysp} \right ]  = 0.
\end{align*}
An analogous statement holds if $\theta_n < 0$.
\end{ppn}

Proposition \ref{approxthm} was proved previously for
$3 \leq N \leq 5$ in \cite{dkm1}.  We give a proof of Proposition \ref{approxthm} for $N \geq 6$ in Appendix A using a stability result from \cite{stab}.

In the above decomposition, we have the following localized orthogonality property.
\begin{lem}\label{locorth}
Let $\{u_{0,n},u_{1,n})\}_n,$ $\{U^j_L\}_j$, $\{t_{j,n},\lambda_{j,n}\}_{j,n}$ and $\theta_n$ satisfy the assumptions of Proposition \ref{approxthm}.  Let
$\{\rho_n\}_n, \{\sigma_n\}$ be sequences with $0 \leq \rho_n < \sigma_n \leq +\infty$.  Then
\begin{align*}
j \neq k \implies \lim_{n \rightarrow \infty} \int_{\rho_n \leq |x| \leq \sigma_n} \nabla_{t,x} U^j_n(\theta_n,x) \cdot \nabla_{t,x}
U^k_n(\theta_n,x) \hspace{2pt} dx &= 0, \\
j \leq J \implies \lim_{n \rightarrow \infty} \int_{\rho_n \leq |x| \leq \sigma_n} \nabla_{t,x} U^j_n(\theta_n,x) \cdot \nabla_{t,x}
w^J_n(\theta_n,x)& \\  + \nabla_{t,x} U^j_n(\theta_n,x) \cdot \nabla_{t,x}
r^J_n(\theta_n,x)  \hspace{2pt} dx &= 0.
\end{align*}
\end{lem}

The proof for even $N$ was given in \cite{cks} (see Corollary 8).  We give a proof of Lemma \ref{locorth} for odd $N$ in Appendix B.

\subsubsection{Ordering of the profiles} Let $\{(u_{0,n},u_{1,n})\}_n$ be a bounded sequence of radial functions in $\energysp$ admitting a profile decomposition with
profiles $\{U^j_L\})j$ and parameters $\{\lambda_{j,n},t_{j,n}\}_{j,n}$, and let $U^j$ be their nonlinear profiles.  We introduce the following pre-order $\preccurlyeq$ on the
profiles as follows.

\begin{defn}\label{proforddef}
For $j,k \geq 1$, we say that
$$
\{ U^j_L, \{t_{j,n},\lambda_{j,n}\}\} \preccurlyeq \{ U^k_L, \{t_{k,n},\lambda_{k,n}\}\}
$$
or simply $(j) \preccurlyeq (k)$ if there is no ambiguity, if one of the following holds:
\begin{enumerate}
\item $U^k$ scatters forward in time,
\item $U^j$ does not scatter forward in time and
$$
\forall T < T_+(U^j), \mbox{ } \lim_{n \rightarrow \infty}
\frac{\lambda_{j,n}T + t_{j,n} - t_{k,n} }{\lambda_{k,n}} < T_{+}(U^k)
$$
The above limit exists due to the arguments in \cite{dkm6}.
\end{enumerate}
We say that $(j) \prec (k)$ if
$$
(j) \preccurlyeq (k) \mbox{ and } \left ((k) \preccurlyeq (j) \mbox{ does not hold. } \right )
$$
\end{defn}

The following lemma proved in \cite{dkm6} (see Claim 3.7) states that we may always assume the profiles appearing
in a profile decomposition are ordered.

\begin{lem}\label{proford}
Let $\{(u_{0,n},u_{1,n})\}_n$ be a bounded sequence of radial functions in $\energysp$ admitting a profile decomposition with profiles $\{U^j_L\}$ and parameters $\{\lambda_{j,n},t_{j,n}\}$.
Then one can assume without loss of generality that the profiles are ordered, that is
$$
\forall j \leq k, \mbox{ } (j) \preccurlyeq (k).
$$
\end{lem}
\subsection{Energy Trapping}
Using variational arguments stemming from the fact that $W$ is the unique (up to translation, scaling, and change of sign) minimizer in the
Sobolev inequality, we have the following lemma from \cite{dkm1} (see Claim 2.3).
\begin{lem}\label{energytrap}
Let $f \in \dot H^1(\R^N)$.  Then
\begin{align*}
\|\nabla f \|^2_{L^2} \leq \| \nabla W\|^2_{L^2} \mbox{ and }
E(f,0) \leq E(W,0) \implies \| \nabla f \|_{L^2}^2 \leq N E(f,0).
\end{align*}
Moreover, there exists $c > 0$ such that if $\| \nabla f \|_{L^2}^2 \leq  \left ( \frac{N}{N-2} \right )^{\frac{N-2}{2}} \| \nabla W \|^2_{L^2}$, then
\begin{align*}
E(f,0) \geq c \min \left \{ \|\nabla f \|^2_{L^2}, \left ( \frac{N}{N-2} \right )^{\frac{N-2}{2}}  \| \nabla W \|^2_{L^2} - \| \nabla f \|_{L^2}^2 \right \}
\end{align*}
\end{lem}

\subsection{Linear Behavior}
We now state a localization property and exterior energy estimates for solutions to the linear wave equation \eqref{lw}.
Lemma \ref{loclem} can be easily deduced in odd dimensions by the strong Huygens principle (see Lemma 4.1 in \cite{dkm1}), but is in fact true also in
even dimensions (see Lemma 9 in \cite{cks}).

\begin{lem}\label{loclem}
Let $v$ be a solution of the linear wave equation \eqref{lw}, with initial data $(v_0,v_1)$.  Let $\{\lambda_n\}_n$, $\{t_n\}_n$ be two sequences with $\lambda_n > 0$ and $t_n \in \R$.  Let
$$
v_n(t,x) = \frac{1}{\lambda_n^{\frac{N-2}{2}}} v\left ( \frac{t}{\lambda_n}, \frac{x}{\lambda_n} \right )
$$
and assume $\lim_{n \rightarrow \infty} \frac{t_n}{\lambda_n} = \ell  \in [-\infty,+\infty]$.  Then, if $\ell = \pm \infty$,
$$
\lim_{R \rightarrow \infty} \limsup_{n \rightarrow \infty}
\int_{\left | |x| - |t_n| \right | \geq R \lambda_n} | \nabla v_n(t_n) |^2 + | \partial_t v_n(t_n) |^2 + \frac{|v_n(t_n)|^2}{|x|^2} dx
= 0,
$$
and if $\ell \in \R$,
$$
\lim_{R \rightarrow \infty} \limsup_{n \rightarrow \infty}
\int_{\{|x| \geq R \lambda_n\} \\ \cup \{|x| \leq \frac{1}{R} \lambda_n\}} | \nabla v_n(t_n) |^2 + | \partial_t v_n(t_n) |^2 + \frac{|v_n(t_n)|^2}{|x|^2} dx
= 0
$$
\end{lem}

Exterior energy bounds for solutions to the linear equation were first shown for $N = 3$ in \cite{dkm1} (see Lemma 4.2) and for $N = 5$ in \cite{kls1} (see Proposition 4.1). These bounds were recently generalized
to all odd dimensions in \cite{klls1} (see Corollary 1 and Theorem 2).

\begin{ppn}\label{extbds}
Let $N \geq 3$ be odd.  Let $v$ be a solution to the linear wave equation \eqref{lw} with radial initial data $(v_0,v_1)$.  Then for every $R > 0$
\begin{align}
\max_{\pm} \inf_{\pm t \geq 0} \int_{|x| > R + |t|}
|\nabla v(t)|^2 + |\partial_t v(t)|^2 dx \geq \frac{1}{2} \left \| \pi_{P(R)}^{\perp} (v_0,v_1) \right \|^2_{\energysp(|x| > R)} \label{extbd1}
\end{align}
Here
$$
P(R) := \mbox{span } \left \{(|x|^{2k_1-1},0), (0,|x|^{2k_2 - 1}) : k_1 = 1,\ldots, \left [ \frac{N+2}{4} \right ],
 k_2 = 1,\ldots, \left [ \frac{N}{4} \right ] \right \},
$$
and $\pi^{\perp}_{P(R)}$ denotes the orthogonal projection onto the compliment of this plane.  The left hand side of \eqref{extbd1}
vanishes for all data in $P(R)$.

In the case $R = 0$,  we have
\begin{align}
\max_{\pm} \inf_{\pm t \geq 0} \int_{|x| > |t|}
|\nabla v(t)|^2 + |\partial_t v(t)|^2 dx \geq \frac{1}{2} \left \| (v_0,v_1) \right \|^2_{\energysp} \label{extbd2}
\end{align}
\end{ppn}

We remark that Proposition \ref{extbds}, specifically \eqref{extbd2}, is false in even dimensions for arbitrary data $(v_0,v_1)$.  However, if $N \equiv 0 \mbox{ mod }4$, then
\eqref{extbd2} holds for data of the form $(v_0,0)$.  If $N \equiv 2 \mbox{ mod }4$, then
\eqref{extbd2} holds for data of the form $(0,v_1)$.  For more details, see \cite{cks}.

The right-hand side of \eqref{extbd1} can be computed explicitly viz
 \begin{align*}
 \left \| \pi_{P(R)}^{\perp} (v_0,v_1) \right \|^2_{\energysp(|x| > R)} =&
 \int_R^\infty \left ( | \nabla v_0(r) |^2 + |v_1(r)|^2 \right ) r^4 \vspace{2pt} dr \\
&- \sum_{i,j = 1}^{\left [ \frac{N+2}{4} \right ]} \lambda_i(R) \lambda_j(R) \frac{R^{2i + 2j - N - 2}}{N + 2 - 2i - 2j} \\
&- \sum_{i,j = 1}^{\left [ \frac{N}{4} \right ]} \mu_i(R) \mu_{j}(R) \frac{R^{2i + 2j - N}}{N - 2i - 2j}
 \end{align*}
where
\begin{align*}
\lambda_j(R) &= \sum_{i = 1}^{\left [ \frac{N+2}{4}\right ]}
\frac{-R^{N + 2 - 2i - 2j}}{(N-2j)(N+2-2i-2j)} a_i a_j \int_R^\infty v_0'(r) r^{2i - 2}dr, \quad 1 \leq j \leq \left [\frac{N+2}{4}\right ], \\
\mu_j(R) &= \sum_{i = 1}^{\left [ \frac{N}{4}\right ]}
\frac{R^{N - 2i - 2j}}{N-2i-2j} b_i b_j \int_R^\infty v_1(r) r^{2i - 1}dr, \quad 1 \leq j \leq \left [\frac{N}{4}\right ],
\end{align*}
with explicit constants $a_i$'s and $b_i$'s depending only on $i$ and $N$.  For more details, see \cite{klls1}.

\subsection{Classification of solutions with the compactness property}
Finally, we have the following classification of radial solutions $u$ to \eqref{nlw} with pre-compact trajectories in
$\energysp$ up to symmetries.  We say that a radial solution $u$ to $\eqref{nlw}$ has the \emph{compactness property} on an
interval $J \subseteq I_{\max}(u)$ if there exists a function $\lambda(t) > 0$, $t \in J$ so that the trajectory
$$
K = \left \{ \left ( \lambda(t)^{(N-2)/2} u(t,\lambda(t) \cdot), \lambda(t)^{N/2} \partial_t u(t,\lambda(t) \cdot) \right ): t \in J  \right \}
$$
is pre-compact in $\energysp$.
\begin{thm}\label{cptsolns}
Let $u$ be a nonzero radial solution of $\eqref{nlw}$.  Assume that $u$ has the compactness property on $I_{\max}(u)$.  Then there exists $\lambda_0 > 0$ and a sign $\iota_0 \in \{\pm 1\}$ such that
$$
u(t,x) = \frac{\iota_0}{\lambda_0^{\frac{N-2}{2}}} W\left ( \frac{x}{\lambda_0}\right ).
$$
\end{thm}

Theorem \ref{cptsolns} was proved in \cite{dkm1} (see Theorem 2) for $3 \leq N \leq 5$.  We give a proof of
Theorem \ref{cptsolns} for $N \geq 6$ in Appendix C.

\section{Self-similar regions for blow--up solutions}

In this section we show that no energy of the singular part (see below) of the solution can concentrate in the self similar region
$|x| \simeq T_+ - t$.  This allows us to deduce that the expansion in \eqref{thm1} holds in a slightly weaker sense (see Section 5).

\subsection{Extraction of the linear term} By \cite{dkm1} (see Section 3), there exists $(v_0,v_1) \in \energysp$ such that
$$
(u(t),\partial_t u(t)) \rightharpoonup (v_0,v_1) \quad \mbox{as } t \rightarrow 1^-
$$
weakly in $\energysp$.  In fact, one has the stronger property that if $\varphi \in C^{\infty}_0$ and if $\varphi = 1$ near the origin, then
\begin{align}\label{loccgc}
(1-\varphi)(u(t)-v_0,\partial_t u(t)-v_1) \rightarrow 0 \quad \mbox{as } t \rightarrow 1^-
\end{align}
in $\energysp$.  The construction of $(v_0,v_1)$ in \cite{dkm1} is stated for $N \leq 5$ but requires only the small data theory which is valid for all $N \geq 3$.

 Let $v$ be the solution to \eqref{nlw} with initial data $(v_0,v_1)$ at time $t = 1$.  We call $v$ the \emph{regular} part of $u$ and $a = u - v$ its \emph{singular} part.  Note $a$ is well-defined on $[t_-,1)$ for some $t_- > \max \{T_-(u),T_-(v) \}$.  By finite speed of propagation and \eqref{loccgc}, it follows that $\vec a$ is supported in the cone
 \begin{align}\label{asupp}
 \left \{ (t,x) : t_- \leq t < 1 \mbox{ and } |x| \leq 1-t \right \}.
 \end{align}

We claim $\lim_{t \rightarrow 1^-} E(\vec a(t) )$ exists, and in fact,
\begin{align}\label{alimit1}
\lim_{t \rightarrow 1^-} E(\vec a(t)) = E(u_0,u_1) - E(v_0,v_1).
\end{align}
Indeed, since $(u(t),\partial_t u(t)) \rightharpoonup
(v_0, v_1)$ and $(v(t), \partial_t v(t)) \rightarrow (v_0,v_1)$ in $\energysp$ as $t \rightarrow 1^-$, we have that
\begin{align*}
\int_{\R^N} |\nabla_{t,x} (u - v)(t)|^2 dx &= \int_{\R^N} |\nabla_{t,x} u(t)|^2 + |\nabla_{t,x} v(t)|^2 - 2 \nabla_{t,x} u \cdot \nabla_{t,x} v(t) dx
\\ &= \int |\nabla_{t,x} u(t)|^2 - |\nabla_{t,x} v(t)|^2 dx + o(1)
\end{align*}
as $t \rightarrow 1^-$.
Now
$$
\int_{\R^N} |u(t)|^{\frac{2N}{N-2}} dx = \int_{\R^N} |v(t)|^{\frac{2N}{N-2}} dx + \int_{\R^N} |a(t)|^{\frac{2N}{N-2}} dx + \epsilon(t),
$$
where
$$
\epsilon(t) \lesssim \int_{\R^N}\left ( |v(t)|^{\frac{N+2}{N-2}}|a(t)| + |v(t)||a(t)|^{\frac{N+2}{N-2}} \right )dx.
$$
Since supp $\vec a(t) \subseteq \{ |x| \leq 1-t \}$ and $v(t)$ is regular up to $t = 1$, it easily follows that $\lim_{t 
\rightarrow 1^-} \epsilon(t) = 0$ which yields the claim.

The goal of this section is to prove the following:
\begin{ppn}\label{selfsimilarblowup}
Let $u$ and $a$ be as above.  Then
\begin{align*}
\forall c_0 \in (0,1), \quad \lim_{t \rightarrow 1^-}
\int_{c_0(1-t) \leq |x| \leq 1-t} \left ( |\nabla a(t,x)|^2 +
|\partial_t a(t,x)|^2 \right ) dx = 0.
\end{align*}
\end{ppn}

The proof is based on the main argument from Section 3 of \cite{dkm3} adapted to higher dimensions.

\subsection{Renormalization}  Assume without loss of generality that $t_- = 0$.   The proof proceeds by contradiction.  Assume that there exists a sequence $\{t_n\}_n$ in $(0,1)$ and $c_0 > 0$ such that 
\begin{align*}
\forall n, \mbox{ } t_n < 1 \mbox{ and } \lim_{n \rightarrow \infty} t_n = 1, \\
\int_{c_0(1-t_n) \leq |x| \leq 1-t_n} \left ( |\nabla a(t_n,x)|^2 +
|\partial_t a(t_n,x)|^2 \right ) dx \geq c_0.
\end{align*}
Let $T \in (0,1)$, and let
\begin{align*}
I_n = \left [ \frac{T-t_n}{1-t_n}, 1 \right ).
\end{align*}
Note that for large $n$, $[0,1) \subset I_n$.  Define for $\tau \in I_n$,
\begin{align*}
u_n(\tau,y) &= (1-t_n)^{(N-2)/2} u\left (t_n + (1-t_n)\tau, (1-t_n)y \right ), \\
v_n(\tau,y) &= (1-t_n)^{(N-2)/2} v\left (t_n + (1-t_n)\tau, (1-t_n)y \right ),
\end{align*}
and note that \eqref{asupp} implies
\begin{align}\label{ansupp}
\supp \left (\vec u_n(\tau) - \vec v_n(\tau) \right ) \subseteq \{ |y| \leq 1 - \tau \}.
\end{align}

\subsection{Profile decomposition and analysis of the profiles}
After extraction of a subsequence, suppose $\left \{ \vec u_n(0) - \vec v_n(0) \right \}_n$ admits a profile decomposition with profiles $\{ U^j_L \}_j$
 and parameters $\{\lambda_{j,n}, \tau_{j,n} \}_{j,n}$:
   $$
   \vec u_n(0,y) = \vec v_n(0,y) + \sum_{j = 1}^J \vec U^j_{L,n}(0,y) + \vec w^J_n(0,y).
   $$

   As usual, we will assume that for all $j$,
\begin{align*}
\tau_{j,n} = 0 \mbox{ for all } n \geq 1,\mbox{ or } \lim_{n \rightarrow \infty} \frac{-\tau_{j,n}}{\lambda_{j,n}} = \pm \infty,
\end{align*}
and that the following limits exist in $[-\infty,\infty]$:
\begin{align*}
\tilde \tau_j &= \lim_{n \rightarrow \infty} -\tau_{j,n}, \\
\tilde \lambda_j &= \lim_{n \rightarrow \infty} \lambda_{j,n}
\end{align*}
\begin{lem}\label{taulim}
Fix $j \geq 1$.  Then
\begin{align*}
\tilde \tau_j \in [-1,1]
\end{align*}
and the sequence $\{ \lambda_{j,n} \}_n$ is bounded.
\end{lem}
\begin{proof}
By \eqref{ansupp} and an application of Lemma \ref{bddlem} with $\mu_n = 1$ for all $n$, 
it is immediate that for all $j$, the sequences $\{\lambda_{j,n} \}_n$, $\{-\tau_{j,n} \}_n$ are bounded.  Hence, $\tilde \tau_j, \tilde \lambda_j \in \R$.

Fix $j \geq 1$.  If $\tau_{j,n} = 0$ for all $n$, then we are done. Assume
$$
\lim_{n \rightarrow \infty} \frac{-\tau_{j,n}}{\lambda_{j,n}} = \pm \infty,
$$
and, seeking a contradiction that $|\tilde \tau_j| = 1 + \eta > 1$.  Then $\lim_{n \rightarrow \infty} \frac{-\tau_{j,n}}{\lambda_{j,n}} = \pm \infty$
implies $\tilde \lambda_j = 0$.  Let $\epsilon > 0$.  By Lemma \ref{loclem}, there exists 
$R = R(\epsilon) > 0$ such that 
\begin{align*}
\limsup_{n \rightarrow \infty} \int_{\left | |y| - |\tau_{j,n}| \right | \geq R \lambda_{j,n}} |\nabla U^j_{L,n}(0)|^2 + |\partial_t U^j_{L,n}(0)|^2 dy 
< \epsilon.
\end{align*}
Since $|\tilde \tau_j| = 1 + \eta > 1$ and $\lim_{n \rightarrow \infty} \lambda_{j,n} = 0$,  we have for large $n$
$$
\{ \left | |y| - |\tau_{j,n} | \right | \leq R \lambda_{j,n} \} \subset
\{ |y| \geq 1 \}.
$$
Then, \eqref{ansupp}, Lemma \ref{locorth}, and our choice of $R$ imply that as $n \rightarrow \infty$
\begin{align*}
0 &= \int_{|y| \geq 1} | \nabla (u_n - v_n)(0) |^2 + |\partial_t (u_n - v_n)(0) |^2 dy \\
&\geq  \int_{|y| \geq 1} | \nabla U^j_{L,n}(0) |^2 + |\partial_t U^j_{L,n}(0) |^2 dy + o(1) \\
&\geq \int_{\left | |y| - |\tau_{j,n} | \right | \leq R \lambda_{j,n}} \nabla U^j_{L,n}(0) |^2 + |\partial_t U^j_{L,n}(0) |^2 dy + o(1) \\
&\geq \| \vec U^j_{L,n}(0) \|_{\energysp}^2 - \epsilon + o(1).
\end{align*}
This implies by conservation of the free energy, 
$$
\| (U^j_0, U^j_1) \|^2_{\energysp} = \| \vec U^j_{L,n}(0) \|_{\energysp}^2 < 2 \epsilon.
$$
Since $\epsilon$ was arbitrary, we have that $U^j_L \equiv 0$, a contradiction.  Thus, $|\tilde \tau_j|\leq 1$. 
\end{proof}
Since $\vec u(t_n) \rightharpoonup \vec v(0)$ in $\energysp$, the term $\vec v_n (0)$ may be viewed as a profile
\begin{align*}
\left ( \frac{1}{\lambda_{0,n}^{\frac{N-2}{2}}}U^0_L \left ( \frac{-\tau_{0,n}}{\lambda_{0,n}} \right ), \left ( \frac{y}{\lambda_{0,n}} \right ), \frac{1}{\lambda_{0,n}^{\frac{N-2}{2}}}
\partial_\tau U^0_L \left ( \frac{-\tau_{0,n}}{\lambda_{0,n}} \right ), \left ( \frac{y}{\lambda_{0,n}} \right )\right )
\end{align*}
up to a term converging to 0 in $\energysp$.  Namely:
$$
(U^0_{0,L},U^0_{1,L}) = (v_0,v_1), \quad \lambda_{0,n} = \frac{1}{1-t_n}, \quad \tau_{0,n} = 0.
$$

By the pseudo-orthogonality of the parameters, there is at most one index $j \geq 1$ such that $\tilde \lambda_j > 0$.  For this profile, we have
$$
\frac{\left |\tau_{j,n}\right |}{\lambda_{j,n}} \leq C,
$$
for all $n$.  Hence, for this profile we can assume $\tau_{j,n} = 0$ for all $n$.  Reordering the profiles, we will always assume that this index is 1, setting $U^1 = 0$ if there is no index with the preceding property.  Since $\tilde \lambda_1 > 0$, after rescaling we can also assume that $\lambda_{1,n} = 1$ for all $n$.

\begin{clm}
The profile $\vec U^1(0)$ is supported in $\{ |y| \leq 1 \}$.
\end{clm}

\begin{proof}
Let $(\psi_0,\psi_1) \in C^{\infty}_0 \times C^{\infty}_0$ with $\supp (\psi_0,\psi_1) \subset \{ |y| > 1 \}$.  By the 
construction of a profile decomposition, we have 
$$
\vec u_n(0) - \vec v_n(0) \rightharpoonup \vec U^1(0)
$$
in $\energysp$. The support property \eqref{ansupp} then implies
\begin{align*}
0 &= \lim_{n \rightarrow \infty} \int \nabla \psi_0 \cdot \nabla (u_n(0) - v_n(0) ) + \psi_1 (u_n(0) - v_n(0)) dy \\
&= \int \psi_0 \cdot \nabla U^1(0) + \psi_1 \partial_\tau U^1(0) dy.
\end{align*}
The claim follows.
\end{proof}

The idea of the proof of Proposition \ref{selfsimilarblowup} is to rule out three cases, one of which must arise in the above
 profile decomposition
in order for energy concentration to occur in the self similar region.  Informally, these cases are: $U^1 \neq 0$ and all of the other profiles are concentrating
energy closer to the origin than $U^1$, or there is a profile $U^j$, $j \geq 2$, that is concentrating energy further into the self-similar region than $U^1$,
or all of the profiles are concentrating energy at the origin (in particular $U^1 = 0$) and energy concentration in the self-similar region arises from the dispersive error $w^J_n$.

We will now make these three cases more precise.  Let
\begin{align*}
r_0 = \inf \{ \rho \in [0,1]: \supp \vec U^1(0) \subseteq \{ |y| \leq \rho \} \}.
\end{align*}
The three cases mentioned previously can be distinguished as follows:
\begin{itemize}
\item Case 1: $r_0 > 0$ and for all $j \geq 2$, $|\tilde \tau_j | \leq r_0$.
\item Case 2: there exists $j \geq 2$ such that $|\tilde \tau_j | > r_0$.
\item Case 3: $r_0 = 0$ and for all $j \geq 2$, $\tilde \tau_j = 0$.
\end{itemize}

These three possible cases yield the following lemma as a consequence.

\begin{lem}\label{channelofenergyI}
Let $I_n^+ = I_n \cap [1/2, +\infty)$, and let $I_n^- = I_n \cap (-\infty,-1/2]$.  Then there exists $\eta_1 > 0$ such that the following holds for large $n$;
for all sequences $\{\theta_n\}_n$ with $\theta_n \in I_n^+$, or for all sequences $\{\theta_n\}_n$ with $\theta_n \in I_n^-$
\begin{align*}
\int_{|y| \geq |\theta_n|} \left | \nabla_{\tau,y} (u_n - v_n)(\theta_n,y) \right |^2 dy \geq \eta_1.
\end{align*}
\end{lem}

 This will contradict the following claim and conclude the proof of Proposition \ref{selfsimilarblowup}.

\begin{clm}\label{vanishingclaimI}
Let $u$ satisfy the assumptions of Proposition \ref{selfsimilarblowup}.  Let $\theta_n^+ = \frac{1}{2}$ and $\theta_n^- = \frac{T-t_n}{1-t_n}$. Then for large $n$, $\theta^{\pm}_n \in I^{\pm}_n$ and
\begin{align*}
\lim_{n \rightarrow \infty} \int_{|y| \geq |\theta_n^{\pm}|}
\left |\nabla_{\tau_,y} (u_n - v_n)(\theta^{\pm}_n,y) \right |^2 dy = 0.
\end{align*}
\end{clm}

\begin{proof}[Proof of Claim \ref{vanishingclaimI}]
Going back to the original variables, we write
\begin{align}\label{origvarI}
\int_{|y| \geq |\theta_n^{\pm}|}
\left |\nabla_{\tau_,y} (u_n - v_n)(\theta^{\pm}_n,y) \right |^2 dy = \int_{|x| \geq (1-t_n)|\theta_n^{\pm}|}
\left |\nabla_{t,x} (u - v)(t_n + (1-t_n)\theta^{\pm}_n,x) \right |^2 dx.
\end{align}
For $\theta_n^+$, the right-hand side of \eqref{origvarI} is equal to $$
\int_{|x| \geq \frac{1-t_n}{2}}
\left |\nabla_{t,x} a\left ( \frac{1+t_n}{2} , x \right ) \right |^2 dx = 0, \quad \mbox{for all } n \geq 1,
$$
since $\supp \vec a\left ( \frac{1+t_n}{2} \right ) \subseteq \left \{ |x| \leq \frac{1-t_n}{2} \right \}$.
For $\theta_n^-$, the right-hand side of \eqref{origvarI} is equal to $$
\int_{|x| \geq t_n - T}
\left |\nabla_{t,x} a\left ( T, x \right ) \right |^2 dx \rightarrow \int_{|x| \geq 1 - T}
\left |\nabla_{t,x} a\left ( T, x \right ) \right |^2 dx = 0, \quad \mbox{as } n \rightarrow \infty,
$$
since $t_n \rightarrow 1$ and $\supp \vec a \left ( T \right ) \subseteq \{ |x| \leq 1-T \}$.
\end{proof}

\begin{proof}[Proof of Lemma \ref{channelofenergyI}]  We begin with Case 1.

\subsubsection*{Case 1}  Denote by $r = |y|$ the radial variable.  By definition of $r_0$, $\supp \vec U^1(0) \subseteq 
\{ |y| \leq r_0 \}$ and for small positive $\epsilon_0 < r_0 /100$,
\begin{align*}
\int_{r_0 - \epsilon < |y| < r_0} \left | \nabla U_0^1(0) \right |^2 + |U^1_1|^2 dy = \eta_0
\end{align*}
is small and positive.  Define radial functions $\left ( \tilde U_0, \tilde U_1 \right ) \in \energysp$ by
\begin{align*}
\tilde U_0(r) &=
  \begin{cases}
   U^1_0(r_0 - \epsilon_0) & \text{if } 0 \leq r \leq r_0 - \epsilon_0 \\
   U^1_0(r) & \text{if }  r > r_0 - \epsilon_0
  \end{cases} \\
  \tilde U_1(r) &=
  \begin{cases}
   0 & \text{if } 0 \leq r \leq r_0 - \epsilon_0 \\
   U^1_1(r) & \text{if }  r > r_0 - \epsilon_0
  \end{cases}.
\end{align*}
  Let $\tilde U$ be the solution to \eqref{nlw} with initial data $(\tilde U_0, \tilde U_1)$.  Note that,
by the small data theory,  $\|(\tilde U_0, \tilde U_1 ) \|^2_{\energysp} = \eta_0$ and $\eta_0$ 
sufficiently small imply $\tilde U$ is globally defined and forward and backwards in time.

  Let $\varphi \in C^{\infty},$ be radial such that $\varphi(y) = 1$ if $|y| \geq r_0 - \epsilon_0$ and $\varphi(y) = 0$ if $|y| \leq r_0 - 2 \epsilon_0$.  Fix $J \geq 1$ and define
  \begin{align}
  \tilde u_{0,n} &= v_n(0) + \tilde U(0) + \varphi \sum_{j = 2}^{J} U^j_{L,n}(0) + \varphi w^J_{0,n}, \\
   \tilde u_{0,n} &= \partial_\tau v_n(0) + \partial_\tau \tilde U(0) + \varphi \sum_{j = 2}^{J} \partial_\tau U^j_{L,n}(0) + \varphi w^J_{1,n}.
  \end{align}

The above definition is independent of $J$.

Let $\tilde u_n$ be the solution of \eqref{nlw} with data $\left (\tilde u_{0,n}, \tilde u_{1,n} \right )$ at $\tau = 0$.  Note that
$$
|y| \geq r_0 - \epsilon_0 \implies (u_n(0,y), \partial_\tau u_n(0,y)) = \left (\tilde u_{0,n}(y), \tilde u_{1,n}(y) \right ).
$$
By finite speed of propagation, if $\tau \in I_{\max}(u_n) \cap I_{\max}(\tilde u_n)$, then
$$
|y| \geq r_0 - \epsilon_0 + |\tau| \implies (u_n(\tau,y), \partial_\tau u_n(\tau,y)) = \left (\tilde u_n(\tau,y), \tilde u_{n}(\tau, y) \right ).
$$
We divide the proof into four steps.

\subsubsection*{Step 1} In this step we show that for $\epsilon_0$ sufficiently small, either for all $\tau \geq 0$ or for all $\tau \leq 0$ we have
\begin{align}
\int_{r_0 - \epsilon_0 + |\tau| \leq |x| \leq r_0 + |\tau|} | \nabla_{\tau,x} \tilde U(\tau) |^2 dx \geq \frac{ \eta_0}{4}. \label{step1blowup}
\end{align}
Note that by choosing $\epsilon_0$ sufficiently small, we can insure that
\begin{align}\label{cptchan1}
\left \| \pi^{\perp}_{P(r_0 - \epsilon_0)} \left (\tilde U_0, \tilde U_1 \right ) \right \|^2_{\energysp} \geq \frac{3 \eta_0}{4}.
\end{align}
Indeed, by the Cauchy-Schwarz inequality for $1 \leq i \leq \left [ \frac{N+2}{4} \right ]$
\begin{align*}
\int_{r_0 - \epsilon_0}^{r_0} \tilde U_0'(r) r^{2i - 2} dr \leq C(r_0) \epsilon_0^{\frac{1}{2}} \eta_0^{\frac{1}{2}}.
\end{align*}
Similarly for $1 \leq i \leq \left [ \frac{N}{4} \right ]$
\begin{align*}
\int_{r_0 - \epsilon_0}^{r_0} \tilde U_1(r) r^{2i - 1} dr \leq C(r_0) \epsilon_0^{\frac{1}{2}} \eta_0^{\frac{1}{2}}.
\end{align*}
By the remark following \eqref{extbds}, for $\epsilon_0$ sufficiently small
\begin{align*}
\left \| \pi_{P(r_0 - \epsilon_0)}^\perp (\tilde U_0, \tilde U_1) \right \|^2 \geq \eta_0 - C(N,r_0) \epsilon_0 \eta_0 \geq \frac{3\eta_0}{4}
 \end{align*}
Denote by $\tilde U_L$ the solution to the linear wave equation \eqref{lw} with data $\vec {\tilde U}(0) = (\tilde U_0, \tilde U_1)$.  As $\eta_0 = \left \| \vec {\tilde U}(0) \right \|^2_{\energysp}$ is small,
the small data theory for the Cauchy problem \eqref{nlw} implies that for all $\tau \in \R$,
\begin{align*}
\left \| \vec {\tilde U}(\tau) - \vec {\tilde U}_L (\tau)  \right \|^2_{\energysp} \leq C \eta_0^{\frac{N+2}{N-2}}
\end{align*}
By Proposition \ref{extbds}, \eqref{cptchan1}, and by choosing $\eta_0$ sufficiently small, we have either for all $\tau \geq 0$ or for all $\tau \leq 0$
\begin{align*}
\int_{r_0 - \epsilon_0 + |\tau|}^{r_0 + |\tau|} | \nabla_{\tau,r} \tilde U(\tau,r) |^2 r^{N-1} dr &\geq \int_{r_0 - \epsilon_0 + |\tau|}^{r_0 + |\tau|} | \nabla_{\tau,r}
\tilde U_L(\tau,r) |^2 r^{N-1} dr - C \eta_0^{\frac{N+2}{N-2}} \\
&\geq \frac{1}{2} \left \| \pi^{\perp}_{P(r_0 - \epsilon_0)} \left (\tilde U_0, \tilde U_1 \right ) \right \|^2_{\energysp} - C \eta_0^{\frac{N+2}{N-2}} \\
&\geq \frac{\eta_0}{4}
\end{align*}
as desired.  Let us assume that \eqref{step1blowup} holds for al $\tau \geq 0$.  The case $\tau \leq 0$ is similar.

\subsubsection*{Step 2}  In this step, we show that for each fixed $J \geq 2$,
\begin{align}
\varphi(y) \sum_{j = 1}^J \left ( U^j_{L,n}(0,y) , \partial_\tau U^j_{L,n}(0,y) \right ) =
\sum_{j = 1}^J \varphi(|\tilde \tau_j|) \left ( U^j_{L,n}(0,y) , \partial_\tau U^j_{L,n}(0,y) \right ) + o(1)
\label{step2blowupa}
\end{align}
as $n \rightarrow \infty$ in $\energysp$.

Let $2 \leq j \leq J$.  We first consider the case that $\lim_{n \rightarrow \infty} \frac{-\tau_{j,n}}{\lambda_{j,n}} \in \{\pm \infty \}$. Let $\epsilon > 0$.
Since $j \geq 2$,
we have that $\lambda_{j,n} \rightarrow 0$ as $n \rightarrow \infty$. Observe that
\begin{align}
\left \| \varphi \vec U^j_{L,n}(0) - \varphi(|\tilde \tau_j|) \vec U^j_{L,n}(0) \right \|_{\energysp}^2 &\lesssim \int |\nabla \varphi(y)|^2 |U^j_{L,n}(0,y)|^2 dy \label{step21}\\
&+ \int |\varphi(y) - \varphi(|\tilde \tau_j|)|^2 |\nabla_{\tau,y}U_{L,n}^j(0,y)|^2 dy \label{step22}.
\end{align}
By Lemma \ref{loclem}, there exists $R = R(\epsilon) > 0$ such that for $n$ sufficiently large
$$\int_{||y| - |\tau_{j,n}|| \geq R \lambda_{j,n}} |\nabla U_{L,n}^j(0,y)|^2 + |\partial_t U_{L,n}^j(0,y)|^2 +
\frac{|U_{L,n}^j(0,y)|^2}{|x|^2} dy < \epsilon.$$
Hence \eqref{step22} can be estimated as $n \rightarrow \infty$ by
\begin{align*}
\eqref{step22} &\lesssim \epsilon + \sup_{||y| - |\tau_{j,n}|| \leq R \lambda_{j,n}} |\varphi(y) - \varphi(|\tilde \tau_j|)|^2
\| (U^j_0,U^j_1) \|^2_{\energysp} \\
&\lesssim \epsilon + o(1),
\end{align*}
where the implied constant depends only on $\varphi$ and $\|(U^j_0,U^j_1)\|_{\energysp}$.
Similarly, using H\"older's inequality and Sobolev embedding $\dot H^1 \hookrightarrow L^{\frac{2N}{N-2}}$,
we have as $n \rightarrow \infty$
\begin{align*}
\eqref{step21} &\lesssim  \int_{||y| - |\tau_{j,n}|| \geq R \lambda_{j,n}}
\frac{|U_{L,n}^j(0,y)|^2}{|x|^2} dy + |\{ ||y| - |\tau_{j,n}|| \leq R \lambda_{j,n} \}|^{\frac{2}{N}} \| \nabla U_{L,n}^j(0) 
\|^2_{L^2} \\
&\lesssim \epsilon + R^2 o(1).
\end{align*}
where the implied constant depends only on $\varphi$,$r_0$, and $\|(U^j_0,U^j_1)\|_{\energysp}$. This proves \eqref{step2blowupa}
in the case that $\lim_{n \rightarrow \infty} \frac{-\tau_{j,n}}{\lambda_{j,n}} \in \{\pm \infty \}$.

Suppose now that  we are in the case $\tau_{j,n} = 0$ for all $n \geq 1$ (recall we always assume that
either $\tau_{j,n} = 0$ for all $n$ or $\lim_n -\tau_{j,n}/\lambda_{j,n} \in \{\pm \infty \}$) .  Then $\tilde \tau_j = 0$.  Therefore, we wish to
show that
\begin{align}\label{step23}
\lim_{n \rightarrow \infty} \left \| \varphi \vec U^j_{L,n}(0) \right \|_{\energysp}^2 = 0.
\end{align}
We estimate
\begin{align*}
\left \| \varphi \vec U^j_{L,n}(0) \right \|_{\energysp}^2 &\lesssim \int |\nabla \varphi(y)|^2 |U^j_{L,n}(0,y)|^2 dy
+ \int |\varphi(y)|^2 |\nabla_{\tau,y}U_{L,n}^j(0,y)|^2 dy.
\end{align*}
By Lemma \ref{loclem}, there exists $R = R(\epsilon) > 0$ such that for $n$ sufficiently large
$$\int_{\{|y| \geq R \lambda_{j,n}\} \cup \{|y| \leq \frac{1}{R}\lambda_{j,n}\}} |\nabla U_{L,n}^j(0,y)|^2 + |\partial_t U_{L,n}^j(0,y)|^2 +
\frac{|U_{L,n}^j(0,y)|^2}{|x|^2} dy < \epsilon.$$
Since $\lim_{n \rightarrow \infty} \lambda_{j,n} = 0$, we have  $\{ |y| \geq r_0 - 2 \epsilon_0 \}
\subset \{ |y| \geq R \lambda_{j,n}\}$ for $n$ sufficiently large. This fact, our choice of $R$, and the support properties of $\varphi$ imply
\begin{align*}
\limsup_{n \rightarrow \infty} &\left (
\int |\nabla \varphi(y)|^2 |U^j_{L,n}(0,y)|^2 dy
+ \int |\varphi(y)|^2 |\nabla_{\tau,y}U_{L,n}^j(0,y)|^2 dy \right ) \\
&\lesssim \limsup_{n \rightarrow \infty} \int_{\{|y| \geq R \lambda_{j,n}\}} |\nabla U_{L,n}^j(0,y)|^2 + |\partial_t U_{L,n}^j(0,y)|^2 +
\frac{|U_{L,n}^j(0,y)|^2}{|x|^2} dy \lesssim \epsilon,
\end{align*}
where the implied constant depends only on $\varphi$.  Since $\epsilon$ was arbitrary, this proves \eqref{step23} and concludes
Step 2.

\subsubsection*{Step 3}  In this step we show that for each fixed $j \geq 2$, we have:
\begin{align}\label{vandisp1}
\lim_{n \rightarrow \infty} \left \| S(t) \left (\varphi(|\tilde \tau_j|)  U^j_{L,n}(0,y),
\varphi(|\tilde \tau_j|) \partial_\tau U^j_{L,n}(0,y) \right ) \right \|_{S(0,r_0/4)} = 0,
\end{align}
and that
\begin{align}\label{vandisp2}
\lim_{J \rightarrow \infty} \limsup_{n \rightarrow \infty} \left \| S(t) (\varphi w^J_{0,n}, \varphi w^J_{1,n} )\right \|_{S(\R)} = 0.
\end{align}
Indeed \eqref{vandisp2} follows from Lemma \ref{displem}.  To show \eqref{vandisp1}, we write
\begin{align}\label{explint}
 \left \| S(t) \left ( U^j_{L,n}(0,y) ,
\partial_\tau U^j_{L,n}(0,y) \right ) \right \|_{S(0,r_0/4)}^{\frac{2(N+1)}{N-2}} =
\int^{\frac{r_0/4 - \tau_{j,n}}{\lambda_{j,n}}}_{-\frac{\tau_{j,n}}{\lambda_{j,n}}} \int_{\R^N} \left | U^j_L(\tau,y)\right |^{\frac{2(N+1)}{N-2}} dy d\tau.
\end{align}
If $|\tilde \tau_j| \leq r_0/4$, then $\varphi(|\tilde \tau_j|) = 0$ and \eqref{vandisp1} follows.  If $\tilde \tau_{j} > r_0/4$ (respectively $\tilde \tau_{j} <-r_0/4$), then
the lower limit (respectively upper limit) in the integral \eqref{explint} tends to $+\infty$ (respectively $-\infty$) and \eqref{vandisp1} follows.

\subsubsection*{Step 4}  In this step we conclude the proof of Lemma \ref{channelofenergyI} for Case 1.  By Step 2 and \eqref{vandisp2}
from Step 3,  we have that the sequence $\left \{ \left (\tilde u_{0,n}, \tilde u_{1,n} \right )\right \}_n$ admits a profile decomposition
\begin{align*}
\tilde u_{0,n} &= v_n(0) + \tilde U_0(0) + \sum_{j = 2}^J \varphi(|\tau_j|) U^{j}_{L,n}(0) + \tilde w^{J}_{0,n}, \\
\tilde u_{1,n} &= \partial_\tau v_n(0) + \tilde U_1(0) + \sum_{j = 2}^J \varphi(|\tau_j|) \partial_\tau U^{j}_{L,n}(0) + \tilde w^{J}_{1,n},
\end{align*}
where
$$
\lim_{J \rightarrow \infty} \limsup_{n \rightarrow \infty} \| \tilde w^J_n \|_{S(\R)} = 0.
$$
By Proposition \ref{approxthm}, Step 2 \eqref{vandisp1}, and the local theory for the Cauchy problem we have that
\begin{align*}
\vec{\tilde u}_n\left ( \frac{r_0}{4} \right ) = \vec v_n \left ( \frac{r_0}{4} \right ) + \vec{\tilde U} \left ( \frac{r_0}{4} \right ) +
\sum_{j = 2}^J \varphi(|\tilde \tau_j|) \vec{U}^j_{L,n} \left ( \frac{r_0}{4} \right ) + \vec{\tilde w}^J_n \left ( \frac{r_0}{4} \right ) +
\vec{\tilde r}^J_n \left ( \frac{r_0}{4} \right )
\end{align*}
with
\begin{align*}
\lim_{J \rightarrow \infty}  \limsup_{n \rightarrow \infty} \left \| \vec{ \tilde r}^J_n \left ( \frac{r_0}{4} \right ) \right \|_{\energysp}   =  0 .
\end{align*}

Let $\psi \in C^{\infty}$ be radial such that $\psi(y) = 1$ if $|y| \geq r_0$, and $\psi(y) = 0$ if $|y| \leq 9r_0/10$.  Fix $J \geq 1$, and define
\begin{align*}
( \check u_{0,n}, \check u_{1,n} ) = \varphi (y) \left ( \vec{\tilde u}_n \left ( \frac{r_0}{4} \right )  
- \vec{v}_n \left ( \frac{r_0}{4} \right ) \right ).
\end{align*}
Let $\check u_n$ be the solution to $\eqref{nlw}$ with $\vec{\check u}_n \left ( \frac{r_0}{4}\right ) =  \left (\check u_{0,n},
 \check u_{1,n} \right )$, and note that by finite speed of propagation
\begin{align*}
|y| \geq \frac{5 r_0}{4} - \epsilon_0 \implies
\left (
\check u_n \left ( \frac{r_0}{4}, y \right ), \partial_\tau \check u_n \left ( \frac{r_0}{4}, y \right ) \right ) = \left (
u_n \left ( \frac{r_0}{4}, y \right ), \partial_\tau u_n \left ( \frac{r_0}{4}, y \right ) \right ).
\end{align*}

By the same reasoning as in Step 2, we see that for all $ 2 \leq j \leq J $,
\begin{align}\label{step4eq1}
\psi(y) \varphi(|\tilde \tau_j|) \vec U^J_{L,n}\left ( \frac{r_0}{4}, y \right ) =
\psi \left ( \left |\frac{r_0}{4} + \tilde \tau_j \right | \right ) \varphi(|\tilde \tau_j|) \vec U^J_{L,n}\left ( \frac{r_0}{4},
 y \right )
+ o(1)
\end{align}
in $\energysp$ as $n \rightarrow \infty$.  Since $|\tilde \tau_j| \leq r_0$ for all $j$, if $\tilde \tau_j < 0$ then $\left |\frac{r_0}{4} + \tilde \tau_j \right |$
is outside the support of $\psi$.  If $\tilde \tau_j > 0$, then $\psi \left ( \left |\frac{r_0}{4} + \tilde \tau_j \right | \right ) \varphi(|\tilde \tau_j) |
= \varphi(\tilde \tau_j)$.  These facts and \eqref{step4eq1} imply that $( \check u_{0,n}, \check u_{1,n} )$ admits the following profile decomposition
\begin{align*}
\check u_{0,n} &= v_n \left ( \frac{r_0}{4} \right ) + \psi(y) \tilde U \left ( \frac{r_0}{4} \right ) +
 \sum_{\tilde \tau_j > 0} \varphi(|\tau_j|) U^j_{L,n} \left ( \frac{r_0}{4} \right ) + \check w^J_{0,n}, \\
\check u_{1,n} &= \partial_\tau v_n \left ( \frac{r_0}{4} \right ) + \psi(y) \partial_\tau \tilde U \left ( \frac{r_0}{4} \right ) +
 \sum_{\tilde \tau_j > 0} \varphi(|\tau_j|) \partial_\tau U^j_{L,n} \left ( \frac{r_0}{4} \right )
+ \check w^J_{1,n}
\end{align*}
where
$$
\lim_{J \rightarrow} \limsup_{n \rightarrow \infty} \| S(t) (\check w_{0,n}^J, \check w_{1,n}^J \|_{S(\R)} = 0.
$$
By the same reasoning as in Step 3, we see that for every $j \geq 2$ with $\tilde \tau_j > 0$,
\begin{align*}
\lim_{n \rightarrow \infty} \left \| S(t) \left (\varphi(|\tilde \tau_j|)  U^j_{L,n} \left ( \frac{r_0}{4},y \right ),
\varphi(|\tilde \tau_j|) \partial_\tau U^j_{L,n} \left (\frac{r_0}{4},y \right ) \right ) \right \|_{S(r_0 / 4,+\infty)} = 0.
\end{align*}
By the smallness of $\eta_0$, the solution to \eqref{nlw} with initial data $ \left ( \psi \tilde U \left ( \frac{r_0}{4} \right ), \psi \tilde \partial_\tau \tilde U
 \left ( \frac{r_0}{4} \right ) \right )$ is globally defined, scatters, and coincides with $\tilde U$ for any $\tau,y$ such that
$\tau \geq r_0/4$, $|y| \geq \tau + r_0 - \epsilon_0$.  Hence by Proposition \ref{approxthm}, if $\tau \geq r_0 / 4$ and
$|y| \geq \tau + r_0 - \epsilon_0$,
\begin{align}\label{case1step4}
\vec u_n (\tau,y) - \vec v_n(\tau, y) = \vec{\tilde U}(\tau,y) +
\sum_{\tilde \tau_j > 0} \varphi (\tilde \tau_j) \vec U^j_{L,n}(\tau, y) + \vec{\check w}^J_n(\tau,y) + \vec{\check r}^J_n(\tau,y)
\end{align}
where
\begin{align*}
\lim_{J \rightarrow \infty} \limsup_{n \rightarrow \infty} \left [ \sup_{\tau \geq r_0/4} \left \| \vec{ \check r}^J_n \left ( \tau \right )
\right \|_{\energysp} + \| \check r_n^J \|_{S((r_0 / 4, \infty))} \right ]  =  0 .
\end{align*}

Let $\theta_n \in I_n^+ \subset \left [\frac{r_0}{4}, \infty \right )$.  By Lemma \ref{locorth} and Step 1, for large $n$ we have that
\begin{align*}
\int_{r_0 - \epsilon_0 + \theta_n \leq |y| \leq r_0 + \theta_n} &|\nabla_{\tau,y} (u_n - v_n) (\theta_n ,y) |^2 dy \\&\geq
\int_{r_0 - \epsilon_0 + \theta_n \leq |y| \leq r_0 + \theta_n} |\nabla_{\tau,y} \tilde U (\theta_n ,y) |^2 dy + o_n(1) \\
 &\geq \frac{\eta_0}{8}.
\end{align*}
This proves Lemma \ref{channelofenergyI} for Case 1.

\subsubsection*{Case 2} Assume that there exists $j \geq 2$ such that $|\tilde \tau_j| > r_0$.  Let $\tilde r_0 = \sup_{j \geq 2} |\tilde
\tau_j| > r_0$, and choose $\epsilon_0 < \frac{\tilde r_0 - r_0}{100}$.  By reordering the profiles, we may assume that
$\tilde r_0 - \epsilon_0 < |\tilde \tau_2 | \leq \tilde r_0$.  Assume that $\tilde \tau_2 > 0$.  The case $\tilde \tau_2 < 0$
is similar. Let $\varphi \in C^{\infty},$ be radial such that $\varphi(y) = 1$ if $|y| \geq \tilde r_0 - \epsilon_0$ and
$\varphi(y) = 0$ if $|y| \leq \tilde r_0 - 2 \epsilon_0$.  Define
\begin{align}
\tilde u_{0,n} &= v_{0,n} + \varphi \sum_{j = 2}^J U^j_{L,n}(0) + \varphi w^J_{0,n} \label{case2a}\\
\tilde u_{1,n} &= v_{1,n} + \varphi \sum_{j = 2}^J \partial_\tau U^j_{L,n}(0) + \varphi w^J_{1,n}. \label{case2b}
\end{align}
The proof proceeds exactly as in Case 1 (without ($\tilde U, \partial_\tau \tilde U)$), replacing $r_0$ wit $\tilde r_0$ up to
\eqref{case1step4}.  

Let $\theta_n \in I_n^+ \subset \left [\frac{\tilde r_0}{4}, \infty \right )$.  Since $\lambda_{2,n} \rightarrow 0$,  by
Lemma \ref{loclem} there exists $R_0 > 0$ such that
$$
\limsup_{n \rightarrow \infty}
\int_{\left | |y| - (\theta_n - \tau_{2,n}) \right | \geq R_0 \lambda_{2,n}} | \nabla_{\tau,y} U^2_{L,n}(\theta_n,y) |^2 dy
< \frac{1}{2}\|(U_0^2,U_1^2) \|_{\energysp}^2.
$$
For $n$ sufficiently large, we have $\{ |y| \leq \theta_n + \tilde r_0 - \epsilon_0 \} \subset 
\{ \left | |y| - (\theta_n - \tau_{2,n}) \right | \geq R_0 \lambda_{2,n}  \}$.
Indeed, if $|y| \leq \theta_n + \tilde r_0 - \epsilon_0$, then as $n \rightarrow \infty$
\begin{align*}
\theta_n - \tau_{2,n} - |y| &\geq - \tau_{2,n} - (\tilde r_0 - \epsilon_0) \\
&= \tilde \tau_2 - (\tilde r_0 - \epsilon) + o(1) \\
&\geq R_0 \lambda_{2,n}.
\end{align*}
Lemma \ref{locorth} and conservation of the free energy imply that as $n \rightarrow \infty$
\begin{align*}
\int_{\tilde r_0 - \epsilon_0 + \theta_n \leq |y|} &|\nabla_{\tau,y} (u_n - v_n) (\theta_n ,y) |^2 dy \\&\geq
\int_{\tilde r_0 - \epsilon_0 + \theta_n \leq |y|} |\nabla_{\tau,y} U^2_{L,n} (\theta_n ,y) |^2 dy + o(1) \\
&= \int_{\R^N} |\nabla_{\tau,y} U^2_{L,n}(\theta_n,y)|^2 dy - \int_{|y| \leq \tilde r_0 - \epsilon_0 + \theta_n}
 |\nabla_{\tau,y} U^2_{L,n}(\theta_n,y)|^2 dy + o(1) \\
&\geq \| (U^2_0,U^1_0) \|_{\energysp}^2 -
\int_{\left | |y| - (\theta_n - \tau_{2,n}) \right | \geq R_0 \lambda_{2,n}} | \nabla_{\tau,y} U^2_{L,n}(\theta_n,y) |^2 dy + o(1) \\
&\geq \frac{1}{4} \| (U^2_0,U^1_0) \|_{\energysp}^2
\end{align*}
This concludes Case 2.

\subsubsection*{Case 3} We now assume that $r_0 = 0$ (so that $U^1 \equiv 0$) and $\tilde \tau_j = 0$ for all $j$.  Our assumption on $u$ implies that that
there exists $\epsilon_0 > 0$ such that for all $J \geq 2$
\begin{align*}
\limsup_{n \rightarrow \infty} \int_{|y| \geq \epsilon_0} |\nabla w^J_{n,0}|^2 + |w^J_{1,n}|^2 dy \geq \eta_0.
\end{align*}
Let $\varphi \in C^\infty$ be radial such that $\varphi = 1$ if $|y| \geq \epsilon_0$ and $\varphi = 0$ if $|y| \leq \epsilon_0/2$.
Define $(\tilde u_{0,n}, \tilde u_{1,n})$ as in \eqref{case2a} and \eqref{case2b} from Case 2, and let $\tilde u_n$ be the solution to \eqref{nlw} with initial data
$(\tilde u_{0,n}, \tilde u_{1,n})$.  Then Lemma \ref{loclem} and Lemma \ref{displem} imply
\begin{align*}
\lim_{n \rightarrow \infty} \left \| \sum_{j =2}^J \varphi \vec{U}^j_{L,n}(0)  \right \|_{\energysp} &= 0 , \\
\lim_{J \rightarrow \infty} \limsup_{n \rightarrow \infty} \| S(\tau) (\varphi w^J_{0,n},\varphi w^J_{1,n}) \|_{S(\R)} &= 0.
\end{align*}
Hence
\begin{align*}
\tilde u_{0,n} = v_{0,n} + \tilde w_{0,n}, \quad u_{1,n} = v_{1,n} + \tilde w_{1,n},
\end{align*}
where
\begin{align*}
\lim_{n \rightarrow \infty} \| S(\tau) (\tilde w_{0,n}, \tilde w_{1,n}) \|_{S(\R)} = 0, \\
\limsup_{n \rightarrow \infty} \int_{|y| \geq \epsilon_0} |\nabla \tilde w_{0,n} |^2 + | \partial_\tau \tilde w_{0,n}|^2 dy \geq \eta_0.
\end{align*}
  In particular,
$(\tilde w_{0,n}, \tilde w_{1,n})  \rightharpoonup 0$ in $\energysp$ as $n \rightarrow \infty$.  This implies (since $P(\epsilon_0)$ is finite dimensional)
\begin{align*}
\lim_{n \rightarrow \infty} \left \| \pi_{P(\epsilon_0)} (\tilde w_{0,n}, \tilde w_{1,n}) \right \|^2_{\energysp( |x| \geq \epsilon_0)} = 0
\end{align*}
which yields
\begin{align*}
\limsup_{n \rightarrow \infty} \left \| \pi^\perp_{P(\epsilon_0)} (\tilde w_{0,n}, \tilde w_{1,n}) \right \|^2_{\energysp( |x| \geq \epsilon_0)} \geq \eta_0.
\end{align*}
Hence, either for all $\tau \geq 0$ or for all $\tau \leq 0$ and for all $n$ sufficiently large
$$
\int_{\epsilon_0 + |\tau| < |y|}  |\nabla \tilde w_n(\tau) |^2 + |\partial_\tau \tilde w_n(\tau)|^2 dy \geq \frac{\eta_0}{2}
$$
where $w_n(t) = S(t)(\tilde w_{0,n}, \tilde w_{1,n} )$.  We then argue as in Case 1 and Case 2 using Proposition \ref{extbds} and Proposition \ref{approxthm} directly for all
$\tau \geq 0$ or for all $\tau \leq 0$ globally in time.  This concludes Case 3.
\end{proof}

\section{Self--similar regions for global solutions}

In this section we construct the linear term $v_L$ appearing in Theorem \ref{thm2} and show that no energy of $u - v_L$ can concentrate in the self similar region
$|x| \simeq t$.  This allows us to deduce that the expansion in Theorem \ref{thm2} holds in a slightly weaker sense (see Section 5).  After
constructing $v_L$, the proof is very similar to the finite time blow--up case.

\subsection{Extraction of the linear term}  In this subsection we construct the linear solution appearing in Theorem \ref{thm2}.  The
main property that this solution has is that it completely captures the linear behavior of the nonlinear wave in exterior regions.
More precisely, we have the following.

\begin{ppn}\label{globallin}
Let $N \geq 3$.  Let $u$ be a radial solution of \eqref{nlw} such that $T_+(u) = +\infty$.  Then there exists a radial solution $v_L$ of the
linear wave equation \eqref{lw} such that
$$
S(-t)\vec u(t) \rightharpoonup \vec v_L(0) \quad \mbox{in } \energysp \mbox{ as } t \rightarrow +\infty,
$$
and for all $A \in \R$
\begin{align}
\lim_{t \rightarrow +\infty} \int_{|x| \geq t - A} |\nabla_{t,x}(u - v_L)(t,x)|^2 dx = 0. \label{linbehavior}
\end{align}
\end{ppn}

The proof of Proposition \ref{globallin} is essentially the same as those in \cite{dkm3} ($N = 3$) and \cite{ckls} ($N = 4)$).
Let $\{\varphi_\delta\}_\delta$ be a family of radial $C^{\infty}$ functions on $\R^N$, defined for $\delta>0$ such that
\begin{align*}
0 \leq \varphi_\delta \leq 1, \quad |\nabla \varphi_{\delta}| \leq C \delta^{-1}, \\
\varphi_\delta(x) =
\begin{cases}
1 & \mbox{if } |x| > 1 - \delta, \\
0 & \mbox{if } |x| < 1 - 2 \delta.
\end{cases}
\end{align*}

\begin{lem}\label{globallinlem}
Let $u$ be as in Proposition \eqref{globallin}, and let $\epsilon > 0$ be small.  Then there exists $t_n \rightarrow +\infty$ and
$\delta > 0$ small such that $\varphi_{\delta}\left ( \frac{x}{t_n} \right ) \vec u(t_n)$ admits a profile decomposition with profiles
$\{ U^j_L \}_j$ and parameters $\{\lambda_{j,n}, t_{j,n}\}_{j,n}$ such that
\begin{align}
\forall j \geq 2, \quad \lim_{n \rightarrow \infty} -&\frac{t_{j,n}}{\lambda_{j,n}} = + \infty, \label{globallinlem1} \\
\forall n \geq 1,\mbox{ } t_{1,n} = 0 \quad &\mbox{and} \quad \| (U^1_0, U^1_1) \|_{\energysp} \leq \epsilon. \label{globallinlem2}
\end{align}
\end{lem}
\begin{proof} We divide the proof into two steps.
\subsubsection*{Step 1}  In this step we show that there exists $\delta' > 0$ and a sequence $s_n \rightarrow +\infty$ such that
$\left \{ \varphi_{\delta'} \vec u(s_n) \right \}_n$ admits a profile decomposition with profiles $\{V_j\}_j$ and parameters
$\{\mu_{j,n},s_{j,n}\}_{j,n}$ satisfying
\begin{align*}
\forall j \geq 2, \quad \lim_{n \rightarrow \infty} -\frac{s_{j,n}}{\mu_{j,n}} \in \{\pm \infty\} \quad &\mbox{and} \quad
\lim_{n \rightarrow \infty} -\frac{s_{j,n}}{s_n} \in [-1,2\delta'-1] \cup [1-2\delta',1], \\
s_{1,n} = 0 \quad &\mbox{and} \quad \|(V^1_0,V_1^1)\|_{\energysp} \leq \frac{\epsilon}{2}.
\end{align*}
First, note that by finite speed of propagation and the small data theory
\begin{align}\label{globallinlem3}
\lim_{R \rightarrow \infty} \limsup_{t \rightarrow +\infty} \int_{|x| \geq R + t} |\nabla u(t) |^2 + (\partial u(t))^2 dx = 0.
\end{align}
Indeed, let $\epsilon > 0$ be small, and define $(\tilde u_0, \tilde u_1) \in \energysp$ by
\begin{align*}
\tilde u_0(r) &=
  \begin{cases}
   u_0(R) & \text{if } 0 \leq r \leq R \\
   u_0(r) & \text{if }  r > R
  \end{cases}, \\
  \tilde u_1(r) &=
  \begin{cases}
   0 & \text{if } 0 \leq r \leq R \\
   u_1(r) & \text{if }  r > R
  \end{cases}.
\end{align*}
Choose $R > 0$ large enough so that
\begin{align*}
\| (\tilde u_0, \tilde u_1) \|^2_{\energysp} = \int_{|x| \geq R} |\nabla u_0|^2 + |u_1|^2 dx < \epsilon^2 \ll \delta_0^2,
\end{align*}
where $\delta_0$ is from the local Cauchy theory.  Let $\tilde u$ be the solution to \eqref{nlw} with initial data $(\tilde u_0,\tilde u_1)$.
Then by the local Cauchy theory, $\tilde u$ is global and
$$
\sup_{t \in \R} \| (\tilde u(t), \partial_t \tilde u(t) \|_{\energysp} \leq C \epsilon.
$$
By finite speed of propagation, $\tilde u(t,x) = u(t,x)$ for $|x| \geq R + t$ so
$$
\limsup_{t \rightarrow +\infty} \int_{|x| \geq R + t} |\nabla u(t) |^2 + (\partial u(t))^2 dx \leq C \epsilon
$$
as desired.

Since $u(t)$ remains bounded in the energy space on $[0,+\infty)$ by assumption, there exists a sequence $\{s_n\}_n$ with $s_n \rightarrow +\infty$
such that $\{ \vec u(s_n) \}_n$ admits a profile decomposition with profiles $\{\tilde V^j_L \}_j$, parameters
$\{\mu_{j,n}, s_{j,n} \}_{j,n}$, and remainder $(\tilde w^J_{0,n}, \tilde w^J_{1,n})$.  As we can always extract subsequences, without loss of generality we will always assume that all real valued sequences converge in
$\bar \R$.  In particular, either
\begin{align*}
\lim_{n \rightarrow \infty} \frac{-s_{j,n}}{\mu_{j,n}} = \pm \infty, \quad \mbox{or} \quad \forall n \geq 1,\mbox{ } s_{j,n} = 0.
\end{align*}
Define
$$
\tau_j = \lim_{n \rightarrow \infty} \frac{-s_{j,n}}{s_n}.
$$

\begin{clm}
For all $j$, $|\tau_j| \leq 1$, and $\lim_{n \rightarrow \infty} \frac{\mu_{j,n}}{s_n} = 0,$ except for at most one index $j$, for which the limit is finite.
\end{clm}
\begin{proof}
Consider $(u_{0,n},u_{1,n}) = (u(s_n), \partial_t u(s_n))$.  By \eqref{globallinlem3},
$$
\lim_{R \rightarrow \infty} \limsup_{n \rightarrow \infty} \int_{|x| \geq Rs_n} |\nabla u_{0,n}|^2 + |u_{1,n}|^2 dx = 0.
$$
Hence we can apply Lemma \ref{bddlem} and deduce that for all $j$, $|\tau_j| < \infty$.  Moreover $\lim_{n \rightarrow \infty} \frac{\mu_{j,n}}{s_n} < \infty$, and for all $j$
except at most one, the limit is 0.

If $\tau_j = 0$ then we are done.  Suppose $\tau_j \neq 0$ and that $|\tau_j| = 1 + \eta$, $\eta > 0$.  In particular, this means we are also assuming that
$$
\lim_{n \rightarrow \infty} \frac{|s_{j,n}|}{\mu_{j,n}} = \infty.
$$
By \eqref{globallinlem3} and Lemma \ref{locorth},
\begin{align*}
\lim_{R \rightarrow +\infty} \limsup_{n \rightarrow \infty} \int_{|x| \geq s_n + R} |\nabla_{t,x} \tilde V^j_{L,n}(0) |^2 dx = 0.
\end{align*}
We will combine this will Lemma \eqref{loclem}.  Let $\epsilon > 0$.  There exists $R$ and $N_0$ such that for all $n \geq N_0$
$$
\int_{||x| - |s_{j,n}|| \geq R \mu_{j,n}} |\nabla_{t,x} \tilde V^j_{L,n}(0)|^2 dx < \epsilon.
$$
We note that for fixed $\tilde R$ large and for $n$ large,
\begin{align*}
\{ ||x|-|s_{j,n}|| \leq R \mu_{j,n} \} \subset \{ |x| \geq s_n + \tilde R \}.
\end{align*}
Indeed if $||x| - |s_{j,n}|| \leq R \mu_{j,n}$, then
\begin{align*}
|x| \geq |s_{j,n}| - R \mu_{j,n} &= \frac{|s_{j,n}|}{s_n} \left ( 1 - R \frac{\mu_{j,n}}{|s_{j,n}|} \right ) s_n \\
&= ( 1 + \eta + o(1)) (1 - R o(1) ) s_n \\
&\geq (1 + \delta) s_n
\end{align*}
for some $\delta > 0$ and for all $n$ sufficiently large.  Since $s_n \rightarrow +\infty$, for $n$ sufficiently large $\{ |x| \geq s_n + \delta s_n \} \subset \{ |x| \geq s_n + \tilde R \}$ which proves the
desired inclusion.

Hence for all $n \geq N_1$,
$$
\int_{\R^N} |\nabla_{t,x} \tilde V^j_{L,n}(0) |^2 dx \leq 2 \epsilon.
$$
By conservation of the free energy and since $\epsilon$ was arbitrary, $\tilde V^j_L \equiv 0$, which is a contradiction.  Thus, $|\tau_j| \leq 1$.
\end{proof}
Next, not that if $j$ is such that $\lim_{n \rightarrow \infty} \frac{\mu_{j,n}}{s_n} > 0$ (and finite by the previous claim), we cannot have $\lim_{n \rightarrow \infty}
\frac{|s_{j,n}|}{\mu_{j,n}} = \infty$.  Hence $s_{j,n} = 0$ for all $n$.  This happens for at most one $j$. We assume $j = 1$.  By rescaling, we can also assume
$\mu_{1,n} = s_n$.

We now claim
\begin{align}
\supp \left ( \vec{\tilde V}^1(0) \right ) \subseteq \{ |x| \leq 1\}.
\end{align}
Indeed, first note that by definition of the profiles, $$\left (s_n^{\frac{N-2}{2}} u(s_n, s_n \cdot) , s_n^{\frac{N}{2}} \partial_t u(s_n, s_n \cdot)\right ) \rightharpoonup (\tilde V^1_0, \tilde V^1_1)$$
in $\energysp$ as $n \rightarrow \infty$. Let $(\psi_0, \psi_1) \in C^\infty_0 \times C^\infty_0$ with supp$(\psi_0, \psi_1) \subset
\{ |x| > 1 + \eta \}$ for some $\eta > 0$.  By Cauchy-Schwarz and \eqref{globallinlem3}, we have that
\begin{align*}
\left | \int_{\R^N} \nabla \psi_0 \cdot \nabla \tilde V^1_1 + \psi_1 \tilde V^1_1 dx\right | \lesssim \lim_{n \rightarrow \infty}
\left (\int_{|x| \geq s_n + \eta s_n} |\nabla u(s_n)|^2 + |\partial_t u(s_n)|^2 dx \right )^{\frac{1}{2}} = 0.
\end{align*}
This proves the desired support property of $\tilde V^1$.
Define the first profile
\begin{align*}
(V^1_0(y), V^1_1(y)) = \psi_{\delta'}(y) (\tilde V^1_0(y), \tilde V^1_1(y))
\end{align*}
with parameters $\mu_{1,n} = s_n,$ $s_{1,n} = 0$.   For $\delta'$ sufficiently small, $\|(V^1_0,V^1_1)\|_{\energysp} < \epsilon$.
As in Step 2 of the proof of Lemma \eqref{channelofenergyI}, we have that for each fixed $j \geq 2$,
\begin{align}\label{step1b}
\varphi_{\delta'}\left ( \frac{x}{s_n} \right ) \left ( \tilde V^j_{L,n}(0,x) , \partial_t \tilde V^j_{L,n}(0,x) \right ) =
\varphi_{\delta'}(| \tau_j|) \left ( \tilde V^j_{L,n}(0,x) , \partial_t \tilde V^j_{L,n}(0,x) \right ) + o(1)
\end{align}
as $n \rightarrow \infty$ in $\energysp$. Indeed, after a change of variables,
\begin{align*}
\left \| \varphi_{\delta'}\left ( \frac{x}{s_n} \right ) \vec{\tilde V}^j_{L,n}(0,x) -
\varphi_{\delta'}(| \tau_j|) \vec{\tilde V}^j_{L,n}(0,x) \right \|_{\energysp} =
\left \| \varphi_{\delta'}\left ( x \right ) \vec{\check V}^j_{L,n}(0,x) -
\varphi_{\delta'}(| \tau_j|) \vec{\check V}^j_{L,n}(0,x) \right \|_{\energysp}
\end{align*}
where $\check V^j_L$ is just  $\tilde V^j_L$ with new parameters $\left \{ \frac{\tau_{j,n}}{s_n}, \frac{\lambda_{j,n}}{s_n} \right \}_{j,n}$.  We can then
reapply the argument from Step 2 of the proof of Lemma \ref{channelofenergyI} verbatim to obtain \eqref{step1b}.  Moreover, if $|\tau_j| \leq 1 - 2 \delta',$ then
$\varphi_{\delta'}\left ( \frac{x}{s_n} \right ) \left ( \tilde V^j_{L,n}(0) , \partial_t \tilde V^j_{L,n}(0) \right ) = o(1)$ in $\energysp$ as $n \rightarrow
\infty$.  Define new profiles
$$
V^j_L = \varphi_{\delta'}(|\tau_j|) \tilde V^j_L, \quad \mbox{with parameters} \quad
\{\mu_{j,n}, s_{j,n}\}_{j,n}.
$$
Then
\begin{align*}
\varphi_{\delta'}\left (\frac{x}{s_n} \right ) (u(s_n), \partial_t u(s_n)) &= \sum_{j = 1}^J
\varphi_{\delta'}\left ( \frac{x}{s_n} \right ) \left ( \tilde V^j_{L,n}(0) , \partial_t \tilde V^j_{L,n}(0) \right ) +
\varphi_{\delta'}\left (\frac{x}{s_n} \right ) (\tilde w_{0,n}^J, \tilde w_{1,n}^J) \\
&= \sum_{j = 1}^J
 \left ( V^j_{L,n}(0) , \partial_t V^j_{L,n}(0) \right ) \\&+
(\epsilon_{0,n}^J, \epsilon_{1,n}^J)  +
\varphi_{\delta'}\left (\frac{x}{s_n} \right )( \tilde w_{0,n}^J, \tilde w_{1,n}^J)
\end{align*}
where
$$
\lim_{n \rightarrow \infty} \|(\epsilon_{0,n}^J, \epsilon_{1,n}^J)\|_{\energysp} = 0.
$$
By Lemma \eqref{displem}
$$
\lim_{J \rightarrow \infty}
\limsup_{n \rightarrow \infty} \left \| S(t) \varphi_{\delta'}\left (\frac{x}{s_n} \right )( \tilde w_{0,n}^J, \tilde w_{1,n}^J)
\right \|_{S(\R)} = 0.
$$
Thus, we obtain that $\left \{ \varphi_{\delta'}\left ( \frac{x}{s_n} \right ) \vec u(s_n) \right \}$ has the desired profile
decomposition. This concludes Step 1.

\subsubsection*{Step 2} Let $u_n$ be the solution to \eqref{nlw} with initial data $\varphi_{\delta'} \left ( \frac{x}{s_n} \right )
\vec u(s_n)$.  We use Proposition \ref{approxthm} with $\theta_n = s_n / 2$.  Indeed, by construction, for $j \geq 2$,
$\lim_{n \rightarrow \infty} \frac{|s_{j,n}|}{\mu_{j,n}} = + \infty$ and $|\tau_j | \in [1 - 2\delta',1]$. So
\begin{align*}
\lim_{n \rightarrow \infty} \frac{-s_{j,n}}{\mu_{j,n}} = +\infty &\implies
\lim_{n \rightarrow \infty} \frac{s_n / 2 - s_{j,n}}{\mu_{j,n}} = + \infty, \\
\lim_{n \rightarrow \infty} \frac{-s_{j,n}}{\mu_{j,n}} = -\infty &\implies
\lim_{n \rightarrow \infty} \frac{s_n / 2 - s_{j,n}}{\mu_{j,n}} =
\lim_{n \rightarrow \infty} \frac{- s_{j,n}}{\mu_{j,n}} \left ( 1 - \frac{s_n}{2s_{j,n}} \right ) =
- \infty.
\end{align*}
Thus, we can apply Proposition \ref{approxthm} to obtain
$$
\vec u_n (s_n /2) = \sum_{j = 1}^J \vec V^j_n( s_n / 2) + \vec w^J_n(s_n /2) + \vec r_n^J(s_n /2 )
$$
where
$$
\lim_{J \rightarrow \infty} \limsup_{n \rightarrow \infty} \left [
\| r_n^J \|_{S(0, s_n/2)} +
\sup_{t \in [0,s_n/2]} \| \vec r^J_n(t) \|_{\energysp}
\right ] = 0.
$$
Let
$$
t_n = \frac{3}{2}s_n, \quad t_{j,n} = s_{j,n} - \frac{s_n}{2}, \quad \delta = \frac{\delta'}{3}.
$$
Since $\vec u_n(0,x) = \vec u(s_n,x)$ if $|x| \geq (1-\delta')s_n$, by finite speed of propagation
$$
|x| \geq (3/2 - \delta') s_n = (1-2\delta) t_n \implies \vec u_n(s_n / 2, x) = \vec u(t_n,x).
$$
Thus,
\begin{align*}
\varphi_{\delta} \left ( \frac{x}{t_n} \right ) \vec u(t_n) =& \varphi_{\delta} \left ( \frac{x}{t_n} \right ) \vec u(s_n/2) \\
=& \sum_{j = 1}^J \varphi_{\delta} \left ( \frac{x}{t_n} \right ) \vec V^j( s_n /2 ) +
\varphi_{\delta} \left ( \frac{x}{t_n} \right ) \vec w^J_n(s_n / 2) \\&+
\varphi_{\delta} \left ( \frac{x}{t_n} \right ) \vec r^J_n(s_n / 2).
\end{align*}
The conclusion of the lemma follows from a similar analysis as in Step 1.
\end{proof}
\begin{proof}[Proof of Proposition \ref{globallin}]  We divide the proof into two steps.
\subsubsection*{Step 1}  First, we show that for all $A \in \R$ there exists a radial solution $v_L^A$ of the linear wave equation \eqref{lw}
such that
\begin{align}
\lim_{t \rightarrow + \infty} \int_{|x| \geq t - A} | \nabla_{t,x} (u - v_L^A)(t,x)|^2 dx = 0. \label{linbehavior2}
\end{align}
Consider the sequence given $\{t_n\}_n$ given by Lemma \ref{globallinlem}, and let $u_n$ be the solution to \eqref{nlw} with initial data
$\varphi_\delta \left ( \frac{x}{t_n} \right) \vec u(t_n)$ at $t = 0$.   By
Proposition \ref{approxthm}, \eqref{globallinlem1}, and \eqref{globallinlem2}, it follows that for $n$ large $u_n$ is globally defined
and scatters for positive times.  Let $n$ be large and fixed and let $\tilde v_{L,n}$ be the solution to the linear equation \eqref{lw}
such that
$$
\lim_{t \rightarrow +\infty} \| \vec u_n(t) - \vec{\tilde v}_{L,n}(t) \|_{\energysp} = 0.
$$
By finite speed of propagation, $\vec u(t_n + t, x) = \vec u_n(t,x)$ for $|x| \geq (1-\delta) t_n + t,$ $t \geq 0$.  Thus,
$$
\lim_{t \rightarrow +\infty} \int_{|x| \geq -\delta t_n + t} |\nabla_{t,x}u(t,x) - \nabla \tilde v_{L,n} (t - t_n, x)|^2 = 0.
$$
By choosing $n$ large enough so that $-\delta t_n  \leq -A$, we get \eqref{linbehavior2} with $v_L^A(t,x) = \tilde v_{L,n}(t-t_n,x)$.

\subsubsection*{Step 2}  Let $\{t_n\}_n$ be a sequence tending to $+ \infty$ so that, after extraction, $S(-t_n)\vec u(t_n)$ has a weak
limit $(v_{0},v_{1})$ in $\energysp$ as $n \rightarrow \infty$.  After further extraction, we can assume that the sequence
$\vec u(t_n)$ has a profile decomposition
$$\vec u (t_n) = \vec v_L(t_n) + \sum_{j = 2}^J \vec{U}^j_{L,n}(0) + (w^J_{0,n},w^J_{1,n}).$$
Here we have chosen the first profile $U^1_L = v_L$ with parameters $\lambda_{1,n} = 1$, $t_{1,n} = -t_n$.
Let $A \in \R$, and let $v_L^A$ be the linear solution from Step 1.  Then $\vec u(t_n) - v_L^A(t_n)$ admits a profile decomposition
$$
\vec u(t_n) - \vec v_L^A(t_n) = \vec v_L(t_n) - \vec v_L^A(t_n) + \sum_{j = 2}^J \vec{U}^j_{L,n}(0) + (w^J_{0,n},w^J_{1,n})
$$
By the localized orthogonality of the profiles \eqref{locorth}, we get that \eqref{linbehavior2} implies that
$$
\lim_{n \rightarrow \infty} \int_{|x| \geq t_n - A} |\nabla_{t,x} (v_L - v_L^A)(t_n,x) |^2 dx = 0.
$$
Since $v_L^A - v_L$ is a solution to the linear wave equation, the previous line and the decay of the free energy outside of the light cone ${ |x| \geq t - A }$
implies that
$$
\lim_{t \rightarrow +\infty} \int_{|x| \geq t - A} |\nabla_{t,x} (v_L - v_L^A)(t,x) |^2 dx = 0.
$$
This together with \eqref{linbehavior2} yields \eqref{linbehavior}.

Let $\{\tilde t_n\}_n$ be another sequence
tending to $+\infty$ such that $S(-\tilde t_n)\vec u(\tilde t_n) \rightharpoonup (\tilde v_0, \tilde v_1)$ in $\energysp$.  The previous
argument implies that for all $A \in \R$
\begin{align*}
\lim_{t \rightarrow +\infty} \int_{|x| \geq t - A} |\nabla_{t,x} (v_L - \tilde v_L)(t,x) |^2 dx = 0.
\end{align*}
By Lemma \ref{loclem}, given $\epsilon > 0$ there exists $A > 0$ sufficiently large so that
\begin{align*}
\limsup_{t \rightarrow +\infty} \int_{|x| \leq t - A} |\nabla_{t,x} (v_L - \tilde v_L)(t,x) |^2 dx < \epsilon.
\end{align*}
By conservation of the free energy, these facts imply that
\begin{align*}
\int_{\R^N} |\nabla(v_0 - \tilde v_1)|^2 + |v_1 - \tilde v_1|^2 dx \leq \epsilon.
\end{align*}
Since $\epsilon$ was arbitrary, we obtain $(v_0,v_1) = (\tilde v_0, \tilde v_1)$.
\end{proof}

Before turning to the main result of this section, we first prove an analog of \eqref{alimit1} which will be used in a
later section.

\begin{clm}\label{globlim}
Let $u$ and $v_L$ be as in Proposition \ref{globallin}. Let $a(t) = u(t) - v_L(t)$.  Then
$\lim_{t \rightarrow +\infty} E(a(t), \partial_t a(t))$ exists and in fact
\begin{align}
\lim_{t \rightarrow +\infty} E(a(t), \partial_t a(t))
= E(u_0,u_1) - \frac{1}{2} \| (v_0,v_1) \|_{\energysp}^2.
\end{align}
\end{clm}

\begin{proof}
Since
$$S(-t)\vec u(t) \rightharpoonup \vec v_L(0) \quad \mbox{in } \energysp \mbox{ as } t \rightarrow +\infty,$$
conservation of the free linear energy implies
\begin{align*}
\int_{\R^N} |\nabla_{t,x} a(t)|^2 dx =& \int_{\R^N} |\nabla_{t,x} u(t)|^2 dx +
\| (v_0,v_1) \|^2_{\energysp} - 2 \int_{\R^N} \nabla_{t,x}S(-t)\vec u(t) \cdot (v_0,v_1) dx \\
=& \int_{\R^N} |\nabla_{t,x} u(t)|^2 dx -
\| (v_0,v_1) \|^2_{\energysp} + o(1),
\end{align*}
as $t \rightarrow +\infty$.

We claim that
\begin{align}\label{vanishingint}
\lim_{t \rightarrow \infty} \int_{\R^N} |v_L(t)|^{\frac{2N}{N-2}} dx = 0.
\end{align}
Since the linear propagator $S(t)$ is unitary on $\dot H \times L^2$, we may assume that $(v_0,v_1) \in C^\infty \times C^\infty$ and
$\supp (v_0,v_1) \subseteq \{ |x| \leq R \}$ for some $R > 0$.  By finite speed of propagation, supp
$\vec v_L(t) \subseteq \{ |x| \leq R + t\}$ for $t > 0$.  We also have the dispersive estimate
$$
\forall x \in\R^N, \forall t > 0, \quad |v_L(t,x)| \leq C t^{-\frac{N-1}{2}}.
$$
Hence
\begin{align*}
\int_{\R^N} |v_L(t)|^{\frac{2N}{N-2}} dx \leq& C t^{- \frac{N(N-1)}{N-2}} | \{ |x| \leq R + t \} | \\
=& C t^{- \frac{N(N-1)}{N-2}} ( R + t)^N \rightarrow 0 \quad \mbox{as } t \rightarrow +\infty
\end{align*}
as desired.

As in the finite time blow--up case,
\begin{align*}
\int_{\R^N} |u(t)|^{\frac{2N}{N-2}} dx &= \int_{\R^N} |v(t)|^{\frac{2N}{N-2}} dx + \int_{\R^N} |a(t)|^{\frac{2N}{N-2}} dx + 
\epsilon(t) \\
&= \int_{\R^N} |a(t)|^{\frac{2N}{N-2}} dx + \epsilon(t) + o(1),
\end{align*}
where
$$
\epsilon(t) \lesssim \int_{\R^N}\left ( |v(t)|^{\frac{N+2}{N-2}}|a(t)| + |v(t)||a(t)|^{\frac{N+2}{N-2}} \right )dx.
$$
By H\"older's inequality, Sobolev's inequality, and \eqref{vanishingint}, it easily follows that $\lim_{t
\rightarrow +\infty} \epsilon(t) = 0$.  This concludes the proof of the claim.
\end{proof}

\subsection{Exclusion of energy concentration in the self--similar region}
The goal of the subsection is to prove the following global analog of Proposition \ref{selfsimilarblowup}.

\begin{ppn}\label{selfsimilarblowup2}
Let $u$ and $v_L$ be as above.  Then
\begin{align*}
\forall c_0 \in (0,1), \quad \lim_{t \rightarrow +\infty}
\int_{|x| \geq c_0 t} \left ( |\nabla (u - v_L)(t,x)|^2 +
|\partial_t(u - v_L)(t,x)|^2 \right ) dx = 0.
\end{align*}
\end{ppn}

The proof of Proposition \ref{selfsimilarblowup2} will be very similar to the proof of 
Proposition \ref{selfsimilarblowup}.  As in the finite time blowup case, we argue by contradiction.  Suppose the conclusion of \eqref{selfsimilarblowup2} does not hold.
Then there exists a sequence of times $\{s_n\}_n$ and $c_0 > 0$ such that $s_n \rightarrow +\infty$ and for all $n$
\begin{align*}
\int_{|x| \geq c_0 s_n} ( |\nabla (u - v_L)(s_n,x)|^2 +
|\partial_t(u - v_L)(s_n,x)|^2  ) dx \geq c_0.
\end{align*}

Let $v$ be the unique solution to \eqref{nlw} such that $T_+(v) = +\infty$ and
$$
\lim_{t \rightarrow +\infty} \| \vec v_L(t) - \vec v(t) \|_{\energysp} = 0.
$$
Without loss of generality, we can assume that $v(t)$ is defined on $[0,+\infty)$.  Note that by our definition of 
$v$ and \eqref{linbehavior}, we have for all $A \in \R$
$$
\lim_{t \rightarrow +\infty} \int_{|x| \geq t - A} |\nabla_{t,x}(u - v)(t,x)|^2 dx = 0.
$$
Let
\begin{align*}
u_n(\tau,y) &= s_n^{\frac{N-2}{2}} u(s_n + s_n\tau, s_ny), \\
v_n(\tau,y) &= s_n^{\frac{N-2}{2}} v(s_n + s_n\tau, s_ny).
\end{align*}
for $\tau \in [-1,+\infty)$ and $y \in \R^N$.  Note that
\begin{align}\label{globalsupport}
\lim_{n \rightarrow \infty} \int_{|y| \geq 1} |\nabla_{\tau,y} (u_n - v_n)(0,y) |^2 dy =
\lim_{n \rightarrow \infty} \int_{|x| \geq s_n} |\nabla_{t,x} (u - v)(s_n,x) |^2 dx = 0.
\end{align}
After extraction, assume that $\vec u_n(0) - \vec v_n(0)$ admits a profile decomposition
with profiles $\{ U^j_L \}_j$ and parameters $\{ \lambda_{j,n}, \tau_{j,n} \}_{j,n}$. As usual we will assume either
\begin{align*}
\tau_{j,n} = 0 \mbox{ for all } n \geq 1,\mbox{ or } \lim_{n \rightarrow \infty} \frac{-\tau_{j,n}}{\lambda_{j,n}} = \pm \infty,
\end{align*}
and that the following limits exist in $[-\infty,+\infty]$:
\begin{align*}
\tilde \tau_j &= \lim_{n \rightarrow \infty} -\tau_{j,n}, \\
\tilde \lambda_j &= \lim_{n \rightarrow \infty} \lambda_{j,n}.
\end{align*}
\begin{lem}\label{taulim2}
Fix $j \geq 1$.  Then
\begin{align*}
\tilde \tau_j \in [-1,1],
\end{align*}
and the sequence $\{ \lambda_{j,n} \}_n$ is bounded.  Furthermore,
\begin{align}\label{globalorthparam}
\lim_{n \rightarrow \infty} s_n \lambda_{j,n} + \frac{1}{\lambda_{j,n}s_n} + |\tau_{j,n} + 1|s_n = +\infty.
\end{align}
\end{lem}

\begin{proof}
The claims that $\tau_j \in [-1,1]$ and that the sequence $\{ \lambda_{j,n} \}_n$ is bounded follow almost verbatim as in the finite
time blow--up case.  One simply uses \eqref{globalsupport} instead of \eqref{ansupp}. The last claim in Lemma \eqref{taulim2} follows
from the fact that $S(-s_n)(\vec u(s_n) - \vec v(s_n)) \rightharpoonup 0$ in $\energysp$ as $n \rightarrow \infty$.
\end{proof}

As in the proof of Proposition \ref{selfsimilarblowup}, by orthogonality of the parameters and \eqref{globalsupport} there
is at most one index $j \geq 1$ such that $\tilde \lambda_j > 0$.  Reordering the profiles, we will always assume that this index is 1, setting
$U^1 = 0$ if there is no index $j$ with the preceding property.  We can also assume that for all $n$, $\tau_{1,n} = 0$ and $\lambda_{j,n} =1$.
It follows from \eqref{globalsupport} that $\vec U^1(0)$ is supported in $\{ |y| \leq 1\}$.  Let 
$$
r_0 = \sup \{ \rho : \supp \vec U^1(0) \subseteq \{ |y| \leq \rho \} \}.
$$
As in the finite time blowup case, we consider three cases:
\begin{itemize}
\item Case 1: $r_0 > 0$ and for all $j \geq 2$, $|\tilde \tau_j | \leq r_0$.
\item Case 2: There exists $j \geq 2$ such that $|\tilde \tau_j | > r_0$.
\item Case 3: $r_0 = 0$ and for all $j \geq 2$, $\tilde \tau_j = 0$.
\end{itemize}
These three cases yield the following consequence.

\begin{lem}\label{channelofenergyII}
Let $I^+ = [1/2, +\infty)$, and let $I^- = [-1,-1/2]$.  Then there exists $\eta_1 > 0$ such that the following holds for large $n$;
for all sequences $\{\theta_n\}_n$ with $\theta_n \in I^+$, or for all sequences $\{\theta_n\}_n$ with $\theta_n \in I^-$
\begin{align*}
\int_{|y| \geq |\theta_n|} \left | \nabla_{\tau,y} (u_n - v_n)(\theta_n,y) \right |^2 dy \geq \eta_1.
\end{align*}
\end{lem}

\begin{proof}
The proof is a verbatim copy of the proof of Lemma \ref{channelofenergyI} and we omit it.
\end{proof}

Lemma \ref{channelofenergyII} contradicts the following claim which finishes the proof of Proposition \ref{selfsimilarblowup2}.

\begin{clm}\label{vanishingclaimII}
Let $u$ satisfy the assumptions of Proposition \ref{selfsimilarblowup2}.  There exists sequences $\{\theta_n^{\pm}\}_n \in I^{\pm}$ such that
\begin{align*}
\lim_{n \rightarrow \infty} \int_{|y| \geq |\theta_n^{\pm}|}
\left |\nabla_{\tau_,y} (u_n - v_n)(\theta^{\pm}_n,y) \right |^2 dy = 0.
\end{align*}
\end{clm}

\begin{proof}[Proof of Claim \ref{vanishingclaimII}]
We go back to the original variables and write
\begin{align}\label{vanishingclaimIIeqI}
\int_{|y| \geq |\theta_n^\pm|} \left |\nabla_{\tau_,y} (u_n - v_n)(\theta^{\pm}_n,y) \right |^2 dy = \int_{|x| \geq s_n |\theta_n^\pm|}
\left |\nabla_{t,x} (u - v)(s_n + s_n \theta^{\pm}_n,x) \right |^2 dx.
\end{align}
Taking $\theta_n^- = -1$ implies that right hand side of \eqref{vanishingclaimIIeqI} goes to 0 as $n \rightarrow \infty$.  Let
$\sigma_n = s_n + \theta_n^+ s_n$.  Then
\begin{align*}
 \int_{|x| \geq s_n |\theta_n^\pm|}
\left |\nabla_{t,x} (u - v)(s_n + s_n \theta^{\pm}_n,x) \right |^2 dx =
 \int_{|x| \geq \sigma_n - s_n}
\left |\nabla_{t,x} (u - v)(\sigma_n,x) \right |^2 dx.
\end{align*}
For each $n$, we have that
$$
\lim_{t \rightarrow +\infty}  \int_{|x| \geq t - s_n}
\left |\nabla_{t,x} (u - v)(t,x) \right |^2 dx = 0.
$$
Thus, we can choose $\sigma_n \geq \frac{3}{2}s_n$ (and hence $\theta_n^+ \geq \frac{1}{2})$ so that
$$
\lim_{n \rightarrow \infty} \int_{|y| \geq |\theta_n^{+}|}
\left |\nabla_{\tau,y} (u_n - v_n)(\theta^{+}_n,y) \right |^2 dy =
\lim_{n \rightarrow \infty} \int_{|x| \geq \sigma_n - s_n}
\left |\nabla_{t,x} (u - v)(\sigma_n,x) \right |^2 dx = 0.
$$
\end{proof}

\section{Expansion as a sum of rescaled solitons}

In this section, we prove Theorem \ref{thm1} and Theorem \ref{thm2}.

\subsection{Expansion up to a dispersive error}

In this subsection, we show that Proposition \ref{selfsimilarblowup} (respectively Proposition \ref{selfsimilarblowup2}) implies:

\begin{cor}\label{thm1alm} Let $N \geq 5$ be odd.  Let $u$ be a radial solution to \eqref{nlw} with $T_+(u) < +\infty$ that satisfies \eqref{bdd}.
Then there exists $(v_0,v_1) \in \energysp$, a sequence $t_n \rightarrow T_+(u)$, an integer $J_0 \geq 1$, $J_0$ sequences of positive numbers $\{\lambda_{j,n}\}_n$,
 $j = 1, \ldots, J_0$, $J_0$ signs $\iota_j \in \{ \pm 1 \}$, and a sequence $\{w_{0,n}\}_n$ in $\dot H^1$ such that
 \begin{align*}
\lambda_{1,n} \ll \lambda_{2,n} \ll \cdots \ll \lambda_{J_0,n} \ll T_+(u) - t_n,
\end{align*}
and
\begin{align*}
\vec u(t_n) = (v_0,v_1) + \sum_{j = 1}^{J_0} \left (\frac{\iota_j}{\lambda_{j,n}^{\frac{N-2}{2}}} W \left (\frac{x}{\lambda_{j,n}} \right ),0 \right ) + (w_{0,n},0) + o(1)
\quad \mbox{in } \energysp \mbox{ as } n \rightarrow \infty.
\end{align*}
where
\begin{align*}
\lim_{n \rightarrow \infty} \| S(t)(w_{0,n},0) \|_{S(\R)} = 0.
\end{align*}
\end{cor}

\begin{cor}\label{thm2alm} Let $N \geq 5$ be odd.  Let $u$ be a radial solution to \eqref{nlw} with $T_+(u) = +\infty$ that satisfies \eqref{bdd}.
Then there exists a solution $v_L$ to the linear wave equation \eqref{lw}, a sequence $t_n \rightarrow +\infty$, an integer $J_0 \geq 0$, $J_0$ sequences of positive numbers $\{\lambda_{j,n}\}_n$,
 $j = 1, \ldots, J_0$, $J_0$ signs $\iota_j \in \{ \pm 1 \}$, and a sequence $\{w_{0,n}\}_n$ in $\dot H^1$ such that
\begin{align*}
\lambda_{1,n} \ll \lambda_{2,n} \ll \cdots \ll \lambda_{J_0,n} \ll t_n,
\end{align*}
and
\begin{align*}
\vec u(t_n) = \vec v_L(t_n) + \sum_{j = 1}^{J_0} \left (\frac{\iota_j}{\lambda_{j,n}^{\frac{N-2}{2}}} W \left (\frac{x}{\lambda_{j,n}} \right ),0 \right ) + (w_{0,n},0) + o(1)
\quad \mbox{in } \energysp \mbox{ as } n \rightarrow \infty.
\end{align*}
where
\begin{align*}
\lim_{n \rightarrow \infty} \| S(t)(w_{0,n},0) \|_{S(\R)} = 0.
\end{align*}
\end{cor}

Corrollary \ref{thm1alm} follows from Proposition 5.1 in \cite{dkm1} and Proposition \ref{selfsimilarblowup}.  Corrollary \ref{selfsimilarblowup2}
follows from the proof of Corrollary 4.2 in \cite{dkm3} and Proposition \ref{selfsimilarblowup2}.  The main idea of the proof in the
finite time blow--up case, say, is that
the vanishing of the energy in self--similar regions implies that there is a sequence of times $\{t_n\}_n$ approaching the maximal forward time
of existence so that the time derivative $\partial_t(u -v)(t_n)$ goes to 0.  Taking a profile decomposition along this sequence of times $\{t_n\}_n$ and using
Proposition \ref{approxthm}, one deduces that the nonlinear profiles for this profile decomposition satisfy the elliptic
equation $-\Delta U = |U|^{4/(N-2)}U$.  Hence, each profile must be $W$ up to sign change and scaling.

\subsection{Expansion in the energy space}

In this subsection, we conclude the proofs of Theorem \ref{thm1} and Theorem \ref{thm2}.

\subsubsection{Finite time blow--up solutions} Assume that $u$ satisfies the assumptions of Theorem \ref{thm1} (with $T_+ = 1$).
Let $w_{0,n}$ be as in Corrollary \ref{thm1alm}.  Theorem \ref{thm1} will follow from Corrollary \ref{thm1alm} and
\begin{align}\label{strengthenI}
\lim_{n \rightarrow \infty} \| w_{0,n} \|_{\dot H^1} = 0.
\end{align}

This will follow in a similar fashion as in the case $N = 3$ (see Section 4.2 of \cite{dkm3}) and $N = 4$
(see Proposition 5.4 in \cite{ckls}) using channels of energy.
We will deduce \ref{strengthenI} by using the following technical lemma which states that in exterior regions we can remove each soliton arising in
the above decomposition by adjusting the sequence of times $\{t_n\}_n$.

\begin{lem}\label{cutoffI}
Let $u$ be as in Corrollary \ref{thm1alm}.  For all $k = 1, \ldots, J_0 + 1$, there exists (after extraction of subsequences in
$n$), a sequence $\{t_n^k \}_n$ such that $t_n \leq t_n^k < 1$ and a sequence $\left \{ \left (u^k_{0,n}, u^k_{1,n} \right )
\right \}_n$ in $\energysp$
such that for large $n$
\begin{align}
\forall x \in \R^N, |x| \geq t^k_n - t_n \implies \left ( u \left (t^k_n, x \right )
, \partial_t u \left (t^k_n, x \right ) \right ) =
\left ( u^k_{0,n} \left ( x \right ), u^k_{1,n} \left (x \right )\right ),  \label{cutoffIeq1}
\end{align}
and
\begin{align}
\left ( u^k_{0,n}, u^k_{1,n} \right ) = &\left ( v \left (t^k_n \right )
, \partial_t v \left (t^k_n \right ) \right ) + \sum_{j = k}^{J_0} \left (\frac{\iota_j}{\lambda_{j,n}^{\frac{N-2}{2}}}
W \left (\frac{x}{\lambda_{j,n}} \right ),0 \right )\\ &+ \left (w_n(t^k_n - t_n), \partial_t w_n(t^k_n - t_n)\right ) + o(1),
\end{align}
in $\energysp$ as $n \rightarrow \infty$.
\end{lem}

The statement and proof of Lemma \ref{cutoffI} is identical to that in \cite{dkm3} (see Lemma 4.5) and we omit it.  

\begin{proof}[Proof of Theorem \ref{thm1}]
By Lemma \ref{cutoffI} with $k = J_0 + 1$, we have
\begin{align}
\left ( u^{J_0+1}_{0,n}, u^{J_0 + 1}_{1,n} \right ) = \left ( v \left (t^{J_0+1}_n \right )
, \partial_t v \left (t^{J_0+1}_n \right ) \right ) + \left (w_n(t^{J_0+1}_n - t_n), \partial_t w_n(t^{J_0+1}_n - t_n) \right ) + o(1),
\end{align}
in $\energysp$ as $n \rightarrow \infty$ where $t_n \leq t^{J_0+1}_n < 1$.  Denote by $u^{J_0 + 1}_n$ the solution of
\eqref{nlw} with initial data $\left (u^{J_0+1}_{0,n}, u^{J_0+1}_{1,n}\right )$ at $t = 0$.  By Proposition \ref{approxthm}, for $t \in [0,t_0]$
($t_0 > 0$ small enough so that $1+t_0$ is in the domain of $v$),
\begin{align}\label{expn1}
\vec u_n^{J_0 + 1}(t) = \vec v \left ( t^{J_0+1}_n + t \right ) + \vec w_n \left (t^{J_0+1}_n + t - t_n \right ) + \vec r_n(t),
\end{align}
where
$$
\lim_{n \rightarrow \infty} \sup_{t \in [0,t_0]} \| \vec r_n(t) \|_{\energysp} = 0.
$$
By finite speed of propagation and \eqref{cutoffIeq1},
\begin{align}\label{finspeedI}
\left ( u_n^{J_0 + 1}(t,x) , \partial_t u_n^{J_0+1}(t,x) \right ) =
\left ( u\left (t_n^{J_0 + 1} + t,x\right ) , \partial_t u\left (t_n^{J_0+1}+t,x\right ) \right )
\end{align}
for $0 < t < 1 - t_n^{J_0 + 1}$, $|x| \geq t + t_n^{J_0 + 1} - t_n$.

Since $\partial_t w_n(0) = 0$, $w_n(-t) = w_n(t)$.  Thus, by Proposition \ref{extbds}, for all $t \in \R$,
\begin{align}\label{lowbdwn1}
\int_{|x| \geq |t|} |\nabla w_n(t)|^2 + |\partial_t w_n(t)|^2 dx \geq \frac{1}{2} \int_{\R^N}
|\nabla w_{0,n}|^2 + |w_{1,n}|^2 dx.
\end{align}

By \eqref{expn1}, \eqref{finspeedI}, and \eqref{lowbdwn1} with $t = \frac{1-t_n^{J_0+1}}{2}$, we get
\begin{align}
\int_{|x| \geq \frac{1+t_n^{J_0+1}}{2} - t_n} &\left | \nabla a \left ( \frac{1 + t_n^{J_0+1}}{2}\right ) \right |^2
+ \left | \partial_t a \left ( \frac{1 + t_n^{J_0+1}}{2}\right ) \right |^2 dx  \\
&\geq \frac{1}{2} \int_{\R^N}
|\nabla w_{0,n}|^2 + |w_{1,n}|^2 dx + o(1).
\end{align}
As $t_n \leq t_n^{J_0 + 1}$, we have
\begin{align*}
\frac{1 + t_n^{J_0+1}}{2} - t_n = \frac{1 - t_n^{J_0+1}}{2} + t_n^{J_0+1} - t_n \geq 1 - \frac{1 + t_n^{J_0+1}}{2}.
\end{align*}
Since supp $\vec a(t) \subseteq \{|x| \leq 1 - t \}$,
\begin{align*}
\int_{|x| \geq \frac{1+t_n^{J_0+1}}{2} - t_n} &\left | \nabla a \left ( \frac{1 + t_n^{J_0+1}}{2}\right ) \right |^2
+ \left | \partial_t a \left ( \frac{1 + t_n^{J_0+1}}{2}\right ) \right |^2 dx = 0,
\end{align*}
which shows
$$
\lim_{n \rightarrow \infty} \int_{\R^N}
|\nabla w_{0,n}|^2 + |w_{1,n}|^2 dx = 0.
$$

It remains to show
$$
\lim_{t \rightarrow 1^+} \int_{|x| \leq 1-t}\frac{1}{2}|\nabla u(t)|^2 + \frac{1}{2}|\partial_t u(t)|^2 -
\frac{N-2}{2N}| u(t)|^{\frac{2N}{N-2}} dx = J_0 E(W,0).
$$
Let $\{ t_n \}_n$ be as in the statement of Theorem \ref{thm1}.  Then since $v(t)$ is regular up to $t = 1$ and
$a(t) = u(t) - v(t)$ is supported in $B_{1-t}$, we have
\begin{align*}
\int_{|x| \leq 1-t}\frac{1}{2}|\nabla u(t)|^2 + \frac{1}{2}|\partial_t u(t)|^2 -
\frac{N-2}{2N}|u(t)|^{\frac{2N}{N-2}} dx = E(a(t),\partial_t a(t)) + o(1)
\end{align*}
as $t \rightarrow 1^+$.  Hence
\begin{align*}
\lim_{t \rightarrow 1^+} \int_{|x| \leq 1-t}\frac{1}{2}|\nabla u(t)|^2 + \frac{1}{2}|\partial_t u(t)|^2 -
\frac{N-2}{2N}| u(t)|^{\frac{2N}{N-2}} dx &= \lim_{t \rightarrow 1^+} E(a(t),\partial_t a(t)) \\
&= \lim_{n \rightarrow \infty} E(a(t_n),\partial_t a(t_n)) \\
&= J_0 E(W,0).
\end{align*}
This completes the proof of Theorem \ref{thm1}.
\end{proof}

\subsubsection{Global solutions} Assume that $u$ satisfies the assumptions of Theorem \ref{thm2}.
Let $w_{0,n}$ be as in Corrollary \ref{thm2alm}.  Theorem \ref{thm2} will follow from Corrollary \ref{thm2alm} and
\begin{align}\label{strengthenII}
\lim_{n \rightarrow \infty} \| w_{0,n} \|_{\dot H^1} = 0.
\end{align}

This will follow in a similar fashion as the finite time blow--up case.  The technical tool analogous to Lemma \ref{cutoffI}
is the following lemma.

\begin{lem}\label{cutoffII}
Let $u$ be as in Corollary \ref{thm2alm}.  For all $k = 1, \ldots, J_0 + 1$, there exists (after extraction of subsequences in
$n$), a sequence $\{t_n^k \}_n$ such that $t_n \leq t_n^k$ and a sequence $\left \{\left (u^k_{0,n}, u^k_{1,n} \right )
\right \}_n$ in $\energysp$
such that for large $n$
\begin{align}
\forall x \in \R^N, |x| \geq t^k_n - t_n \implies \left ( u \left (t^k_n, x \right )
, \partial_t u \left (t^k_n, x \right ) \right ) =
\left ( u^k_{0,n} \left ( x \right ), u^k_{1,n} \left (x \right )\right ),  \label{cutoffIeq1}
\end{align}
and
\begin{align}
\left ( u^k_{0,n}, u^k_{1,n} \right ) = &\left ( v_{L} \left (t^k_n \right )
, \partial_t v_{L} \left (t^k_n \right ) \right ) + \sum_{j = k}^{J_0} \left (\frac{\iota_j}{\lambda_{j,n}^{\frac{N-2}{2}}}
W \left (\frac{x}{\lambda_{j,n}} \right ),0 \right )\\ &+ \left (w_n(t^k_n - t_n), \partial_t w_n(t^k_n - t_n)\right ) + o(1),
\end{align}
in $\energysp$ as $n \rightarrow \infty$.
\end{lem}

\begin{proof}[Proof of Theorem \ref{thm2}]
By Lemma \ref{cutoffII} with $k = J_0 + 1$, we have
\begin{align}
\left ( u^{J_0+1}_{0,n}, u^{J_0 + 1}_{1,n} \right ) = \left ( v \left (t^{J_0+1}_n \right )
, \partial_t v \left (t^{J_0+1}_n \right ) \right ) + \left (w_n(t^{J_0+1}_n - t_n), \partial_t w_n(t^{J_0+1}_n - t_n) \right ) + o(1),
\end{align}
in $\energysp$ as $n \rightarrow \infty$ where $t_n \leq t^{J_0+1}_n$.  Denote by $u^{J_0 + 1}_n$ the solution of
\eqref{nlw} with initial data $\left (u^{J_0+1}_{0,n}, u^{J_0+1}_{1,n} \right ) $ at $t = 0$.  By Proposition \ref{approxthm}, for all $t \geq 0$,
\begin{align}\label{expn2}
\vec u_n^{J_0 + 1}(t) = \vec v_L \left ( t^{J_0+1}_n + t \right ) + \vec w_n \left (t^{J_0+1}_n + t - t_n \right ) + \vec r_n(t),
\end{align}
where
$$
\lim_{n \rightarrow \infty} \sup_{t \in [0,\infty)} \| \vec r_n(t) \|_{\energysp} = 0.
$$
By finite speed of propagation and \eqref{cutoffIeq1},
\begin{align}\label{finspeedII}
\left ( u_n^{J_0 + 1}(t,x) , \partial_t u_n^{J_0+1}(t,x) \right ) =
\left ( u\left (t_n^{J_0 + 1} + t,x\right ) , \partial_t u\left (t_n^{J_0+1}+t,x\right ) \right )
\end{align}
for $t \geq 0$, $|x| \geq t + t_n^{J_0 + 1} - t_n$.

Since $\partial_t w_n(0) = 0$, $w_n(-t) = w_n(t)$.  Thus, by Proposition \ref{extbds}, for all $t \in \R$,
\begin{align}\label{lowbdwn2}
\int_{|x| \geq |t|} |\nabla w_n(t)|^2 + |\partial_t w_n(t)|^2 dx \geq \frac{1}{2} \int_{\R^N}
|\nabla w_{0,n}|^2 + |w_{1,n}|^2 dx.
\end{align}

By \eqref{expn2}, \eqref{finspeedII}, and \eqref{lowbdwn2}, we have for any sequence of times $s_n \geq t_n^{J_0+1}$,
\begin{align*}
\int_{|x| \geq s_n - t_n} &|\nabla ( u - v_L)(s_n) |^2 + |\partial_t
 ( u - v_L)(s_n) |^2dx
\geq \frac{1}{2} \int_{\R^N}
|\nabla w_{0,n}|^2 +  w_{1,n}|^2 dx + o(1).
\end{align*}
Since, for any fixed $n$,
$$
\lim_{s \rightarrow \infty} \int_{|x| \geq s - t_n} |\nabla (u - v_L)(s) |^2 + |\partial_t (u - v_L)(s) |^2 dx = 0,
$$
we can find a sequence $s_n \geq t_n^{J_0+1}$ such that
$$
\lim_{n \rightarrow \infty} \int_{|x| \geq s_n - t_n} |\nabla ( u - v_L)(s_n) |^2 + |\partial_t
 ( u - v_L)(s_n) |^2 dx = 0.
$$
This shows
$$
\lim_{n \rightarrow \infty} \int_{\R^N}
|\nabla w_{0,n}|^2 + | w_{1,n}|^2 dx = 0
$$
as desired.

It remains to show that for all $A > 0$,
$$
\lim_{t \rightarrow +\infty} \int_{|x| \leq t - A}\frac{1}{2}|\nabla u(t)|^2 + \frac{1}{2}|\partial_t u(t)|^2 -
\frac{N-2}{2N}| u(t)|^{\frac{2N}{N-2}} dx
$$
exists, and
$$
\lim_{A \rightarrow +\infty} \lim_{t \rightarrow +\infty} \int_{|x| \leq t - A}\frac{1}{2}|\nabla u(t)|^2 + \frac{1}{2}|\partial_t u(t)|^2 -
\frac{N-2}{2N}| u(t)|^{\frac{2N}{N-2}} dx = J_0 E(W,0).
$$
To lessen notation, let
$$
e(u(t)) = \frac{1}{2}|\nabla u(t)|^2 + \frac{1}{2}|\partial_t u(t)|^2  - \frac{N-2}{2N} |u(t)|^{\frac{2N}{N-2}}.
$$
By Claim \ref{globlim}, and Proposition \ref{selfsimilarblowup2}, we have that
\begin{align*}
\lim_{t \rightarrow +\infty} \int_{\R^N} | v_L(t) |^{\frac{2N}{N-2}} dx &= 0,\\
\lim_{t \rightarrow +\infty} \int_{|x| \geq t / 2} |\nabla (u - v_L)(t) |^2 + |\partial_t (u - v_L)(t) |^2 dx &= 0.
\end{align*}
This along with conservation of the free energy imply that as $t \rightarrow +\infty$,
\begin{align*}
\int_{|x| \leq t - A} e(u(t)) dx &= E(u_0,u_1) - \int_{|x| \geq t - A} e(u(t)) dx \\
&= E(u_0,u_1) - \int_{|x| \geq t - A} e(v(t)) dx + o(1) \\
&= E(u_0,u_1) - \frac{1}{2} \int_{|x| \geq t - A} |\nabla_{t,x}v(t)|^2 dx + o(1) \\
&= E(u_0,u_1) - \frac{1}{2} \| (v_0,v_1) \|_{\energysp}^2 + \frac{1}{2} \int_{|x| \leq t - A}  |\nabla_{t,x}v(t)|^2 dx + o(1).
\end{align*}
By the monononicity of the free energy in exterior light cones, the limit
$$
\lim_{t \rightarrow +\infty} \int_{|x| \leq t - A}  |\nabla_{t,x}v(t)|^2 dx
$$
exists.  Hence $\lim_{t \rightarrow +\infty} \int_{|x| \leq t - A} e(u(t)) dx$ exists.  By Lemma \ref{loclem} and Claim \ref{globlim}
\begin{align*}
\lim_{A \rightarrow +\infty} \lim_{t \rightarrow +\infty} \int_{|x| \leq t - A} |\nabla_{t,x} v_L(t)|^2 dx &= 0,\\
\lim_{t \rightarrow +\infty} E(u(t) - v_L(t), \partial_t u(t) - \partial_t v_L(t)) &= E(u_0,u_1) - \frac{1}{2} \| (v_0,v_1) \|_{\energysp}^2.
\end{align*}
Thus, if $\{t_n\}_n$ is from the Theorem \ref{thm2}, then
\begin{align*}
\lim_{A \rightarrow +\infty} \lim_{t \rightarrow +\infty}\int_{|x| \leq t - A} e(u(t)) dx &= E(u_0,u_1) - \frac{1}{2} \| (v_0,v_1) \|_{\energysp}^2 \\
&= \lim_{t \rightarrow +\infty} E(u(t) - v_L(t), \partial_t u(t) - \partial_t v_L(t) ) \\
&= \lim_{n \rightarrow \infty} E(u(t_n) - v_L(t_n), \partial_t u(t_n) - \partial_t v_L(t_n) ) \\
&= J_0 E(W,0).
\end{align*}
This completes the proof of Theorem \ref{thm2}.
\end{proof}

\section{Type II solutions below $2\| \nabla W\|_{L^2}^2$}

\subsection{Blow--up solutions}

This subsection is devoted to proving Theorem \ref{below1}.  We assume that $u$ is a type II solution with $T^+(u) = 1$ and
\begin{align}\label{belowassump1}
\sup_{0 \leq t < 1} \int_{|x| \leq 1-t} | \nabla u(t) |^2 dx < 2 \| \nabla W \|^2_{L^2}.
\end{align}

We denote the regular part of $u(t)$ by $v(t)$ and the singular part
$a(t) = u(t) - v(t)$ as defined in Section 3.  Recall that supp $a(t) \subset \{|x| \leq 1-t \}$, $\vec a(t) \rightharpoonup 0$ in $\energysp$ as $t \rightarrow 1^-$, and
$$
\lim_{t \rightarrow 1^-} E(a(t),\partial_t a(t)) = E(u_0,u_1) - E(v_0,v_1).
$$
By Theorem \ref{thm1}, we know that there exists a sequence of times $t_n \rightarrow 1^-$, an integer $J_0 \geq 1$, $J_0$ scales $\lambda_{j,n}$ and signs $\iota_j$ for $1 \leq j \leq J_0$ so that
\begin{align}
\lambda_{1,n} \ll \lambda_{2,n} \ll \cdots \ll \lambda_{J_0,n} \ll 1 - t_n,
\end{align}
and
\begin{align}
\vec a(t_n) =  \sum_{j = 1}^{J_0} \left (\frac{\iota_j}{\lambda_{j,n}^{\frac{N-2}{2}}} W \left (\frac{x}{\lambda_{j,n}} \right ),0 \right ) + o(1)
\quad \mbox{in } \energysp \mbox{ as } n \rightarrow \infty.
\end{align}
The assumption \eqref{belowassump1} and the definition of $a(t)$ imply that
\begin{align}
\limsup_{t \rightarrow 1^-} \| \vec a(t) \|_{\energysp}^2
< 2 \| \nabla W \|^2_{L^2}.
\end{align}
By orthogonality of the scales $\lambda_{j,n}$, there can be only one soliton above, i.e., $J_0 = 1$.  Moreover, by replacing $u$ with $-u$ if necessary, we can assume that $\iota = 1$, so that
\begin{align}\label{belowdecomp1}
\vec a(t_n) = \left (\frac{1}{\lambda_{n}^{\frac{N-2}{2}}} W \left (\frac{x}{\lambda_{n}} \right ),0 \right ) + o(1) \quad
\mbox{in } \energysp \mbox{ as } n \rightarrow \infty,
\end{align}
with
$$
\lambda_n \ll 1 - t_n.
$$
Moreover, this implies that
\begin{align}\label{belowdecomp2}
E(u_0,u_1) - E(v_0,v_1) = \lim_{t \rightarrow 1^-} E(a(t),\partial_t a(t)) = E(W,0).
\end{align}

The proof will continue in a similar manner as in \cite{ckls}.  We divide the proof of Theorem \ref{below1} into several steps.

\subsubsection*{Step 1: Preliminary observations on a profile decomposition.}  Let $\tau_n \rightarrow 1^-$.
We wish to show that (after passing to a subsequence if necessary) the decomposition \eqref{belowdecomp1}
holds for the sequence $\{\tau_n\}_n$.  After passing to a subsequence, we obtain a profile decomposition
\begin{align}\label{profdec1}
\vec a(\tau_n) = \sum_{j = 1}^J \vec U^j_{L,n}(0) + \vec w^J_n(0).
\end{align}
By Lemma \ref{proford}, we may assume that the profiles are preordered as in Definition \eqref{proforddef} with
$$
i \leq j \implies \{ \vec U^i_L, \lambda_{i,n}, t_{i,n} \} \preccurlyeq \{ \vec U^j_L, \lambda_{j,n}, t_{j,n} \}.
$$
We may also view \eqref{profdec1} as profile decomposition for $\{ \vec u(\tau_n) \}_n$ since $(v_0, v_1)$ is the weak limit
of $\vec u(t)$ in $\energysp$ as $t \rightarrow 1$.  Indeed we can view $\vec v(\tau_n)$, up to an error $o(1)$ in $\energysp$,
as a profile $\vec U^0_L$ with initial data $(v_0,v_1)$ and parameters $\lambda_{n,0} = 1, t_{n,0} = 0$, and nonlinear profile
$v(t)$ and we write
\begin{align}\label{profdec2}
\vec u(\tau_n) = \vec v(\tau_n) + \sum_{j = 1}^J \vec U^j_{L,n}(0) + \vec w^J_n(0).
\end{align}
By \ref{bddlem} and the support properties of $a(t)$, we must have $|t_{j,n}| + \lambda_{j,n} \leq C_j (1 - \tau_n)$.

Since $u(t)$ breaks down at $t = 1$ and $v(t)$ is regular up to $t = 1$,  Proposition \ref{approxthm} implies that at least
one of the nonlinear profiles $U^j$ with $j \geq 1$ does not scatter forward in time.  By our pre--ordering this means that the
nonlinear profile $U^1$ does not scatter forward in time.   We claim that
\begin{align}\label{profdec3}
\{ U^1_L, \lambda_{1,n}, t_{1,n} \} \prec \{ U^0_L, 1, 0\}.
\end{align}
If \eqref{profdec3} does not hold, then
\begin{align}\label{profdec4}
\forall T < T_+(v_0,v_1) \implies \lim_{n \rightarrow \infty} \frac{T - t_{1,n}}{\lambda_{1,n}} < T_+(U^1),
\end{align}
where $T^+(v_0,v_1)$ is computed from the evolution starting at $t = 1$.  Since $v(t)$ exists in a neighborhood of $t = 1$, choose
$T > 0$ with $T < T_+(v_0,v_1)$.  We know that $|t_{1,n}| + \lambda_{1,n} \leq C (1- \tau_n)$.  This means that
$$
\lim_{n \rightarrow \infty} \frac{T-t_{1,n}}{\lambda_{1,n}} = +\infty,
$$
which contradicts \eqref{profdec4}.  Thus, \eqref{profdec3} holds.

By orthogonality of the $\dot H^1$ norm in our profile decomposition, we must have for all $j \geq 1$ and for all $J \geq 1$,
\begin{align}
 \limsup_{n \rightarrow \infty} \| \nabla U^j_{L,n}(0) \|_{L^2}^2 &< 2 \| \grad W \|_{L^2}^2, \label{profdec5} \\
\limsup_{n \rightarrow \infty} \| \nabla w^J_{0,n} \|_{L^2}^2 &< 2 \| \grad W \|_{L^2}^2. \label{profdec5b}
\end{align}
If $N \geq 4$, then
$$
2 \leq \left ( \frac{N}{N-2} \right )^{\frac{N-2}{2}}.
$$
Indeed, by elementary calculus the function $ (x-2) \log \left ( \frac{x}{x-2} \right )$ is strictly increasing on $[4,\infty)$.  This fact along with
\eqref{profdec5}, \eqref{profdec5b}, and Lemma \ref{energytrap} imply that for all $j, J\geq 1$ for all $n$ sufficiently large,
\begin{align*}
E(\vec U^j_{L,n}(0) ) \geq 0, \quad E(\vec w^J_n(0)) \geq 0,
\end{align*}
as well as
\begin{align*}
\lim_{n \rightarrow \infty} E(\vec U^j_{L,n}(0) ) = 0 &\implies \lim_{n \rightarrow \infty}
\| \vec U^j_{L,n}(0) \|_{\energysp} = 0, \\
\lim_{n \rightarrow \infty} E(\vec w^J_n(0) ) = 0 &\implies \lim_{n \rightarrow \infty}
\| \vec w^J_n(0) \|_{\energysp} = 0.
\end{align*}
These imply that for the nonlinear profile $U^j$, either $E(\vec U^j) > 0$ or $U^j \equiv 0$.

\subsubsection*{Step 2: Minimization process and consequences}
Here we use the general compactness argument for profile decompositions of $\vec u(\tau_n)$ developed in \cite{dkm6}.  We first
introduce some notation from \cite{dkm6}.

Let $\mathcal S_0$ denote the set of sequences $\{\tau_n\}_n$ with $\tau_n \rightarrow 1$
so that $\vec u(\tau_n)$ admits a pre--ordered profile decomposition. Let $\cl T = \{\tau_n\}_n \in \cl S_0$, and let
\begin{align}\label{minproc1}
J_0(\cl T) = \# \mbox{ of profiles of $\vec u(\tau_n)$ that do not scatter forward in time.}
\end{align}

Since $u(t)$ breaks down at time $t = 1$, we must have $J_0(\cl T) \geq 1$ for any $\cl T \in \cl S_0$.  However, by the small data theory
there exists $\delta_0 > 0$ so that if $\| \vec U^j_L \|_{\energysp} < \delta_0,$ then $U^j$ is global and scatters in both
time directions. In particular, if $U^j$ does not scatter forward in time, then $\| \vec U^j_L \|_{\energysp}^2 \geq \delta_0^2$. Since we are assuming $u$ is type II,
$$
\sup_{0 \leq t < 1} \| (u(t),\partial_t u(t)) \|_{\energysp} = M < \infty.
$$
By orthogonality of the free energy, if $J \in \N$ and $J \leq J_0(\cl T)$, then
$$
J \delta_0^2 \leq \limsup_{n \rightarrow \infty} \sum_{j = 1}^J \| \vec U^j_{L,n}(0) \|_{\energysp}^2 \leq
\limsup_{n \rightarrow \infty} \| (u(\tau_n),\partial_t u(\tau_n)) \|_{\energysp}^2 \leq M^2.
$$
Thus, $J_0(\cl T) \leq M^2 \delta_0^{-2}$ for all $\cl T \in \cl S_0$.

Define
$$J_1 (\cl T) = \min \{ j \geq 1 : j \prec j + 1 \}.$$
Here $\prec$ is the strict order from Definition \eqref{proforddef}.  By our definition of $J_0( \cl T)$, we must have
$J_1 (\cl T) \leq J_0 (\cl T)$.  Hence, $J_1 ( \cl T)$ is uniformly bounded on $\cl S_0$ as well.

Now define
\begin{align}\label{minproc2}
J_M &= \max \{ J_0(\cl T) : \cl T \in \cl S_0 \}, \\
\cl S_1 &= \{ \cl T \in \cl S_0 : J_0 (\cl T) = J_M \}.
\end{align}
For $\cl T \in \cl S_1$, we define $\cl E(\cl T)$ to be the sum of the energies of the nonlinear profiles that do not scatter, i.e.,
\begin{align}\label{minproc3}
\cl E(\cl T) = \sum_{j = 1}^{J_M} E ( \vec U^j).
\end{align}
Let
\begin{align}\label{minproc4}
\cl E_m = \inf \{ \cl E (\cl T) : \cl T \in \cl S_1 \}.
\end{align}

We now recall a result from \cite{dkm6} (see Corollary 4.3).

\begin{lem}
The infimum $\cl E_m$ is attained, i.e., there exists $\cl T_0 \in \cl S_1$ so that
$$
\cl E_m = \cl E (\cl T_0).
$$
\end{lem}

With the above lemma, we can define
\begin{align}
\cl S_2 &= \{ \cl T \in \cl S_1 : \cl E (\cl T) = \cl E_m \} \neq \varnothing, \\
J_m &= \min \{ J_1(\cl T) : \cl T \in \cl S_2 \}, \\
\cl S_3 &= \{ \cl T \in \cl S_2 : J_1(\cl T) = J_m \} \neq \varnothing.
\end{align}

In the radial setting, necessarily $J_m = 1$.  This follows from the following lemma proved in \cite{dkm6} (see Lemma 4.11).

\begin{lem}\label{minproclem}
There exists $\cl T_0 \in \cl S_3$ such that for all $j = 1, \ldots, J_m$,
\begin{align*}
t_{j,n} &= 0, \\
\lambda_{j,n} &= \lambda_{1,n},
\end{align*}
and $U^j$ has the compactness property on $I_{\max}(U^j)$.
\end{lem}

By orthogonality of the parameters, it follows that $J_m = 1$.  Moreover, by Proposition \ref{cptsolns}, after rescaling
$\lambda_{1,n}$, we have that $U^1 = \pm W$.

To proceed further, we distinguish two cases:
\begin{itemize}
\item Case 1: $v(t)$ scatters forward in time.
\item Case 2: $v(t)$ does not scatter forward in time.
\end{itemize}

We now determine the value of $J_M$ based on what case we are in.

\begin{clm}
In Case 1, we have $J_M = 1$ and $\cl E_m \geq E(W,0)$.  In Case 2, we have $J_M = 2$ and $\cl E_m \geq
E(W,0) + E(v_0,v_1)$.
\end{clm}

\begin{proof}
Let $\cl T_0 = \{\tau_n\}_n$ be the sequence of times given by Lemma \ref{minproclem}. Then $J_m = 1$ and
$U^1 = \pm W$.  Since $\cl T_0 \subset \cl S_3 \subset \cl S_2$,
$$
\cl E_m = \sum_{j = 1}^{J_M} E(U^j).
$$
By \eqref{belowdecomp2}, we have $E(u_0,u_1) = E(W,0) + E(v_0,v_1)$.  We now perform an expansion of $E(\vec a)$
along $\cl T_0$.  We know that all the nonzero profiles, as well as $w_n^J$ have positive energy.  Then in Case 1,
$$
\cl E_m = E(W,0) + \sum_{j = 2}^{J_M} E(\vec U^j) \geq E(W,0).
$$
If we are in Case 2, then for some $2 \leq j_L \leq J_M$, $U^{j_L} = v$.  Hence,
$$
\cl E_m = E(W,0) + \sum_{j = 2}^{J_M} E(\vec U^j) \geq E(W,0) + E(v_0,v_1).
$$

We now prove the statements about $J_M$.  Suppose we are in Case 1. Then
for $J$ sufficiently large, there exists an index $j_L$ with $J_M + 1 \leq j_L \leq J$ so that $U^{j_L} = v$.  By the Pythagorean expansion of the energy,
\begin{align*}
E(u_0,u_1) &= \sum_{j = 1}^{J_M} E(\vec U^j) + \sum_{j = J_M + 1}^J E(\vec U^j) + E(\vec w_n^J) + o(1) \\
&= E(W,0) + \sum_{j = 2}^{J_M} E(\vec U^j) + \sum_{j = J_M + 1}^J E(\vec U^j) + E(\vec w_n^J) + o(1) \\
&\geq E(W,0) + \sum_{j = 2}^{J_M} E(\vec U^j) + E(v_0,v_1) + o(1).
\end{align*}
as $n \rightarrow \infty$.  Subtracting $E(W,0) + E(v_0,v_1)$ from both sides, letting $n \rightarrow \infty$, and using the fact that $E(u_0,u_1) = E(W,0) + E(v_0,v_1)$, we
obtain
$$
0 = \sum_{j = 2}^{J_M} E(\vec U^j).
$$
Hence $J_M = 1$ in Case 1.  In Case 2, one similarly shows that $J_M = 2$.
\end{proof}

\subsubsection*{Step 3: Compactness of the singular part, $a(t)$}

In this step, we prove the following result.

\begin{lem}\label{complem}
For any sequence $\{ \tau_n \}_n$, with $\tau_n \rightarrow 1$, there exists a subsequence, still denoted by
$\tau_n$, and scales $\lambda_n > 0$ so that $\left (\lambda_n^{(N-2)/2} a(\tau_n \lambda_n x), \lambda_n^{N/2}
\partial_t a(\tau_n, \lambda_n x)\right )$ converges in $\energysp$.
\end{lem}

\begin{proof}
Let $\tau_n \rightarrow 1$.  After passing to a subsequence and reordering, we may assume $\{\tau_n\}_n \in \cl S_0$ so that the
profile decomposition of $\vec u(\tau_n)$ is pre--ordered.  By Step 1 and Step 2, we know that $(v_0,v_1)$ is a profile and
that either $J_M = 1$ or $J_M = 2$ depending on whether or not $v(t)$ scatters forward in time.  We also know that
$\vec U^1_L \prec (v_0,v_1)$.  Further, all of the profiles other than $(v_0,v_1)$ have positive energy and so does $
\vec w^J_n$.  As we shall see, Lemma \ref{complem} follows from the following claim.

\begin{clm}
All of the profiles that scatter forward in time must be identically 0,
\begin{align}\label{comp1}
\lim_{n \rightarrow \infty} \| (w^J_{0,n},w^J_{1,n}) \|_{\energysp} = 0,
\end{align}
and the nonlinear profile $U^1$ cannot scatter backwards in time.
\end{clm}
\begin{proof}
To prove the first two parts of the claim, we again rely on the positivity of the energies.  Since $J_M = 1$ or $J_M = 2$,
we know $\{ \tau_n \}_n \in \cl S_1$. In Case 1 we have
\begin{align*}
E(W,0) &= E(u_0,u_1) - E(v_0,v_1) = E(\vec U^1) + \sum_{j = 2}^J E(\vec U^j) + E(\vec w^J_n) + o(1) \\
&\geq \cl E_m + \sum_{j = 2}^J E(\vec U^j) + E(\vec w^J_n) + o(1) \\
&\geq E(W,0) + \sum_{j = 2}^J E(\vec U^j) + E(\vec w^J_n) + o(1).
\end{align*}
By canceling $E(W,0)$ from both sides, we have that $U^j = 0$ for all $j \geq 2$ and $\lim_{n \rightarrow \infty} \| (w^J_{0,n},w^J_{1,n}) \|_{\energysp} = 0$.
The same proof applies in Case 2 as well.

We now prove that the nonlinear profile $U^1$ cannot scatter backward in time. Suppose this was not the case, and that
$U^1$ scatters as $t \rightarrow -\infty$.  Let $t_0 > 0$ be small so that
$v(t)$ is defined on $[1-2t_0,1]$.  Proposition \ref{approxthm} gives for all $t_n \in [-t_0,0]$, for large $n$,
\begin{align*}
\vec u (\tau_n + t_n) = v(\tau_n + t_n) + \vec U^1_n(t_n) + o(1),
\end{align*}
in $\energysp$.  Set $t_n = 1 - \tau_n - t_0$.  This gives
\begin{align}\label{comp2}
\vec u (1-t_0) - \vec v(1-t_0) = \vec U^1_n(1-\tau_n - t_0) + o(1),
\end{align}
in $\energysp$.  Since $U^1$ does not scatter forward in time,  $\| \vec U^1_n(1-\tau_n - t_0) \|_{\energysp} \geq \frac{1}{C}\delta_0$ where
$\delta_0$ is from the local Cauchy theory.  Hence \eqref{comp2} is a nontrivial profile decomposition for the fixed function
$\vec u (1-t_0) - \vec v(1-t_0) \neq 0$.  This means that necessarily $1 - \tau_n - t_0 - t_{1,n} = c_0$ and $\lambda_{1,n} = 1$ for all $n$.  But by
Step 1, $\lambda_{1,n} \leq C (1 - \tau_n)$ for all $n$.  This is a contradiction, and thus $U^1$ does not scatter
backward in time.
\end{proof}

Since $\vec U^1_L \prec (v_0,v_1)$, $U^1$ does not scatter forward in time.  By the claim, $U^1$ does not scatter backwards in time. We then have that $\left | \frac{-t_{1,n}}{\lambda_{1,n}} \right | \leq C < \infty$.
Hence, we can assume without loss of generality that $t_{1,n} = 0$ for all $n$.  We now have that
\begin{align}
(a(\tau_n, x), \partial_t a(\tau_n , x) ) = \left (
\frac{1}{\lambda_{1,n}^{\frac{N-2}{2}}} U^1 \left ( 0 , \frac{x}{\lambda_{1,n}} \right ),
\frac{1}{\lambda_{1,n}^{\frac{N}{2}}} \partial_t U^1 \left ( 0 , \frac{x}{\lambda_{1,n}} \right )
\right ) + o(1)
\end{align}
in $\energysp$ as $n \rightarrow \infty$. This proves the lemma.
\end{proof}

\subsubsection*{Step 4: Conclusion of the Proof of Theorem \eqref{below1}}

Let $\{\tau_n\}_n$ be any sequence with $\tau_n \rightarrow 1$.  From Step 3, we know that there exists a function $\lambda(t) > 0$, $t \in [0,1)$ so that the trajectory
$$
K = \left \{ \left ( \lambda(t)^{(N-2)/2} a(t,\lambda(t) \cdot), \lambda(t)^{N/2} \partial_t a(t,\lambda(t) \cdot) \right ): t \in [0,1)  \right \}
$$
has compact closure in $\energysp$.  Hence, after extraction, there exists $(U_0,U_1) \in \energysp$ so that
$$
\left ( \lambda(\tau_n)^{(N-2)/2} a(\tau_n,\lambda(\tau_n) \cdot), \lambda(\tau_n)^{N/2} \partial_t a(\tau_n,\lambda(t_n) \cdot) \right )
\rightarrow (U^0,U^1)
$$
in $\energysp$ as $n \rightarrow \infty$.

Let $U$ be the solution to \eqref{nlw} with data $(U_0,U_1)$. By Lemma 8.5 from
\cite{dkm1}, we have that $U$ has the compactness property on $I_{\max}(U)$.  By Theorem \ref{cptsolns}, $U = \pm W$
up to scaling.  As this is true for any sequence $\{\tau_n\}_n$, a diagonal argument gives that
\begin{align}\label{concl1}
d\left ( \vec a(t), \cl O^+ \cup \cl O^-  \right ) \rightarrow 0 \quad \mbox{as } t\rightarrow 1,
\end{align}
where $d$ is the $\energysp$ distance to a set and
$$
\cl O^{\pm} = \left \{ \left ( \frac{\pm 1}{\lambda^{\frac{N-2}{2}}} W \left ( \frac{\cdot}{\lambda}\right ), 0 \right )
: \lambda > 0 \right \}.
$$
It is easy to see that
$$
d_0 = d \left (\cl O^+, \cl O^- \right ) > 0.
$$
Define the sets of time
$$
\cl U^{\pm} = \left \{ t \geq 0 : d \left (\vec a(t) , \cl O^{\pm} \right ) < d_0/2 \right \}.
$$
By our definition of $d_0$, $\cl U^+$ and $\cl U^-$ are disjoint.  We have also proved that for some $t_0$ sufficiently close
to $1$, $[t_0,1) \subset \cl U^+ \cup \cl U^-$. By continuity of $t \mapsto \vec a(t)$, both $\cl U^+$ and $\cl U^-$ are open.

By the conclusion of Theorem \ref{thm1}, $\cl U^+ \cap [t_0,1)$ is not empty.  By connectedness, $[t_0, 1) \subset \cl U^+$. In view of
\eqref{concl1}, this implies that there exists a function $\lambda(t) > 0$ such that
$$
\lim_{t \rightarrow 1} \left ( \lambda(t)^{\frac{N-2}{2}} a(t, \lambda(t)\cdot) , \lambda(t)^{\frac{N}{2}} \partial_t a(t, \lambda(t) \cdot ) \right ) = (W,0) \quad \mbox{in } \energysp.
$$
By Lemma A.1 of \cite{ckls} (with $G = (\R^{\times}, \cdot)$ acting
on $\energysp$ by scaling), we can choose $\lambda$ to be continuous. This completes the proof of Theorem \ref{below1}.

\subsection{Global solutions}

In this subsection, we prove Theorem \ref{below2}.  We assume that $u$ does not scatter in forward time, so that our goal is to
prove that there exists a positive continuous function $\lambda(t)$ such that
\begin{align*}
\lim_{t \rightarrow \infty} \frac{\lambda(t)}{t} = 0,
\end{align*}
and
\begin{align}\label{gbelow1}
\vec u(t) = \vec v_L(t) +  \left (\frac{\iota}{\lambda(t)^{(N-2)/2}} W \left (\frac{x}{\lambda(t)} \right ),0 \right ) + o(1)
\quad \mbox{in } \energysp \mbox{ as } t \rightarrow \infty.
\end{align}
We also assume that there exists $A > 0$ such that
\begin{align}\label{gbelow2}
\limsup_{t \rightarrow \infty} \int_{|x| \leq t - A} |\nabla u(t)|^2 + |\partial_t u(t)|^2 dx < 2 \| \nabla W \|_{L^2}^2.
\end{align}
By Theorem \ref{thm2}, we have that there exists a sequence $t_n \rightarrow +\infty$, an integer $J_0 \geq 0$, $J_0$ sequences of positive numbers $(\lambda_{j,n})$,
 $j = 1, \ldots, J_0$, $J_0$ signs $\iota_j \in \{ \pm 1 \}$ such that
\begin{align*}
\lambda_{1,n} \ll \lambda_{2,n} \ll \cdots \ll \lambda_{J_0,n} \ll t_n,
\end{align*}
and
\begin{align}\label{gbelow3}
\vec u(t_n) = \vec v_L(t_n) + \sum_{j = 1}^{J_0} \left (\frac{\iota_j}{\lambda_{j,n}^{\frac{N-2}{2}}} W \left (\frac{x}{\lambda_{j,n}} \right ),0 \right ) + o(1)
\quad \mbox{in } \energysp \mbox{ as } n \rightarrow \infty.
\end{align}

As in the blow--up case, we divide the proof of Theorem \ref{below2} into steps.

\subsubsection*{Step 1: Preliminary observations on a profile decomposition.}

Let $v(t)$ be the unique solution to \eqref{nlw} such that
$$
\lim_{t \rightarrow \infty} \| \vec v(t) - \vec v_L(t) \|_{\energysp} = 0.
$$
Let $a(t) = u(t) - v(t)$.

By Proposition \ref{selfsimilarblowup2} and Lemma \ref{loclem}, we know that
\begin{align*}
\lim_{t \rightarrow +\infty} \int_{|x| \geq t/2} |\nabla a(t)|^2 dx &=  0, \\
\lim_{t \rightarrow +\infty} \int_{|x| \leq t/2} |\nabla v(t)|^2 dx &= 0.
\end{align*}
Hence, as $t \rightarrow \infty$,
\begin{align*}
\int_{\R^N} |\nabla a(t)|^2 dx &=  \int_{|x| \leq t/2} |\nabla a(t)|^2 dx + o(1) \\
&= \int_{|x| \leq t/2} |\nabla u(t)|^2 dx + o(1) \\
&< 2 \| \nabla W \|^2_{L^2} + o(1).
\end{align*}
Hence for all $t \geq T$,
\begin{align}\label{gbelow4}
\| \nabla a(t) \|_{L^2}^2 < 2 \| \nabla W \|^2_{L^2}.
\end{align}
The convergence \eqref{gbelow3} becomes
\begin{align}\label{gbelow5}
\vec a(t_n) = \sum_{j = 1}^{J_0} \left (\frac{\iota_j}{\lambda_{j,n}^{\frac{N-2}{2}}} W \left (\frac{x}{\lambda_{j,n}} \right ),0 \right ) + o(1)
\quad \mbox{in } \energysp \mbox{ as } n \rightarrow \infty.
\end{align}
By orthogonality arguments, \eqref{gbelow4} implies $J_0 \leq 1$.

We claim $J_0 = 1$.  By Claim \ref{globlim}, $E(\vec a(t))$ has a limit. By \eqref{gbelow5},
$$
\lim_{t \rightarrow +\infty} E(\vec a(t)) = J_0 E(W,0).
$$
If $J_0 = 0$, then by Lemma \ref{energytrap}, we have that
$$
\lim_{t \rightarrow +\infty} \| \vec u(t) - \vec v_L(t) \|_{\energysp} = \lim_{t \rightarrow \infty} \| \vec a(t) \|_{\energysp} = 0,
$$
which implies $u$ scatters forward in time.  This contradicts our initial assumption that $u$ does not scatter forward in time.  Thus $J_0 = 1$, and
$$
\lim_{t \rightarrow +\infty} E(\vec a(t)) = E(W,0).
$$
Moreover, after
a sign change, we assume that $\iota_1 = 1$ so that \eqref{gbelow5} becomes
$$
\vec a(t_n) = \left (\frac{1}{\lambda_{n}^{\frac{N-2}{2}}} W \left (\frac{x}{\lambda_{n}} \right ),0 \right ) + o(1)
\quad \mbox{in } \energysp \mbox{ as } n \rightarrow \infty.
$$

Let $\tau_n \rightarrow +\infty$.  We now make some observations about the profile decomposition of $\vec a(\tau_n)$.  By Lemma
\ref{proford}, we may assume that $\{\tau_n\}_n \in \cl S_0$, i.e. the profile decomposition
$\{ \vec U^j_L, \lambda_{j,n}, t_{j,n} \}_{j \geq 1}$ is ordered by $\preccurlyeq$.  By Lemma \ref{bddlem} and Proposition \ref{selfsimilarblowup2}, we have that the parameters
satisfy
$$
\lambda_{j,n} + |t_{j,n}| \leq C_j \tau_n.
$$

\begin{clm}
Let $\vec U^0_L = \vec v_L(0)$, $\lambda_{0,n} = 1$, and $t_{0,n} = \tau_n$ (with nonlinear profile $U^0(t) = v(t)$).  Then $\{ \vec U^j_L, \lambda_{j,n}, t_{j,n} \}_{j \geq 0}$
is a profile decomposition for $\vec u(\tau_n)$.
\end{clm}

\begin{proof}
The only thing that needs to be checked is the pseudo--orthogonality of the parameters.  This is a consequence of the construction of a profile decomposition
and $\vec S(-t)u(t) \rightharpoonup (v_0,v_1)$.
\end{proof}

This situation is the general case.  More precisely, as $S(-t)\vec u(t) \rightharpoonup (v_0,v_1)$ as $t \rightarrow +\infty$, and from the general construction of profile decomposition (see \cite{bulut}), we have

\begin{clm}
Let $\{\tau_n\}_n$ be any sequence tending to $+\infty$.  The sequence $\{\vec u(\tau_n)\}_n$ admits a profile decomposition with profiles
$\{U^j\}_j$ and parameters $\{\lambda_{j,n}, t_{j,n} \}_{j,n}$ ordered by $\preccurlyeq$.  Then for some $j_L \geq 2$,
$$
U^{j_L}_L = v_L, \quad \lambda_{j_L,n} = 1, \quad t_{j_L,n} = \tau_n.
$$
Also, $\{U^j_L, \lambda_{j,n}, t_{j,n}\}_{j \neq j_L}$ is a $\preccurlyeq$--ordered profile decomposition of $\vec a(\tau_n)$.
\end{clm}

Since $u$ does not scatter in forward time, by Proposition \ref{approxthm} at least one of the nonlinear profiles $U^j$ does not scatter forward in time.  Due to the
ordering, $\preccurlyeq$, this means that $U^1$ does not scatter forward in time.  Since $U^0 = v$ scatters forward in time, we conclude that $1 \prec 0$.

By the Pythagorean expansion of the $\dot H^1$ norm and the bound on $a$ \eqref{gbelow4}, we have
\begin{align*}
\forall j \geq 1, \quad \limsup_{n \rightarrow \infty} \| \nabla U^j_L (-t_{j,n} / \lambda_{j,n} ) \|^2_{L^2} &< 2 \| \nabla W \|_{L^2}^2, \\
\forall J \geq 1, \quad \limsup_{n \rightarrow \infty} \| \nabla w^J_{0,n} \|^2_{L^2} &< 2 \| \nabla W \|_{L^2}^2.
\end{align*}

As in the blow--up case  Lemma \ref{energytrap} implies that for all $j, J \geq 1$, for all $n$ sufficiently large,
\begin{align*}
E(\vec U^j_L(-t_{j,n}/\lambda_{j,n})) &\geq 0, \\
E(\vec w^J_n(0)) &\geq 0,
\end{align*}
as well as
\begin{align*}
\forall j \geq 1 \quad E(\vec U^j) > 0 \mbox{ or } U^j = U^j_L \equiv 0, \\
\lim_{n \rightarrow \infty} E(\vec w^J_n(0)) = 0 \implies
\lim_{n \rightarrow \infty} \| (w_{0,n}^J, w_{1,n}^J) \|_{\energysp} = 0.
\end{align*}

\subsubsection*{Step 2: Minimization process and consequences}

We again use a general compactness argument for profile decompositions of $\vec u(\tau_n)$ developed in \cite{dkm6}.  This step is similar to the finite time blow--up case.  We first
recall some notation.

Let $\mathcal S_0$ denote the set of sequences $\{\tau_n\}_n$ with $\tau_n \rightarrow +\infty$
so that $\vec u(\tau_n)$ admits a pre--ordered profile decomposition. Let $\cl T = \{\tau_n\}_n \in \cl S_0$, and let
\begin{align*}
J_0(\cl T) &= \# \mbox{ of profiles of $\vec u(\tau_n)$ that do not scatter forward in time.} \\
J_1 (\cl T) &= \min \{ j \geq 1 : j \prec j + 1 \}.
\end{align*}

Here $\prec$ is the strict order from Definition \ref{proforddef}.  As in the finite time blow--up case,
$J_1 (\cl T) \leq J_0 (\cl T)$ which is uniformly bounded on $\cl S_0$ so that $J_1 ( \cl T)$ is uniformly bounded on $\cl S_0$ as well.

Now define
\begin{align}
J_M &= \max \{ J_0(\cl T) : \cl T \in \cl S_0 \}, \\
\cl S_1 &= \{ \cl T \in \cl S_0 : J_0 (\cl T) = J_M \}.
\end{align}
For $\cl T \in \cl S_1$, we define $\cl E(\cl T)$ to be the sum of the energies of the nonlinear profiles that do not scatter, i.e.,
\begin{align}\label{gminproc3}
\cl E(\cl T) = \sum_{j = 1}^{J_M} E ( \vec U^j).
\end{align}
Let
\begin{align}\label{gminproc4}
\cl E_m = \inf \{ \cl E (\cl T) : \cl T \in \cl S_1 \}.
\end{align}
Again by \cite{dkm6} the infimum $\cl E_m$ is attained, i.e., there exists $\cl T_0 \in \cl S_1$ so that
$$
\cl E_m = \cl E (\cl T_0).
$$
We then define
\begin{align}
\cl S_2 &= \{ \cl T \in \cl S_1 : \cl E (\cl T) = \cl E_m \} \neq \varnothing, \\
J_m &= \min \{ J_1(\cl T) : \cl T \in \cl S_2 \}, \\
\cl S_3 &= \{ \cl T \in \cl S_2 : J_1(\cl T) = J_m \} \neq \varnothing.
\end{align}

In the radial setting, necessarily $J_m = 1$.  This again follows from the lemma proved in \cite{dkm6}:

\begin{lem}\label{gminproclem}
There exists $\cl T_0 \in \cl S_3$ such that for all $j = 1, \ldots, J_m$,
\begin{align*}
t_{j,n} &= 0, \\
\lambda_{j,n} &= \lambda_{1,n},
\end{align*}
and $U^j$ has the compactness property on $I_{\max}(U^j)$.
\end{lem}

By orthogonality of the parameters, it follows that $J_m = 1$.  Moreover, by Proposition \ref{cptsolns}, after rescaling
$\lambda_{1,n}$, we have that $U^1 = \pm W$.  We now make the following claim.

\begin{clm}
We have that $J_M = 1$ and $\cl E_m \geq E(W,0)$.
\end{clm}

\begin{proof}
Let $\cl T_0 = \{\tau_n\}_n$ be the sequence of times given by Lemma \eqref{gminproclem}. Then $J_m = 1$ and
$U^1 = \pm W$.  Since $\cl T_0 \subset \cl S_3 \subset \cl S_2$,
$$
\cl E_m = \sum_{j = 1}^{J_M} E(U^j).
$$
We know that all of the nonzero profiles other than $v$, as well as $w_n^J$ have positive energy.
 We also
know that $v_L$ must appear in the profile decomposition, at some index $j_L$ with $j_L > J_M$, because $v_L$ has a scattering
profile $v$.
This implies that $\cl E_m \geq E(W,0)$.

We now show $J_M = 1$. Up to an $o(1)$ term in $\energysp$,
$$
\vec u(\tau_n) - \vec v(\tau_n) = \sum_{j =1}^{J_M} \vec U^j_{L,n}(0) + \sum_{j = J_M + 1}^{j_L - 1} \vec U^j_{L,n}(0)
+ \vec w^{j_L}_n(0).
$$
By the Pythagorean expansion of the energy, we obtain
\begin{align*}
E(W,0) &= \lim_{n \rightarrow \infty} E(\vec a(\tau_n)) \\
&\geq E(W,0) + \sum_{j = 2}^{J_M} E(\vec U^j).
\end{align*}
This implies that $J_M = 1$ as desired.
\end{proof}

\subsubsection*{Step 3: Compactness of the singular part, $a(t)$}

In this step, we prove the following result.

\begin{lem}\label{gcomplem}
For any sequence $\{ \tau_n \}_n$, with $\tau_n \rightarrow +\infty$, there exists a subsequence, still denoted by
$\tau_n$, and scales $\lambda_n > 0$ so that $\left (\lambda_n^{(N-2)/2} a(\tau_n \lambda_n x), \lambda_n^{N/2}
\partial_t a(\tau_n, \lambda_n x)\right )$ converges in $\energysp$.
\end{lem}

\begin{proof}
Let $\tau_n \rightarrow +\infty$.  After passing to a subsequence and reordering, we may assume $\{\tau_n\}_n \in \cl S_0$ so that the
profile decomposition of $\vec u(\tau_n)$ is pre--ordered.  By Step 1 and Step 2, we know that $v_L$ is a profile,
$J_M = 1$, and $\cl E_m \geq E(W,0)$  We also know that
$\vec U^1$ does not scatter forward in time.  Further, all of the profiles other than $v_L$ have positive energy and so does $
\vec w^J_n$.  As in the finite time blow--up case, Lemma \ref{gcomplem} follows from the claim:

\begin{clm}
All of the profiles other than $v_L$ that scatter forward in time must be identically 0,
\begin{align}\label{gcomp1}
\lim_{n \rightarrow \infty} \| (w^J_{0,n},w^J_{1,n}) \|_{\energysp} = 0,
\end{align}
and the nonlinear profile $U^1$ cannot scatter backwards in time.
\end{clm}
\begin{proof}
To prove the first two parts of the claim, we again rely on the positivity of the energies.  Since $J_M = 1$,
we know $\{ \tau_n \}_n \in \cl S_1$.  If $v_L = U^{j_L}_L$, then for all $J \geq j_L$, we have (up to an $o(1)$ term)
\begin{align*}
\vec u(\tau_n) - \vec v(\tau_n) = \vec U^1_{L,n}(0) + \sum_{j = 2}^{j_L - 1} \vec U^j_{L,n}(0)+
\sum_{j = j_L + 1}^{J} \vec U^j_{L,n}(0) + \vec w_n^J(0).
\end{align*}
The Pythagorean expansion of the energy yields
\begin{align*}
E(W,0) &= E(\vec U^1 ) + \sum_{j = 2}^{j_L - 1} E(\vec U^j)+
\sum_{j = j_L + 1}^{J} E(\vec U^j)+ E(\vec w_n^J(0)) + o(1) \\
&\geq E(W,0) + \sum_{j = 2}^{j_L - 1} E(\vec U^j)+
\sum_{j = j_L + 1}^{J} E(\vec U^j)+ E(\vec w_n^J(0)) + o(1).
\end{align*}
This shows $U^j = 0$ for all $j \neq j_L$ and \eqref{gcomp1}.

We now prove that the nonlinear profile $U^1$ cannot scatter backward in time. Suppose this were not that case, and that
$U^1$ scatters as $t \rightarrow -\infty$.  Choose $t_0 > 0$ so large so that $t_0 > T_-(v)$. Set $\theta_n = t_0 - \tau_n$.
Proposition \ref{approxthm} gives for $n$ sufficiently large
\begin{align}\label{gcomp2}
\vec u (t_0) - \vec v(t_0) = \vec U^1_n(t_0 - \tau_n) + o(1),
\end{align}
in $\energysp$.  Since $U^1$ does not scatter forward in time,  $\| \vec U^1_n(\tau_n - t_0) \|_{\energysp} \geq \frac{1}{C}\delta_0$ where
$\delta_0$ is from the local Cauchy theory.  Hence \eqref{gcomp2} is a nontrivial profile decomposition for the fixed function
$\vec u (t_0) - \vec v(t_0) \neq 0$.  This means that necessarily $t_0 - \tau_n - t_{1,n} = c_0$ and $\lambda_{1,n} = 1$ for all $n$.  But then
$t_{1,n} \rightarrow +\infty$ as $n \rightarrow \infty$.  This is a contradiction since $U^1$ does not scatter
forward in time.  Thus, $U^1$ does not scatter backwards in time.
\end{proof}

By the claim, $U^1$ does not scatter forward or backwards in time. We then have that $\left | \frac{-t_{1,n}}{\lambda_{1,n}} \right | \leq C < \infty$.
Hence, we can assume without loss of generality that $t_{1,n} = 0$ for all $n$.  We now have that
\begin{align*}
(a(\tau_n, x), \partial_t a(\tau_n , x) ) = \left (
\frac{1}{\lambda_{1,n}^{\frac{N-2}{2}}} U^1 \left ( 0 , \frac{x}{\lambda_{1,n}} \right ),
\frac{1}{\lambda_{1,n}^{\frac{N}{2}}} \partial_t U^1 \left ( 0 , \frac{x}{\lambda_{1,n}} \right )
\right ) + o(1)
\end{align*}
in $\energysp$ as $n \rightarrow \infty$. This proves the lemma.
\end{proof}

\subsubsection*{Step 4: Conclusion of the Proof of Theorem \eqref{below2}}

Here the argument is exactly the same as the corresponding step for the finite time blow--up case, and
we only sketch it.  Let $\{\tau_n\}_n$ be any sequence with $\tau_n \rightarrow +\infty$.  From Step 3, we know that there exists a function $\lambda(t) > 0$, $t \in [0,1)$ so that the trajectory
$$
K = \left \{ \left ( \lambda(t)^{(N-2)/2} a(t,\lambda(t) \cdot), \lambda(t)^{N/2} \partial_t a(t,\lambda(t) \cdot) \right ):
t \in [0,+\infty)  \right \}
$$
has compact closure in $\energysp$.  By Lemma 8.5 from \cite{dkm1} and Theorem \ref{cptsolns}, after extraction,
$$
\left ( \lambda(\tau_n)^{(N-2)/2} a(\tau_n,\lambda(\tau_n) \cdot), \lambda(\tau_n)^{N/2} \partial_t a(\tau_n,\lambda(t_n) \cdot) \right )
\rightarrow (\pm W,0)
$$
in $\energysp$ as $n \rightarrow \infty$.

A continuity argument implies that the sign can be chosen independent ofn $\{\tau_n\}_n$.  This implies that there exists a function $\lambda(t) > 0$ such that
$$
\lim_{t \rightarrow 1} \left ( \lambda(t)^{\frac{N-2}{2}} a(t, \lambda(t)\cdot) , \lambda(t)^{\frac{N}{2}} \partial_t a(t, \lambda(t) \cdot ) \right ) = (W,0) \quad \mbox{in } \energysp.
$$
By Lemma A.1 of \cite{ckls} we can choose $\lambda$ to be continuous. This completes the proof of Theorem \ref{below2}.

\appendix
\section{}  In this appendix, we prove the Proposition \ref{approxthm} for $N \geq 6$. First, we recall a long time
perturbation result from \cite{stab} (see Theorem 3.6).
We denote the following spaces
\begin{align*}
S(I) &= L^{\frac{2(N+1)}{N-2}}(I \times \R^N), \\
W(I) &= L^{\frac{2(N+1)}{N-1}}
\left ( I ; \dot B^{\frac{1}{2}, 2}_{\frac{2(N+1)}{N-1}} \right ), \\
W'(I) &= L^{\frac{2(N+1)}{N+3}}
\left ( I; \dot B^{\frac{1}{2}, 2}_{\frac{2(N+1)}{N+3}} \right ),
\end{align*}
For $0 < s < 1$ and $1 < p < \infty$, $\dot B^{s, 2}_{p}$ is the standard Besov space on $\R^N$ with norm
\begin{align*}
\| f \|_{\dot B^{s, 2}_{p}} = \left ( \int_{\R^N} |y|^{-N-2s} \| f(x + y) - f(x) \|
_{L^{p}_x}^2 dy \right )^{\frac{1}{2}}.
\end{align*}
We denote the nonlinearity $F(u) = |u|^{\frac{4}{N-2}}u.$
\begin{ppn}\label{longtime}
Let $N \geq 6$.  Assume $\tilde u$ is a near solution on $I \times \R^N$
\begin{align*}
\partial_t^2 \tilde u - \Delta \tilde u = F(\tilde u) + e,
\end{align*}
such that
\begin{align*}
\sup_{t \in I} \| (\tilde u(t),\partial_t \tilde u(t)) \|_{\energysp} + \| \tilde u \|_{W(I)}   &\leq E \\
\| \tilde u_0 - u_0 \|_{\dot H^1} + \| \tilde u - u_1 \|_{L^2} + \|  e \|_{W'(I)} &\leq \epsilon
\end{align*}
Then there exists $\epsilon_0 = \epsilon_0(N,E)$ such that if $0 < \epsilon < \epsilon_0$, there exists a unique solution
to \eqref{nlw} on $I \times \R^N$ with initial data $(u_0,u_1)$ such that
\begin{align*}
\sup_{t \in I} \| (\tilde u(t) - u(t), \partial_t \tilde u(t) - \partial_t u(t)) \|_{\energysp} + 
\| \tilde u - u \|_{S(I)} \leq C \epsilon^c
\end{align*}
where $C > 0$ and $0 < c < 1$ depend only on $N$ and $E$.
\end{ppn}
We remark that the theorem stated in \cite{stab} assumes smallness of $\| D^{\frac{1}{2}} e \|_{L^{\frac{2(N+1)}{N+3}}_{t,x}(I \times \R^N)}$ rather than of the slightly weaker norm $\|  e \|_{W'(I)}$, but the proof adapts nearly verbatim.

\begin{proof}[Proof of Proposition \ref{approxthm}]
Let
$\tilde u^J_n(t,x) = \sum_{j = 1}^J U^j_{n}(t) + w^J_n(t)$.
We will apply Proposition \ref{longtime} to $\tilde u^J_n$ and $u_n$ for $n$ and $J$ large.
We first note that the hypotheses, Strichartz estimates, and a simple bootstrapping argument imply that
\begin{align*}
\forall j \geq 1, \mbox{ } \limsup_{n \rightarrow \infty} \left [  \| U^j_n(t) \|_{W(0,\theta_n)} + \sup_{t \in [0,\theta_n]} \| (U^j_n(t), \partial_t U^j_n(t)) \|_{\energysp}
\right ] < \infty
\end{align*}
By orthogonality of the linear energy, there exists $J_0 \geq 1$ such that
\begin{align*}
\sum_{j > J_0} \| (U^j_0, U^j_1) \|_{\energysp}^2 \ll \delta_0
\end{align*}
where $\delta_0$ is as in the local Cauchy theory.  By the small data theory, $U^j$ is globally defined for all $j > J_0$ and
\begin{align*}
 \| U^j \|_{S(\R)} + \| U^{j} \|_{W(\R)}  +
\sup_{t \in \R} \| (U^j(t), \partial_t U^j(t) ) \|_{\energysp} \leq C \|(U_0^j,U^j_1)\|_{\energysp}.
\end{align*}
Hence
\begin{align*}
\sum_{j > J_0} \| U^{j} \|_{S(\R)}^{\frac{2(N+1)}{N-2}} + \sum_{j > J_0} \| U^{j} \|_{W(\R)}^{\frac{2(N+1)}{N-1}}+
\sum_{j > J_0} \sup_{t \in \R} \| (U^j(t), \partial_t U^j(t) ) \|_{\energysp}^2 < \infty
\end{align*}

Let $J \geq 1$. We first prove that the orthogonality of the parameters 
\begin{align}
i \neq j \implies \lim_{n \rightarrow \infty} \frac{\lambda_{i,n}}{\lambda_{j,n}} + \frac{\lambda_{j,n}}{\lambda_{i,n}}
+\frac{|t_{i,n} - t_{j,n}|}{\lambda_{j,n}} = +\infty, \label{appendixa7}
\end{align}
implies 
\begin{align}
\left \| \sum_{j = 1}^J U^j_n(t) \right \|^{\frac{2(N+1)}{N-1}}_{W(0,\theta_n)} &=
\sum_{j = 1}^J \| U^j_n(t) \|^{\frac{2(N+1)}{N-1}}_{W(0,\theta_n)} + o_n(1) \label{appendixa1} \\
\left \| \sum_{j = 1}^J U^j_n(t) \right \|^{\frac{2(N+1)}{N-2}}_{S(0,\theta_n)} &=
\sum_{j = 1}^J \| U^j_n(t) \|^{\frac{2(N+1)}{N-2}}_{S(0,\theta_n)} + o_n(1) \label{appendixa2}
\end{align}
as $n \rightarrow \infty$. Recall that if $a_1,\ldots,a_J$ are real numbers, then for all $p \geq 2$
$$
\left | \left |\sum_{j = 1}^J a_j \right |^p - \sum_{j = 1}^J |a_j|^p \right | \leq C(J) 
\sum_{i \neq j} |a_i|^{p-1}|a_{j}|.
$$  
This fact implies that 
$$
\left \| \sum_{j = 1}^J U^j_n(t) \right \|_{B^{\frac{1}{2},2}_{\frac{2(N+1)}{N-1}}} = \sum_{j = 1}^J \| U^j_n(t) \|_{B^{\frac{1}{2},2}_{
\frac{2(N+1)}{N-1}}} + O \left (\sum_{i \neq j}^J R^{i,j}_n(t) \right ),
$$
where the implied constant depends only on $J$ and 
\begin{align*}
R^{i,j}_n(t) = \left ( \int_{\R^N}
|y|^{-N-1} \|  |\Delta_y U^i_n(t,x)| |\Delta_y U^{j}_n(t,x)|^{\frac{N+3}{N-1}} \|_{L^{1}}^{\frac{N-1}{N+1}} dy 
\right )^{\frac{1}{2}}. 
\end{align*}
We now show using \eqref{appendixa7} that
\begin{align}
i \neq j \implies \lim_{n \rightarrow \infty} \int_{0}^{\theta_n} \left | R^{i,j}_n(t)\right |^{\frac{2(N+1)}{N-1}} dt = 0. \label{appendixa3}
\end{align}
By density, we assume that $U^i,U^{j} \in C^\infty_0(\R \times \R^N)$ and extend the integration in $t$ 
to $\R$. Suppose first that 
$$
\lim_{n \rightarrow \infty} \frac{\lambda_{i,n}}{\lambda_{j,n}} + \frac{\lambda_{j,n}}{\lambda_{i,n}}  = +\infty.
$$
Without loss of generality, suppose that $\lambda_{i,n}/ \lambda_{j,n} \rightarrow 0$ as 
$n \rightarrow \infty$.  By H\"older's inequality in $x$ then in
$y$,
\begin{align}
R^{i,j}_n(t) &\leq \left ( 
\int_{\R^N} |y|^{-N-1} \| \Delta_y U^i_n(t,x) \|_{L^{\frac{2(N+1)}{N-1}}}^{\frac{N-1}{N+1}}
\| \Delta_y U^i_n(t,x) \|_{L^{\frac{2(N+1)}{N-1}}}^{\frac{N+3}{N-1}} dy 
\right )^{\frac{1}{2}} \\
&\leq
\| U^i_n(t) \|_{B^{\frac{1}{2},2}_{\frac{2(N+1)}{N-1}}}^{\frac{N-1}{2(N+1)}}
\| U^j_n(t) \|_{B^{\frac{1}{2},2}_{\frac{2(N+1)}{N-1}}}^{\frac{N+3}{2(N+1)}}. \label{appendixa6}
\end{align}
Note that by a change of variables in $x$ and $y$, 
$$
\forall i \geq 1, \quad \| U^i_n(t) \|_{B^{\frac{1}{2},2}_{\frac{2(N+1)}{N-1}}} = 
\lambda_{i,n}^{- \frac{N-1}{2(N+1)}}\left \| U^i\left ( \frac{t - t_{i,n}}{\lambda_{i,n}} \right ) \right \|_{B^{\frac{1}{2},2}_{\frac{2(N+1)}{N-1}}}. 
$$
By \eqref{appendixa6} and a change of variables in $t$ we obtain
\begin{align}
\int_{\R} \left | R^{i,j}_n(t)\right |^{\frac{2(N+1)}{N-1}} dt &
\leq \left ( \frac{\lambda_{i,n}}{\lambda_{j,n}} \right )^{\frac{N+3}{2(N+1)}} \int_{\R}
\| U^i(t) \|_{B^{\frac{1}{2},2}_{\frac{2(N+1)}{N-1}}}
\left \| U^j \left ( \frac{\lambda_{i,n}}{\lambda_{j,n}} t + \frac{t_{i,n} - t_{j,n}}{\lambda_{j,n}} \right )
 \right \|_{B^{\frac{1}{2},2}_{\frac{2(N+1)}{N-1}}}^{\frac{N+3}{N-1}}
dt. \label{appendixa4}
\end{align}
Since we are assuming that $U^i,U^{j} \in C^\infty_0(\R \times \R^N)$ and $\lambda_{i,n}/ \lambda_{j,n} \rightarrow 0$, 
the right side of \eqref{appendixa4} tends to $0$ as $n \rightarrow \infty$. Thus 
$$
\lim_{n \rightarrow \infty} \int_{\R} \left | R^{i,j}_n(t)\right |^{\frac{2(N+1)}{N-1}} dt = 0.
$$

Suppose now that there exists a constant $C$ so that $C^{-1} \lambda_{i,n} \leq \lambda_{j,n} \leq C \lambda_{i,n}$.  After 
extraction and rescaling we
may assume that $\lambda_{i,n} = \lambda_{j,n}$.  Then \eqref{appendixa7} reads
$$
i \neq j \implies \lim_{n \rightarrow \infty} 
 \frac{|t_{i,n} - t_{j,n}|}{\lambda_{j,n}} = +\infty.
$$
Returning to \eqref{appendixa4}, we have that 
\begin{align}
\int_{\R} \left | R^{i,j}_n(t)\right |^{\frac{2(N+1)}{N-1}} dt &
\leq \int_{\R}
\| U^i(t) \|_{B^{\frac{1}{2},2}_{\frac{2(N+1)}{N-1}}}^{\frac{N-1}{2(N+1)}}
\left \| U^j \left ( t + \frac{t_{i,n} - t_{j,n}}{\lambda_{j,n}} \right ) \right \|_{B^{\frac{1}{2},2}_{\frac{2(N+1)}{N-1}}}^{\frac{N+3}{2(N+1)}}
dt. \label{appendixa5}
\end{align} 
Since we are assuming that $U^i,U^{j} \in C^\infty_0(\R \times \R^N)$ and $|t_{i,n} - t_{j,n}| / \lambda_{j,n} \rightarrow +\infty$, 
$U^i(t)$ and $ U^j \left ( t + \frac{t_{i,n} - t_{j,n}}{\lambda_{j,n}} \right )$ will have disjoint
support in $t$ for large $n$.  This implies that the right side of \eqref{appendixa5} is 0 for large $n$ so that
$$
\lim_{n \rightarrow \infty} \int_{\R} \left | R^{i,j}_n(t)\right |^{\frac{2(N+1)}{N-1}} dt = 0.
$$
Thus the orthogonality of the parameters implies 
$$
\lim_{n \rightarrow \infty} \int_{\R} \left | R^{i,j}_n(t)\right |^{\frac{2(N+1)}{N-1}} dt = 0
$$
which implies \eqref{appendixa1}.  The proof for \eqref{appendixa2} is similar (and simpler) and we omit it. 
  
By \eqref{appendixa1} and \eqref{appendixa2},
\begin{align*}
\lim_{J \rightarrow \infty} \limsup_{n \rightarrow \infty} \left ( \| \tilde u^J_n \|_{W(0,\theta_n)} + \| \tilde u^J_n \|_{S(0,\theta_n)} \right )< \infty.
\end{align*}
Similarly
\begin{align*}
\lim_{J \rightarrow \infty} \limsup_{n \rightarrow \infty} \left [\sup_{t \in [0,\theta_n]} \| (\tilde u^J_n(t), \partial_t \tilde u_n^J(t)
\|_{\energysp} \right ]< \infty.
\end{align*}
Let
$$
e_n^J =  F\left ( \sum_{j = 1}^J U^j_n(t) + w^J_n(t) \right ) - \sum_{j = 1}^J F(U^j_n(t).
$$
We claim that
\begin{align}
\lim_{J \rightarrow \infty} \limsup_{n \rightarrow \infty} \| e_n^J \|_{W'(0,\theta_n)} = 0. 
\end{align}
By the triangle inequality, this reduces to proving
\begin{align}\label{approxthm1}
\lim_{J \rightarrow \infty} \limsup_{n \rightarrow \infty} \left \| F \left ( \tilde u^J_n(t) - w^J_n(t) \right ) -
F( \tilde u^J_n(t) ) \right \|_{W'(0, \theta_n)} = 0,
\end{align}
and
\begin{align}\label{approxthm2}
\lim_{J \rightarrow \infty} \limsup_{n \rightarrow \infty} \left \| F \left ( \sum_{j = 1}^J U^j_n(t) \right ) -
\sum_{j = 1}^J F(U^j_n(t) \right \|_{W'(0,\theta_n)} = 0.
\end{align}

We now estimate the left hand side of \eqref{approxthm1}.  By the fundamental theorem of calculus we obtain
\begin{align}
\| F(f) - F(g) \|_{\dot B^{\frac{1}{2}, 2}_{\frac{2(N+1)}{N+3}}} &\lesssim \left (
\int_{\R^N} |y|^{-N - 1} \left \| |\Delta_y(f-g)(x)||g(x)|^{\frac{4}{N-2}} \right \|_{L_x^{\frac{2(N+1)}{N+3}}}^2 dy
\right )^{\frac{1}{2}} \label{nonlinest1} \\
&+ \left (
\int_{\R^N} |y|^{-N - 1} \left \| |\Delta_y g(x)||(f-g)(x)|^{\frac{4}{N-2}} \right \|_{L_x^{\frac{2(N+1)}{N+3}}}^2 dy
\right )^{\frac{1}{2}} \label{nonlinest2}.
\end{align}
We apply this estimate to $\tilde u_n^J - w^J_n$ and $\tilde u_n^J$.  Let $\epsilon > 0$ be small.  Choose $J_0 = J_0(\epsilon)$ so that
$$
\limsup_{n \rightarrow \infty} \left [ \left \| \sum_{j > J_0} U^j_n \right \|_{S(\R)} + \left \| \sum_{j > J_0} U^j_n \right \|_{W(\R)} \right ]
< \epsilon.
$$
H\"older's inequality in space and time, \eqref{nonlinest1}, \eqref{nonlinest2}, and our choice of $J_0$ imply for large $n$, the term
$\left \| F \left ( \tilde u^J_n(t) - w^J_n(t) \right ) -
F( \tilde u^J_n(t) ) \right \|_{W'(I)}$ is
\begin{align}
\lesssim&
\left \|
\left ( \int_{\R^N} |y|^{-N - 1}    \left \| |\Delta_y w_n^J(t,x)|\left |\sum_{j = 1}^{J_0} U^j_n(t,x)\right |^{\frac{4}{N-2}} \right \|_{L_x^{\frac{2(N+1)}{N+3}}}^2
dy \right )^{\frac{1}{2}}
\right \|_{L_t^\frac{2(N+1)}{N+3}(0,\theta_n)} \label{nonlinest3} \\
&+ \| w^J_n \|_{W(\R)} \epsilon^{\frac{4}{N-2}} +
\| \tilde u^{J}_n \|_{W(0,\theta_n)}  \| w^J_n \|^{\frac{4}{N-2}}_{S(\R)}. \nonumber
\end{align}
where the implied constant is independent of $n$ and $J$.  By Strichartz estimates and the boundedness of the sequence 
$\{ (w^J_{0,n}, w^J_{1,n} ) \}_n$ in $\energysp$, 
$$
\| w^J_n \|_{W(\R)} \leq C_0.
$$
Since $\lim_{J \rightarrow \infty} \limsup_n \| w^J_n \|_{S(\R)} = 0$, we have for all $J$ and $n$ sufficiently large, 
$$
\| w^J_n \|_{W(\R)} \epsilon^{\frac{4}{N-2}} +
\| \tilde u^{J}_n \|_{W(0,\theta_n)}  \| w^J_n \|^{\frac{4}{N-2}}_{S(\R)}  \leq (C_0 + 1)  \epsilon^{\frac{4}{N-2}}.
$$ 
Thus to show \eqref{approxthm1}, we are reduced to showing that for each $1 \leq j \leq J_0$,
\begin{align}
\lim_{J \rightarrow \infty} \limsup_{n \rightarrow \infty} &\left \|
\left ( \int_{\R^N} |y|^{-N - 1} \left \| \left |\Delta_y w^J_n(t,x) \right|\left | U^j_n(t,x)\right |^{\frac{4}{N-2}} \right \|_{L_x^{\frac{2(N+1)}{N+3}}}^2
dy \right )^{\frac{1}{2}}
\right \|_{L_t^\frac{2(N+1)}{N+3}(0,\theta_n)} \label{nonlinest4} \\
&= 0, \nonumber
\end{align}
where
$$
\lim_{J \rightarrow \infty} \limsup_{n \rightarrow \infty} \| w^J_n \|_{S(\R)} = 0.
$$

Without loss of generality, we may assume $U^j$ is smooth and compactly supported on $\R \times \R^N$.  By a change of 
variables
\begin{align*}
\eqref{nonlinest4} = \lim_{J \rightarrow \infty} \limsup_{n \rightarrow \infty}\left \|
\left ( \int_{\R^N} |y|^{-N - 1} \left \| \left |\Delta_y \tilde{w}^J_n(t,x) \right|\left | U^j(t,x)\right |^{\frac{4}{N-2}} \right \|_{L_x^{\frac{2(N+1)}{N+3}}}^2
dy \right )^{\frac{1}{2}}
\right \|_{L_t^{\frac{2(N+1)}{N+3}}}, 
\end{align*}
where 
$$
\tilde w_n^J(t,x) = \lambda_{j,n}^{\frac{N-2}{2}} w^J_n(\lambda_{j,n} t + t_{j,n}, \lambda_{j,n} x + x_{j,n}). 
$$
Note that 
$$
\lim_{J \rightarrow \infty} \limsup_{n \rightarrow \infty} \| \tilde{w}^J_n \|_{S(\R)} =
\lim_{J \rightarrow \infty} \limsup_{n \rightarrow \infty} \| w^J_n \|_{S(\R)} = 0
$$
By a slight abuse of notation, we will continue to write $w^J_n$ instead of $\tilde{w}^J_n$ during the rest of the argument.
Let $\varphi \in C^\infty_0(\R \times \R^N)$ be nonnegative, such that $\varphi \equiv 1$ on a neighborhood of the support
of $U^j$. Then
\begin{align*}
\left | w^J_n(t,x+y) - w^J_n(t,x) \right | \left | U^j(t,x) \right |^{\frac{4}{N-2}} &\leq
\left | (\varphi w^J_n)(t,x+y) - (\varphi w^J_n)(t,x) \right | \left | U^j(t,x) \right |^{\frac{4}{N-2}} \\
&+ |U^j(t,x)|^{\frac{4}{N-2}} |w^J_n(x+y)| |\varphi(t,x+y) - \varphi(t,x)|.
\end{align*}
We insert this estimate into \eqref{nonlinest4}, and we see that 
\begin{align*}
\eqref{nonlinest4} \leq& \left \|
\left ( \int_{\R^N} |y|^{-N - 1} \left \| |\Delta_y(\varphi w_n^J)(t,x)|\left | U^j_n(t,x)\right ||^{\frac{4}{N-2}} \right \|_{L_x^{\frac{2(N+1)}{N+3}}}^2
dy \right )^{\frac{1}{2}}
\right \|_{L_t^\frac{2(N+1)}{N+3}(0,\theta_n)}  \\
&+ \left \|
\left ( \int_{\R^N} |y|^{-N - 1} \left \| |\Delta_y \varphi(t,x)||w^J_n(t,x+y)| \left | U^j_n(t,x)\right ||^{\frac{4}{N-2}} \right \|_{L_x^{\frac{2(N+1)}{N+3}}}^2
dy \right )^{\frac{1}{2}}
\right \|_{L_t^\frac{2(N+1)}{N+3}(0,\theta_n)}.
\end{align*}
By H\"older's inequality in space and time and the fact 
$\|\cdot \|_{\dot B^{\frac{1}{2}, 2}_2} \simeq \| \cdot \|_{\dot H^{\frac{1}{2}}}$, the first term appearing on the right hand 
side above is dominated by 
$$
 \| \varphi w^J_n \|_{L^2\left ( \R ; \dot H^{\frac{1}{2}}\right )} \| U^j \|_{L^{\frac{4(N+1)}{N-2}}}^{\frac{4}{N-2}},
$$ 
while the second term is dominated by 
$$
 \|\varphi\|_{L^{2(N+1)/5}\left (\R ; \dot B
^{\frac{1}{2}, 2}_{\frac{2(N+1)}{5}} \right )} \| w^J_n \|_{S(\R)} \| U^j \|_{L^\infty}^{\frac{4}{N-2}}.
$$
Since $\lim_{J \rightarrow \infty} \limsup_n \| w^J_n \|_{S(\R)} = 0$, we have that 
\begin{align*}
\eqref{nonlinest4} &\lesssim \limsup_{J \rightarrow \infty} \limsup_{n \rightarrow \infty}
\left [ \left (\int \| (\varphi w^J_n)(t) \|^2_{\dot H^{\frac{1}{2}}} dt \right )^{1/2} + \| w^J_n \|_{S(\R)} \|\varphi\|_{L^{2(N+1)/5} \left (\dot B
^{\frac{1}{2}, 2}_{\frac{2(N+1)}{5}}\right )}
\right ] \\
&\lesssim \limsup_{J \rightarrow \infty} \limsup_{n \rightarrow \infty} \left ( \int \| 
(\varphi w^J_n)(t) \|^2_{\dot H^{\frac{1}{2}}} dt \right )^{1/2}.
\end{align*}
where the implied constant depends only on $U^j$.  By interpolation, H\"older's inequality in space, and conservation of the free energy we obtain
$$
\| (\varphi w^J_n)(t) \|^2_{\dot H^{\frac{1}{2}}} \leq \| (\varphi w^J_n)(t) \|_{L^2} \| (\varphi w^J_n)(t) \|_{\dot H^1}  \lesssim \| w^J_n(t) \|_{L^{\frac{2(N+1)}{N-2}}}
$$
where the implied constant depends on $\varphi$ and $\sup_{n} \| (u_{0,n},u_{1,n} ) \|_{\energysp}$.  H\"older's inequality
in time yields
$$
\limsup_{J \rightarrow \infty} \limsup_{n \rightarrow \infty} \int_{|t| \leq T} \| (\varphi w^J_n)(t) \|^2_{\dot H^{\frac{1}{2}}} dt
\lesssim \limsup_{J \rightarrow \infty} \limsup_{n \rightarrow \infty} \| w^J_n \|_{S(\R)} = 0.
$$
This shows \eqref{nonlinest4} $= 0$ as desired.  Tracing back through the argument, this yields \eqref{approxthm1}.

We now show \eqref{approxthm2}.  Estimating similarly to \eqref{nonlinest1} and \eqref{nonlinest2} we obtain for fixed $J \geq 1$,
\begin{align*}
\left \| F \left ( \sum_{j = 1}^J U^j_n(t) \right ) -
\sum_{j = 1}^J F(U^j_n(t)) \right \|_{\dot B^{\frac{1}{2}, 2}_{\frac{2(N+1)}{N+3}}} \leq C(J) \sum_{i \neq j}
R^{i,j}_n(t)
\end{align*}
where 
$$
R^{i,j}_n(t) = \left ( \int_{\R^N} |y|^{-N-1} \| |\Delta_y U_n^i(x)| |U_n^j(x)|^{\frac{4}{N-2}} 
\|_{L_x^{\frac{2(N+1)}{N+3}}}^2 dy 
 \right )^{\frac{1}{2}}
$$
If $i \neq j$, then orthogonality of the parameters implies that
$$
\lim_{n \rightarrow \infty} \int_0^{\theta_n} \left |R^{i,j}_n(t) \right|^{\frac{2(N+1)}{N+3}} dt = 0.
$$
Hence, for each fixed $J \geq 1$,
$$
\lim_{n \rightarrow \infty} \left \| F \left ( \sum_{j = 1}^J U^j_n(t) \right ) -
\sum_{j = 1}^J F(U^j_n(t)) \right \|_{W'(0,\theta_n)} = 0,
$$
which shows \eqref{approxthm2}.

Thus, we have shown that
$$
\lim_{J \rightarrow \infty} \limsup_{n \rightarrow \infty} \| \tilde e^J_n \|_{W'(0,\theta_n)} = 0.
$$
This along with
$$
(\tilde u^J_n(0), \partial_t \tilde u^J_n(0)) = (u_n(0), \partial_t u_n(0)) + o(1) \quad \mbox{in } \energysp
\mbox{ as } n \rightarrow \infty,
$$
yield the conclusion of Proposition \ref{approxthm} by Theorem \ref{longtime}.
\end{proof}

\section{}

In this section we prove Lemma \ref{locorth} for odd $N$.  Lemma \ref{locorth} was proved for even $N$
in \cite{cks} using Fourier methods.  The proof for odd $N$ follows in a very similar fashion.

We first reduce Lemma \ref{locorth} to the case $\theta_n = 0$.  Indeed, let $s_{j,n} = \frac{\theta_n - t_{j,n}}{\lambda_{j,n}}$.
Extracting subsequences, we can assume that $s_{j,n}$ has a limit $s_j$ in $\bar \R$ as $n \rightarrow \infty$.  Observe
that there exists a unique solution $V^j_L$ of the linear wave equation such that
$$
\lim_{n \rightarrow \infty} \left \| \vec V^j_L (s_{j,n}) - \vec U^j_n(s_{j,n}) \right \|_{\energysp} = 0.
$$
Indeed, if $s_j \in I_{\max}(U^j)$, then we can take $V^j_L(t) = S(t-s_j)\vec U^j(s_j)$.  If $s_j = +\infty$ (respectively
$s_j = -\infty$), then
our assumption that $\limsup_{n \rightarrow \infty} \| U^j_n \|_{S(0,\theta_n)} < \infty$ implies $U^j$ scatters forward in time
(respectively backwards in time)
and the existence of $V^j_L$ follows.

Letting $\tau_{j,n} = -\theta_n + t_{j,n}$, it is easy to see that the parameters $\{\lambda_{j,n}, \tau_{j,n} \}$ are orthogonal.
By Proposition \ref{approxthm}, we have $\vec u(\theta_n)$ admits the following profile decomposition
$$
\vec u(\theta_n) = \sum_{j = 1}^J \vec V^J_{L,n}(0) +  (\eta_{0.n}^J, \eta_{1,n}^J),
$$
where
$$
(\eta_{0,n}^J, \eta_{1,n}^J) = (w^J_n(\theta_n),\partial_t w^J_n(\theta_n) + (r^J_n(\theta_n), \partial_t r_n^J(\theta_n))
 + o_n(1) \quad
\mbox{in } \energysp.
$$
After these considerations, we see that it suffices to prove Lemma \eqref{locorth} with $U^j_n$ replaced with $V^j_{L,n}$ and
with $w^J_n + r_n^j$ replaced by $\eta^J_n$.  So without loss of generality, we may suppose that we are
in the case $\theta_n = 0$ for all $n$.
By the construction of a profile decomposition,
$$
\forall J \geq 1, \forall 1 \leq j \leq J, \quad
\vec S(t_{j,n}/\lambda_{j,n} ) \left (\lambda_{j,n}^{(N-2)/2} w^J_{0,n}(\lambda_{j,n} \cdot ), \lambda_{j,n}^{N/2} w^J_{1,n}(\lambda_{j,n} \cdot) \right )
\rightharpoonup 0.
$$
This along with orthogonality of the parameters easily imply that the case $\theta_n = 0$ reduces to the following lemma.
\begin{lem}\label{appblem1}
Let $N \geq 3$ be odd.  Let $(w_{0,n},w_{1,n})$ be a bounded sequence of radial functions in $\energysp$.  Let $\{t_n\}_n$,
$\{r_n\}$ be two sequences with $r_n \geq 0$.  Assume that $\vec S(-t_n)(w_{0,n},w_{1,n}) \rightharpoonup 0$ in $\energysp$ as
$n \rightarrow \infty$.  Then for any $(u_0,u_1) \in \energysp$, one has (after passing to a subsequence if necessary)
\begin{align*}
\lim_{n \rightarrow \infty} \int_{|x| \geq r_n} \nabla_{t,x} S(t_n)(u_0,u_1) \cdot (\nabla w_{0,n}, w_{1,n}) dx &= 0, \\
\lim_{n \rightarrow \infty} \int_{|x| \leq r_n} \nabla_{t,x} S(t_n)(u_0,u_1) \cdot (\nabla w_{0,n}, w_{1,n}) dx &= 0.
\end{align*}
\end{lem}

Before proceeding further, we establish conventions and gather facts that will be needed for the proof.  The solution $u(t)$ to
the linear equation \eqref{lw} with initial data $(u_0,u_1)$ is
given by
$$
u(t) = \cos(t |\nabla|) u_0 + \frac{\sin(t|\nabla|)}{|\nabla|} u_1.
$$
The Fourier transform of $f$ is defined to be
\begin{align}\label{appb1}
\hat f(\xi) = \int_{\R^N} e^{-ix \cdot \xi} f(x) dx, \quad f(x) = (2 \pi)^{-N} \int_{\R^N} e^{i x \cdot \xi} \hat f (\xi) d\xi,
\end{align}
and Parseval's identity is
\begin{align}\label{appb2}
\int_{\R^N} f(x) \overline{g(x)} dx = (2\pi)^{-N} \int_{\R^N} \hat f(\xi) \overline{\hat g(\xi)} d\xi.
\end{align}

If $f$ is radial, then $\hat f$ is also radial.  Recall that
$$
\widehat{\sigma_{S^{N-1}}}(\xi) = (2 \pi)^{\frac{N}{2}} |\xi|^{-\nu} J_{\nu}(|\xi|), \quad \nu = \frac{N-2}{2},
$$
where $J_{\nu}$ is the Bessel function of the first kind of order $\nu$.  It is characterized as being the solution
to Bessel's equation
$$
x^2 J''_{\nu}(x) + x J'_{\nu}(x) + (x^2 - \nu^2 ) J_{\nu}(x) = 0
$$
which is regular at $x = 0$ (unique up to a multiplicative constant).  The inversion formula takes the form
\begin{align*}
f(r) = (2 \pi)^{-\frac{N}{2}} \int_0^{\infty} \hat f(\rho) J_{\nu}(r \rho) (r \rho)^{-\nu} \rho^{N-1} d\rho.
\end{align*}
For the solution $u(t,r)$, this means that
\begin{align*}
u(t,r) &= (2 \pi)^{-\frac{N}{2}} \int_0^\infty \left(  \cos(t \rho) \hat f(\rho) + \frac{\sin(t\rho)}{\rho} \right )
J_{\nu}(r \rho) (r \rho)^{-\nu} \rho^{N-1} d\rho, \\
\partial_t u(t,r) &= (2 \pi)^{-\frac{N}{2}} \int_0^\infty \left(  -\sin(t \rho) \rho \hat f(\rho) + \cos(t\rho) \right )
J_{\nu}(r \rho) (r \rho)^{-\nu} \rho^{N-1} d\rho.
\end{align*}
The standard asymptotics for the Bessel functions are given by
\begin{align}\label{appb3}
J_{\nu}(x) &= \sqrt{\frac{2}{\pi x}}\left [ (1 + \omega_2(x)) \cos(x - \tau) + \omega_1(x) \sin (x - \tau) \right ], \\
J_{\nu}'(x) &= \sqrt{\frac{2}{\pi x}}\left [ \tilde \omega_1(x) \cos(x - \tau) - (1 - \tilde \omega_2(x)) \sin (x - \tau) \right ],
\end{align}
with phase-shift $\tau = (N-1) \frac{\pi}{4}$, and with the bounds (for $n \geq 0$, $x \geq 1$),
\begin{align}\label{appb4}
|w_1^{(n)}(x)| + | \tilde \omega_1^{(n)}(x)| &\leq C_n x^{-1-n}, \\
|w_2^{(n)}(x)| + | \tilde \omega_2^{(n)}(x)| &\leq C_n x^{-2-n}.
\end{align}

Finally, we have the following limit (in the sense of distributions)
\begin{align}\label{appb9}
\lim_{\epsilon \rightarrow 0^+} \int_{r_0}^\infty \cos(ar) e^{-\epsilon r} dr = \lim_{\epsilon \rightarrow 0^+}
\frac{1}{2} \left ( \frac{e^{-r_0(ia + \epsilon)}}{ia + \epsilon} - \frac{e^{r_0(ia - \epsilon)}}{ia - \epsilon} \right )
= \pi \delta_0(a) - \frac{\sin(r_0 a)}{a}.
\end{align}
We now turn to the proof of Lemma \ref{appblem1}.

\begin{proof}[Proof Lemma \ref{appblem1}]
The proof is very similar to the proof of Lemma 6 from \cite{cks}, and we focus on the main difference.  By conservation of the free energy, we have
\begin{align}
\int_{\R^N}& \nabla_{t,x} S(t_n)(u_0,u_1) \cdot (\nabla w_{0,n}, w_{1,n}) dx \nonumber \\
&= \int_{\R^N}(\nabla u_0,u_1) \cdot \nabla_{t,x} S(-t_n)( w_{0,n}, w_{1,n}) dx \rightarrow 0 \quad \mbox{as }
n \rightarrow \infty. \label{appb6}
\end{align}
Hence, to prove the lemma, we only need to show
\begin{align}\label{appb5}
\lim_{n \rightarrow \infty} \int_{|x| \geq r_n} \nabla_{t,x} S(t_n)(u_0,u_1) \cdot (\nabla w_{0,n}, w_{1,n}) dx = 0.
\end{align}
Since the linear propagator $S(t)$ is unitary, we may suppose that $(u_0,u_1),(w_{0,n},w_{1,n}) \in \cl S(\R^N) \times
\cl S(\R^N)$ with Fourier support away from the origin.
We may also assume (after passing to a subsequence) that the sequences
$$
\{ t_n \}_n, \quad \{ r_n \}_n, \quad \{ t_n - r_n \}_n, \mbox{  and  } \{ t_n + r_n \}_n,
$$
have limits in $\bar \R$.

If $\{ t_n \}_n$ has a limit in $\R$, then $\grad_{t,x} S(t_n)(u_0,u_1)$ converges strongly in
$(L^2)^{N+1}$ and $(\nabla w_{n,0}, w_{1,n})$ converges weakly to 0 in $(L^2)^{N+1}$.  The dominated convergence theorem
shows that
$$
\textbf 1_{|x| \geq r_n} (\nabla w_{n,0}, w_{1,n}) \rightharpoonup 0,
$$
in $(L^2)^{N+1}$ as $n \rightarrow \infty$.  This shows \eqref{appb5} in the case $\lim t_n \in \R$.

Suppose that $\lim t_n \in \{\pm \infty\}$ and $\lim r_n \in [0,\infty)$. Let $\epsilon > 0$.  By Lemma \ref{loclem},
there exists $R = R(\epsilon)$ so that
$$
\limsup_{n \rightarrow \infty} \int_{||x| - |t_n|| \geq R} |\nabla_{t,x} S(t)(u_0,u_1)|^2 dx < \epsilon.
$$
Since $\lim |t_n| = \infty$ and $\lim r_n \in [0,\infty)$, for $n$ sufficiently large, $\{ |x| \leq r_n \} \subset
\{ ||x| - |t_n|| \geq R \}$. By \eqref{appb6},
\begin{align*}
\limsup_{n \rightarrow \infty} & \left | \int_{|x| \geq r_n} \nabla_{t,x} S(t_n)(u_0,u_1) \cdot (\nabla w_{0,n}, w_{1,n}) dx \right | \\
&\leq \limsup_{n \rightarrow \infty} \left | \int_{|x| \leq r_n} \nabla_{t,x} S(t_n)(u_0,u_1) \cdot (\nabla w_{0,n}, w_{1,n}) dx \right |
\leq \sup_n \| (w_{0,n}, w_{1,n})\|_{\energysp} \sqrt \epsilon.
\end{align*}
This shows \eqref{appb5} in the case $\lim r_n \in [0,\infty)$.

It remains to treat the case where both $\{t_n\}_n$ and $\{r_n\}_n$ have infinite limits.  We use the Fourier representation
of $S(t)(u_0,u_1)$ using Bessel functions.  Using the asymptotics \eqref{appb3} and \eqref{appb4} and the
boundedness of the Hilbert and Hankel transforms on the half line (see the end
of the proof of Lemma 6 in \cite{cks}), it can be shown that in this case \eqref{appb5} is equal to
\begin{align*}
\lim_{\epsilon \rightarrow 0^+} \int_{r_n}^\infty \int_0^\infty&
\left ( \cos(t_n) \rho) \rho \widehat{u_0}(\rho) + \sin(t_n \rho) \widehat{u_1}(\rho) \right ) \sin(r \rho - \tau)(r \rho)^{-\nu - \frac{1}{2}}
\rho^{N-1} d\rho \\
\int_0^\infty& \widehat{w_{0,n}}(\sigma) \sigma \sin(r \sigma - \tau) (r\sigma)^{-\nu - \frac{1}{2}} \sigma^{N-1} d\sigma
e^{-\epsilon r} r^{N-1} dr \\
+ \int_{r_n}^\infty \int_0^\infty&
\left ( -\sin(t_n) \rho) \rho \widehat{u_0}(\rho) + \cos(t_n \rho) \widehat{u_1}(\rho) \right ) \cos(r \rho - \tau)(r \rho)^{-\nu - \frac{1}{2}}
\rho^{N-1} d\rho \\
\int_0^\infty& \widehat{w_{0,n}}(\sigma) \cos(r \sigma - \tau) (r\sigma)^{-\nu - \frac{1}{2}} \sigma^{N-1} d\sigma
e^{-\epsilon r} r^{N-1} dr
\end{align*}
up to an $o(1)$ term as $n \rightarrow \infty$.  We now show that the above limit is $o(1)$ as $n \rightarrow \infty$.

In order to carry out the $r$--integration, we use (note $2\tau \in \Z \pi$ when $N$ is odd)
\begin{align*}
\cos (r \rho - \tau) \cos (r \sigma - \tau) =& \frac{1}{2} [(-1)^{\frac{N-1}{2}} \cos(r(\rho + \sigma))
+ \cos(r(\rho - \sigma))], \\
\sin (r \rho - \tau) \sin (r \sigma - \tau) =& \frac{1}{2} [-(-1)^{\frac{N-1}{2}} \cos(r(\rho + \sigma))
+ \cos(r(\rho - \sigma))].
\end{align*}

Carrying out the $r$--integration and passing to the limit, \eqref{appb9} yields the expression (up to a
constant)
\begin{align*}
\int_0^\infty \int_0^\infty &\left ( \cos(t_n) \rho) \rho \widehat{u_0}(\rho) + \sin(t_n \rho) \widehat{u_1}(\rho) \right )
\rho^{\frac{N-1}{2}} \sigma \widehat{w_{0,n}}(\sigma) \sigma^{\frac{N-1}{2}} \\
& \left ( \pi \delta_0 (\rho - \sigma ) - \frac{\sin(r_n(\rho - \sigma))}{\rho - \sigma} + (-1)^{\frac{N-1}{2}}
\frac{\sin(r_n(\rho + \sigma))}{\rho + \sigma}\right ) d\rho d\sigma \\
+ \int_0^\infty & \int_0^\infty \left ( -\sin(t_n) \rho) \rho \widehat{u_0}(\rho) + \cos(t_n \rho) \widehat{u_1}(\rho) \right )
\rho^{\frac{N-1}{2}} \widehat{w_{1,n}}(\sigma) \sigma^{\frac{N-1}{2}} \\
& \left ( \pi \delta_0 (\rho - \sigma ) - \frac{\sin(r_n(\rho - \sigma))}{\rho - \sigma} - (-1)^{\frac{N-1}{2}}
\frac{\sin(r_n(\rho + \sigma))}{\rho + \sigma}\right ) d\rho d\sigma.
\end{align*}
The $\delta_0$ make the following contribution
\begin{align*}
\int_0^\infty & \rho \widehat{u_0}(\rho)\left ( \cos(t_n \rho) \rho \widehat{w_{0,n}}(\rho) -
\sin(t_n \rho) \widehat{w_{1,n}}(\rho) \right ) \rho^{N-1} d\rho \\
&+ \int_0^\infty  \widehat{u_1}(\rho)\left ( \sin(t_n \rho) \rho \widehat(w_{0,n})(\rho) +
\cos(t_n \rho) \widehat{w_{1,n}}(\rho) \right ) \rho^{N-1} d\rho
\end{align*}
which tends to $0$ since $S(-t_n)(w_{0,n},w_{1,n}) \rightharpoonup 0$ in $\energysp$.

Next, we extract terms involving the Hilbert transform kernel $\frac{1}{\rho - \sigma}$.  These terms contribute (up to a constant)
\begin{align*}
\int_0^\infty& \int_0^\infty \rho \widehat{u_0}(\rho)\left ( \cos(t_n \rho) \rho \widehat{w_{0,n}}(\sigma) -
\sin(t_n \rho) \widehat{w_{1,n}}(\sigma) \right ) \frac{\sin(r_n(\rho - \sigma))}{\rho - \sigma} (\rho \sigma)^{\frac{N-1}{2}}
d\rho d\sigma \\
&+ \int_0^\infty \int_0^\infty \widehat{u_1}(\rho)\left ( \sin(t_n \rho) \rho \widehat{w_{0,n}}(\sigma) +
\cos(t_n \rho) \widehat{w_{1,n}}(\sigma) \right ) \frac{\sin(r_n(\rho - \sigma))}{\rho - \sigma} (\rho \sigma)^{\frac{N-1}{2}}
d\rho d\sigma.
\end{align*}
Using simple trigonometric identities, the terms involving $u_0$ can be transformed into
\begin{align*}
\int_0^\infty \int_0^\infty& \rho \widehat{u_0}(\rho) \left [ \sin((t_n + r_n)(\rho - \sigma)) + \sin((t_n - r_n)
(\rho - \sigma)) \right ] \\
&\left [ \cos(t_n \sigma) \sigma \widehat{w_{0,n}}(\sigma) - \sin(t_n \sigma) \widehat{w_{1,n}}(\sigma) \right ] \\
&+ \rho \widehat{u_0}(\rho) \left [ \cos((t_n + r_n)(\rho - \sigma)) - \cos((r_n - t_n)
(\rho - \sigma)) \right ] \\
& \left [ \sin(t_n \sigma) \sigma \widehat{w_{0,n}}(\sigma) + \cos(t_n \sigma) \widehat{w_{1,n}}(\sigma) \right ]
\frac{(\rho \sigma)^{\frac{N-1}{2}}}{\rho - \sigma} d\rho d\sigma.
\end{align*}
Let $\cl F_1$ be the Fourier transform on $\R$.  Define
$$
\tilde u_0 = \cl F_1\left ( \textbf 1_{\R^+} \rho \widehat{u_0}(\rho) \rho^{\frac{N-1}{2}} \right ) \in L^2(\R).
$$
Then with some constant $c$,
$$
\int_0^\infty e^{\pm i B_n (\rho - \sigma)} \frac{\rho \widehat{u_0}(\rho) \rho^{\frac{N-1}{2}}}{\rho - \sigma} d\rho
= c \cl F_1 \left ( \mbox{sign}(\cdot \pm B_n) \tilde u_0 \right ) (\sigma).
$$
If $B_n$ has a limit in $\bar \R$, then this converges strongly in $L^2(\R)$.  In our case $B_n = t_n \pm r_n$.  Thus,
the terms involving $u_0$ can be reduced to the form $\la f_n, \tilde f_n \ra \rightarrow 0$ where $\tilde f_n$ converges
strongly in $L^2(\R)$ and $f_n \rightharpoonup 0$ in $L^2(\R)$.  Analogously, the terms involving $u_1$ can be transformed
into
\begin{align*}
\int_0^\infty \int_0^\infty& \widehat{u_1}(\rho) \left [ \sin((t_n + r_n)(\rho - \sigma)) - \sin((t_n - r_n)
(\rho - \sigma)) \right ] \\
&\left [ \sin(t_n \sigma) \sigma \widehat{w_{0,n}}(\sigma) + \cos(t_n \sigma) \widehat{w_{1,n}}(\sigma) \right ] \\
&- \widehat{u_1}(\rho) \left [ \cos((t_n + r_n)(\rho - \sigma)) - \cos((t_n - r_n)
(\rho - \sigma)) \right ] \\
& \left [ \cos(t_n \sigma) \sigma \widehat{w_{0,n}}(\sigma) - \sin(t_n \sigma) \widehat{w_{1,n}}(\sigma) \right ]
\frac{(\rho \sigma)^{\frac{N-1}{2}}}{\rho - \sigma} d\rho d\sigma.
\end{align*}
which converges to 0 by the same reasoning.

Finally, we extract the terms involving the Hankel transform kernel $\frac{1}{\rho + \sigma}$.  These terms contribute (up to a
constant)
\begin{align*}
\int_0^\infty& \int_0^\infty \rho \widehat{u_0}(\rho)\left ( \cos(t_n \rho) \rho \widehat{w_{0,n}}(\sigma) +
\sin(t_n \rho) \widehat{w_{1,n}}(\sigma) \right ) \frac{\sin(r_n(\rho + \sigma))}{\rho + \sigma} (\rho \sigma)^{\frac{N-1}{2}}
d\rho d\sigma \\
&+ \int_0^\infty \int_0^\infty \widehat{u_1}(\rho)\left ( \sin(t_n \rho) \rho \widehat{w_{0,n}}(\sigma) -
\cos(t_n \rho) \widehat{w_{1,n}}(\sigma) \right ) \frac{\sin(r_n(\rho + \sigma))}{\rho + \sigma} (\rho \sigma)^{\frac{N-1}{2}}
d\rho d\sigma.
\end{align*}
Using the same type of trigonometric identities as before, the terms involving $u_0$ can be transformed into
\begin{align*}
\int_0^\infty \int_0^\infty& \rho \widehat{u_0}(\rho) \left [ \cos((t_n - r_n)(\rho + \sigma)) - \cos((t_n + r_n)
(\rho + \sigma)) \right ] \\
&\left [ \sin(t_n \sigma) \sigma \widehat{w_{0,n}}(\sigma) + \cos(t_n \sigma) \widehat{w_{1,n}}(\sigma) \right ] \\
&+ \rho \widehat{u_0}(\rho) \left [ \sin((t_n + r_n)(\rho + \sigma)) - \sin((t_n - r_n)
(\rho + \sigma)) \right ] \\
& \left [ \cos(t_n \sigma) \sigma \widehat{w_{0,n}}(\sigma) - \sin(t_n \sigma) \widehat{w_{1,n}}(\sigma) \right ]
\frac{(\rho \sigma)^{\frac{N-1}{2}}}{\rho + \sigma} d\rho d\sigma.
\end{align*}
As in the case of the Hilbert transform, define
$$
\check u_0 = \cl F_1\left ( \textbf 1_{\R^-} \rho \widehat{u_0}(\rho) \rho^{\frac{N-1}{2}} \right ) \in L^2(\R),
$$
and notice
$$
\int_0^\infty e^{\pm i B_n (\rho + \sigma)} \frac{\rho \widehat{u_0}(\rho) \rho^{\frac{N-1}{2}}}{\rho + \sigma} d\rho
= c \cl F_1 \left ( \mbox{sign}(\cdot \mp B_n) \check u_0 \right ) (\sigma).
$$
By the same reasoning as in the case of the Hilbert transform, the terms involving $u_0$ converge to 0.  Similarly, the terms
involving $u_1$ can be transformed into
\begin{align*}
\int_0^\infty \int_0^\infty& \widehat{u_1}(\rho) \left [ \cos((t_n - r_n)(\rho + \sigma)) - \cos((t_n + r_n)
(\rho + \sigma)) \right ] \\
&\left [ \cos(t_n \sigma) \sigma \widehat{w_{0,n}}(\sigma) - \sin(t_n \sigma) \widehat{w_{1,n}}(\sigma) \right ] \\
&- \widehat{u_1}(\rho) \left [ \sin((t_n + r_n)(\rho + \sigma)) - \sin((t_n - r_n)
(\rho + \sigma)) \right ]  \\
& \left [ \sin(t_n \sigma) \sigma \widehat{w_{0,n}}(\sigma) + \cos(t_n \sigma) \widehat{w_{1,n}}(\sigma) \right ]
\frac{(\rho \sigma)^{\frac{N-1}{2}}}{\rho - \sigma} d\rho d\sigma.
\end{align*}
These terms converge to 0 by the same reasoning as before. This completes the proof of Lemma \ref{appblem1}.
\end{proof}

\section{}

In this appendix, we give a proof of Theorem \ref{cptsolns} for $N \geq 6$.  Theorem \ref{cptsolns} was proved for $3 \leq N \leq 5$ in Section 6
of \cite{dkm1}.  The proof for $N \geq 6$ will follow in a very similar fashion and we will highlight the main differences.  We split the proof up into several steps.
Without loss of generality, we may assume that $\lambda$ is continuous on $I_{\max}(u)$ (see \cite{km06}, Remark 5.4).

\begin{proof}[Proof of Theorem \ref{cptsolns}]  We divide the proof into four steps.

\subsubsection*{Step 1: u is globally defined.}  In this step we show that $I_{\max}(u) = \R$.  Suppose not, i.e., $T^+(u) < \infty.$
By rescaling, we may assume that $T^+ = 1$.  We first remark that by the compactness of $K$, we have $E(u_0,u_1) > 0$ and $T^-(u) = -\infty$ (cf. Section 2 of
\cite{dkm6}).  By Section 4 of \cite{km08}, $\lambda(t) \leq C(1-t)$, and $\supp (u(t),\partial_t u(t))
\subseteq \{ |x| \leq 1-t \}$.  By Section 6 of \cite{km08}, self--similar, compact blow--up is excluded.
This implies
that there exists a sequence $\{\tau_n\}_n$ such that
$$
\tau_n \in (0,1), \quad \lim_{n \rightarrow \infty} \tau_n = 1, \quad \lim_{n \rightarrow \infty}
\frac{\lambda(\tau_n)}{1 - \tau_n} = 0.
$$
We find that there exists a sequence $\{t_n\}_n$ such that
\begin{align}\label{appc1}
\forall n, \forall \sigma \in (0, 1-t_n), \quad \frac{1}{\sigma} \int_{t_n}^{t_n + \sigma} \int_{\R^N} |\partial_t u (t,x)|^2 dx \hspace{1pt} dt
\leq \frac{1}{n}.
\end{align}
(see Section 5 of \cite{dkm1}).  Let $(U^0,U^1) \in \energysp$ be such that for a subsequence
$$
\lim_{n \rightarrow \infty} (\lambda(t_n)^{(N-2)/2}u(t_n, \lambda(t_n)\cdot), \lambda(t_n)^{N/2}\partial_t u(t_n, \lambda(t_n)\cdot))
= (U^0,U^1), \quad \mbox{in } \energysp.
$$
Let $U$ be a solution of \eqref{nlw} with initial condition $(U^0,U^1)$ and $\tau_0 \in (0,T^+(U))$.  Then by Proposition
\ref{longtime},
$$
\lim_{n \rightarrow \infty} \int_0^{\tau_0} \int_{\R^N} \lambda(t_n)^N ( \partial_t u(t_n + \lambda(t_n)s, \lambda(t_n)x))^2
dx ds = \int_0^{\tau_0} \int_{\R^N} (\partial_t U(t))^2 dx dt.
$$
By \eqref{appc1} and a change of variables, we obtain
$$
\lim_{n \rightarrow \infty} \lim_{n \rightarrow \infty} \int_0^{\tau_0} \int_{\R^N} \lambda(t_n)^N ( \partial_t u(t_n + \lambda(t_n)s, \lambda(t_n)x))^2
dx ds = 0.
$$
This implies that $\partial_t U \equiv 0$ for $t \in [0,\tau_0]$ so that $-\Delta U^0 = |U^0|^{4/(N-2)}U^0$.  Hence, $U = 0$ or
$U = W$ up to the invariances of the equation.  If $U = 0$, then $\| (u(t_n), \partial_t u(t_n) )\|_{\energysp} \rightarrow 0$
as $n \rightarrow \infty$.  By the small data theory, $u$ is global and scatters, contradicting $T^+(u) < \infty$.  Thus, $U = W$
up to the invariances of the equation.  By conservation of energy, $E(u_0,u_1) = E(W,0)$.

The solution $u$ of \eqref{nlw} has threshold energy $E(W,0)$, is not globally defined, and satisfies $u_0 \in L^2$ (since
supp $u_0 \subset B_1$).  By the main result of \cite{lz}, $u$ has to be the special solution $W^+$ constructed in this paper which
satisfies $\| u(t) - W \|_{\dot H^1} \leq e^{ct}$ as $t \rightarrow -\infty$.  Since $\supp (u_0,u_1) \subseteq \{ |x| \leq 1\}$, by
finite speed of propagation $\supp (u(t),\partial_t u(t)) \subseteq \{ |x| \leq 1 + |t| \}$.  Hence, as $t \rightarrow -\infty$
\begin{align*}
|t|^{2-N} \simeq \int_{|x| \geq 1 + |t|} |\nabla W|^2 dx \leq \int_{\R^N} |\nabla W - \nabla u(t)|^2 dx
\leq e^{-2c |t|},
\end{align*}
which is impossible.  We conclude that $u$ must be globally defined.

\subsubsection*{Step 2: $E(u_0,u_1) = E(W,0)$.}
The proof that $E(u_0,u_1) = E(W,0)$ is exactly the same as Steps 2--3 of Section 6 of \cite{dkm1} and we omit it.  The key
again is that the only radial solution to the elliptic equation $-\Delta U = |U|^{4/(N-2)} U$ is $W$ up to the invariances of the
equation.

\subsubsection*{Step 3: Convergence in mean to $W$.}  By the main result of \cite{lz}, $\| \nabla u(t) \|_{L^2}^2 \geq
\| \nabla W \|^2_{L^2}$ for all $t$.  If this were not the case, $u$ would scatter in at least one time direction, contradicting
the compactness of $K$.

Define
$$
\textup d(t) = \frac{2(N-1)}{N-2} \int_{\R^N} |\partial_t u(t)|^2 dx + \frac{2}{N-2} \left ( \int_{\R^N} |\nabla u(t)|^2 dx
- \int_{\R^N} |\nabla W|^2 dx \right ) \geq 0.
$$
By the variational characterization of $W$ (see \cite{aub} and \cite{tal}), for any $t_0$, $\textup d(t_0) = 0$ if and only if
$(u_(t_0), \partial_t u(t_0)) = (W,0)$.  In this step we will show by a virial argument that
\begin{align}\label{appc2}
\lim_{T \rightarrow \infty} \frac{1}{T} \int_{-T}^T \textup d(t) dt = 0.
\end{align}
This step and the following step will follow very closely Steps 4--5 of \cite{dkm1} which is based on Section 3 of \cite{dm08}.

Choose a function $\varphi \in C^{\infty}_0$ such that $\varphi = 1$ for $|x| \leq 1$, and write $\varphi_R(x) = \varphi(x/R)$.  Let
$$
g_R(t) = \int_{\R^N} \varphi_R u(t) \partial_t u(t) dx,
$$
and note that $|g_R(t)| \leq C_0 R$ where $C_0$ depends only on $\sup_t \| \vec u(t) \|_{\energysp}$.  Using that $u$ solves
\eqref{nlw}, we get by integration by parts
\begin{align}\label{appc3}
g_R'(t) = \textup d(t) + A_R(t),
\end{align}
where
\begin{align}\label{appc4}
|A_R(t)| \leq \int_{|x| \geq R} |\nabla u(t)|^2 + |\partial_t u(t)|^2 + |u(t)|^{\frac{2N}{N-2}} + \frac{|u(t)|^2}{|x|^2} dx.
\end{align}
Let $\epsilon > 0$.  As in the Step 4 of \cite{dkm1}, we deduce that that there exists a constant $C_1$ independent of $\epsilon$,
and a time $t_1(\epsilon)$ such that
$$
\forall T > 2 t_1(\epsilon), \forall t \in [t_1(\epsilon), T], \quad g'_{\epsilon T} (t) \geq \textup d(t) - C_1 \epsilon.
$$
This follows from the compactness of $K$ and the fact that $\lim_{|t| \rightarrow +\infty} \lambda(t) / |t| = 0$.  Integrating from
$t_1(\epsilon)$ to $T$, we obtain
\begin{align*}
\int_{t_1(\epsilon)}^T \textup{d}(t) dt \leq (C_0 \epsilon + C_1 \epsilon) T,
\end{align*}
and thus,
\begin{align*}
\limsup_{T \rightarrow +\infty} \frac{1}{T} \int_0^T \textup d(t) dt \leq
\limsup_{T \rightarrow +\infty} \frac{1}{T} \int_{t_1(\epsilon)}^T \textup{d}(t) dt \leq (C_1 + C_2) \epsilon.
\end{align*}
The same proof works for negative time, yielding \eqref{appc2}.

\subsubsection*{Step 4: Conclusion of the proof.}  In this step we show that $\textup d(t) \equiv 0$.  As in the
proof of Lemma 3.8 in \cite{dm08}, by refining the bound on
$g_R(t)$ and the estimate on $A_R(t)$ (using $g_R(t)$ and $A_R(t)$ vanish if $u$ is $W$), and by modulating the solution around $W$ for small $\textup d(t)$, we conclude from \eqref{appc3}
that there is a constant $C = C(K) > 0$ such that
\begin{align}\label{appc5}
\forall \sigma, \tau \in \R, \quad \sigma < \tau \implies \int_{\sigma}^{\tau}
\textup d(t) dt \leq C \left ( \sup_{\sigma \leq t \leq \tau} \lambda(t) \right ) (\textup d(\sigma) + \textup d(\tau)).
\end{align}
Using compactness and modulation arguments similar to those used to prove Lemma 3.10 in \cite{dm08}, we get the following
control on $\lambda(t)$:
\begin{align}\label{appc6}
\sigma + \lambda(\sigma) \leq \tau \implies |\lambda(\sigma) - \lambda(\tau)| \leq \int_{\sigma}^{\tau} \textup d(t) dt.
\end{align}
Let $\{\sigma_n \}_n$ and $\{ \tau_n \}_n$ be two sequences such that
\begin{align*}
\lim_{n \rightarrow \infty} \sigma_n = -\infty, \quad \lim_{n \rightarrow \infty} \tau_n = +\infty, \\
\lim_{n \rightarrow \infty} \textup d(\sigma_n) = \lim_{n \rightarrow \infty} \textup d(\tau_n) = 0.
\end{align*}
These sequences exist by Step 3.  We prove that $\lambda$ is bounded.  Let $n_0 \geq 1$ be such that $\textup d(\tau_{n}) \leq 1/4C$ for
all $n \geq n_0$.  Here $C$ is the constant appearing in \eqref{appc5}.
For large $n$, let $t_n \in [\tau_{n_0}, \tau_n]$ be such that
$$
\lambda(t_n) = \sup_{\tau_{n_0} \leq t \leq \tau_n} \lambda(t).
$$
If (after passing to a subsequence) $\lim_{n \rightarrow \infty} \lambda(t_n) = +\infty$, then by the continuity of $\lambda$, $\lim_{n \rightarrow \infty} t_n =
+ \infty$.  In particular, for large $n$, $\tau_{n_0} + \lambda(\tau_{n_0}) \leq t_n$.  By \eqref{appc6} and \eqref{appc5},
this implies that
\begin{align*}
\lambda(t_n) &\leq \lambda(\tau_{n_0}) + \int_{\tau_{n_0}}^{t_n} \textup d(t) dt \\
&\leq \lambda(\tau_{n_0}) + C \lambda(t_n) ( \textup d(\tau_{n_0}) + \textup d(\tau_{n})) \\
&\leq \frac{1}{4C} + \frac{1}{2} \lambda(t_n),
\end{align*}
a contradiction if $\lim_{n \rightarrow \infty} \lambda(t_n) = +\infty$.  Thus, $\lambda$ is bounded on $[0,\infty)$.  A similar proof
yields that $\lambda$ is bounded on $(-\infty, 0]$.  The boundedness of $\lambda$ and \eqref{appc5} imply that
$$
\int_{\R} \textup d(t) dt = \lim_{n \rightarrow \infty} \int_{\sigma_n}^{\tau_n} \textup d(t) dt \leq C
\lim_{n \rightarrow \infty} ( \textup d(\sigma_n) + \textup d(\tau_n)) = 0,
$$
which implies that $\textup d(t) = 0$ for all $t$.  The concludes the proof of Theorem \ref{cptsolns}.
\end{proof}

\end{document}